\documentclass[12pt]{amsart}
\usepackage{amssymb,amsthm,amsmath,amstext}
\usepackage{mathrsfs}  
\usepackage{bm}        
\usepackage{mathtools} 

\usepackage[left=1.28in,top=.8in,right=1.28in,bottom=.8in]{geometry}

\usepackage{graphicx}
\usepackage[all]{xy}

\theoremstyle{plain}
\newtheorem{theorem}{Theorem}[section]
\newtheorem{prop}[theorem]{Proposition}
\newtheorem{lemma}[theorem]{Lemma}

\newtheorem{cor}[theorem]{Corollary}

\theoremstyle{definition}

\theoremstyle{remark}
\newtheorem{remark}[theorem]{Remark}

\newcommand{\sheaf}[1]{\mathscr{#1}}

\newcommand{\OO}{\sheaf{O}}

\newcommand{\PP}{\sheaf{P}}
\newcommand{\XX}{\sheaf{X}}
\newcommand{\YY}{\sheaf{Y}}

\newcommand{\QQ}{\sheaf{Q}}

\newcommand{\Brn}{{}_n\mathrm{Br}}

\newcommand{\Z}{\mathbb Z}

\DeclareMathOperator{\Br}{\mathrm{Br}}

\usepackage{hyperref}
\usepackage{color}

\hypersetup{colorlinks=true,linkcolor=blue,anchorcolor=blue,citecolor=blue,letterpaper=true}

\begin{document}

\title[ Local-global principle]
{ Local-global principle for reduced norms over function fields of
  $p$-adic curves}

\author[Parimala]{R.\ Parimala }
\address{Department of Mathematics \& Computer Science \\ %
Emory University \\ %
400 Dowman Drive~NE \\ %
Atlanta, GA 30322, USA}
\email{parimala@mathcs.emory.edu}
 
 \author[Preeti]{R.\ Preeti}
 \address{Department of Mathematics \\ %
Indian Institute of Technology (Bombay) \\ %
Powai, Mumbai-400076, India } 
\email{preeti@math.iitb.ac.in}
 
\author[Suresh]{V.\ Suresh}
\address{Department of Mathematics \& Computer Science \\ %
Emory University \\ %
400 Dowman Drive~NE \\ %
Atlanta, GA 30322, USA}
\email{suresh@mathcs.emory.edu}

\date{}

\begin{abstract}  Let $F$ be the function field of a $p$-adic curve. Let $D$ be a central simple algebra 
over $F$ of period $n$ and $\lambda \in F^*$. We show that if $n$ is coprime to $p$ and 
$D \cdot (\lambda) = 0$ in $H^3(F, \mu_n^{\otimes 2})$, then $\lambda$ is a reduced norm.
This leads to a Hasse principle for the $SL_1(D)$, namely an element $\lambda \in F^*$ is
a reduced norm from $D$ if and only if it is a reduced norm locally at all discrete valuations of $F$. 
\end{abstract} 
 
\maketitle

\def\ZZ{${\mathbf Z}$}
\def\ih{${\mathbf H}$}
\def\RR{${\mathbf R}$}
\def\IF{${\mathbf F}$}
\def\QQ{${\mathbf Q}$}
\def\IP{${\mathbf P}$}

\section{Introduction} 

Let $K$  be  a $p$-adic field and $F$ a function field in one-variable 
over $K$. Let $\Omega_F$ be the set of all discrete valuations of $F$.
Let $G$ be a  semi-simple simply connected linear algebraic group defined 
over $F$. It was conjectured in (\cite{CTPS})   that the Hasse principle holds
for principal homogeneous spaces under $G$ over $F$ with respect to
$\Omega_F$; i.e. if $X$ is a principal homogeneous space under $G$
over $F$ with $X(F_\nu) \neq \emptyset$ for all $\nu \in \Omega_F$, then
$X(F) \neq \emptyset$. If $G$ is $SL_1(D)$, where $D$ is a central simple
algebra over $F$ of square free index, it follows from the injectivity of the
Rost invariant (\cite{MS}) and a Hasse principle for $H^3(F, \mu_n)$ due to Kato
(\cite{Ka}), that this conjecture holds. This conjecture has been subsequently 
settled for classical  groups of type $B_n$, $C_n$ and $D_n$ (\cite{Hu}, \cite{preeti}). 
It is also settled for groups of type $^2A_n$ with 
the assumption that  $n+1$ is square free (\cite{Hu}, \cite{preeti}).

The main aim of this paper is to prove that the conjecture holds for
$SL_1(D)$ for any central simple algebra $D$  over $F$ with index 
coprime to $p$.  In fact we prove the following (\ref{main_theorem})

\begin{theorem} 
\label{thm}
Let $K$  be  a $p$-adic field and $F$ a function field in 
one-variable  over $K$. Let $D$ be a central simple algebra over $F$
of index coprime to $p$ and $\lambda \in F^*$. If $D \cdot (\lambda) 
= 0 \in H^3(F, \mu_n^{\otimes 2})$, then $\lambda$ is a reduced norm from
$D$. 
\end{theorem}

This together with Kato's result on the Hasse principle for $H^3(F, \mu_n)$ gives
 the  following (\ref{main_cor})
 
 \begin{theorem}
\label{cor}
 Let $K$ be $p$-adic field and  $F$ the function field of a curve 
over $K$.  Let $\Omega_F$ be the set of discrete valuations of $F$.
 Let $D$ be a central simple algebra over $F$ of index coprime to $p$ and 
 $\lambda \in F^*$. If $\lambda$ is a reduced norm from $D \otimes F_\nu$ for all 
 $\nu \in \Omega_F$, then $\lambda$ is a reduced norm from $D$.
  \end{theorem} 

In fact we may restrict the set of discrete valuations to the set of  divisorial discrete valuations of $F$; namely 
those discrete valuations of $F$ centered on a regular proper model of $F$  over the ring of 
integers in $K$.

Here are the main steps in the proof. We reduce to the case
where $D$ is a division algebra of period $\ell^d$ with $\ell$
a prime not equal to $p$. Given a central division algebra
$D$ over $F$ of period  $n = \ell^d$ with $\ell \neq p$ and $\lambda \in 
F^*$ with $D \cdot (\lambda ) = 0 \in H^3(F, \mu_n^{\otimes 2})$, 
we construct a degree $\ell$ extension $L$ of $F$ and $\mu \in L^*$
such that $N_{L/F}(\mu) = \lambda$, $(D \otimes L) \cdot (\mu)= 0$ 
and the index of $D \otimes L$ is 
strictly smaller than the index of $D$. Then, by induction on the index of $D$, 
$\mu$ is a reduced norm  from $D \otimes L$ 
and hence $N_{L/F}(\mu) = \lambda$ is a reduced norm from $D$.

 Let $\XX$ be a regular proper 
2-dimensional scheme over the ring of integers in $K$ with function field
$F$ and $X_0$ the reduced special fibre of $\XX$. 
By the patching techniques of Harbater-Hartman-Krashen (\cite{HH}, \cite{HHK1}), 
 construction of such a pair $(L, \mu)$
is reduced to the construction of compatible  pairs $(L_x, \mu_x)$ over $F_x$ for 
all $x \in X_0$ (\ref{patching_L_mu}), where for any $x \in X_0$, $F_x$ is the field of 
 fractions of the completion of the regular local ring at $x$ on $\XX$. 
We use local and global class field theory to construct such  local pairs $(L_x, \mu_x)$.
Thus this method cannot be extended to the more general situation where
$F$ is a function field in one variable over a complete discretely valued field with arbitrary 
residue field.

Here is the brief description of the organization of the paper.
In \S 3, we prove a few technical results concerning central simple 
algebras and reduced norms over global fields. These results are key to the later 
patching construction of the fields $L_x$ and $\mu_x \in L_x$ with required properties. 

In \S 4 we prove the following local variant of (\ref{thm})
\begin{theorem}
\label{thm_dvr} Let $F$  be a complete discrete valued field with residue  field $\kappa$.
Suppose that $\kappa$ is a local field or a global field. Let $D$ be a central simple algebra 
over $F$ of period $n$. Suppose that $n$ is coprime to char$(\kappa)$. Let $\alpha \in H^2(F,\mu_n)$ be the 
class of $D$ and  $\lambda \in F^*$.
If $\alpha \cdot (\lambda) = 0 \in H^3(F, \mu_n^{\otimes 2})$, then $\lambda$ is a reduced norm 
from $D$. 
\end{theorem}

Let $A$ be a complete regular local ring of dimension 2 with residue field $\kappa$ finite,
 field of fractions $F$ and maximal ideal $m = (\pi, \delta)$. Let $\ell$ be a prime not equal to 
 char$(\kappa)$.  Let $D$ be a central simple algebra over $F$ of index $\ell^n$ with $n \geq 1$
  and $\alpha$ the class of $D$ in $H^2(F, \mu_{\ell^n})$. Suppose that 
 $D$ is unramified on $A$ except possibly at $\pi$ and $\delta$.
  In \S 5, we analyze the structure of $D$.
 We prove that index of $D$ is equal to the period of $D$.   A similar analysis is done 
 by Saltman (\cite{S1}) with the additional assumption that $F$ contains all the primitive $\ell^n$-roots of 
unity, where $\ell^n$ is the index of $D$.
Let $\lambda \in F^*$. Suppose that $\lambda = u\pi^r\delta^t$ for some unit $u \in A$ and
$r, s \in \Z$ and $\alpha \cdot (\lambda) = 0 \in H^3(F, \mu_{\ell^n}^{\otimes 2})$.
In \S 6,   we construct possible pairs $(L, \mu)$ with $L/F$ of degree $\ell$, $\mu \in L$
such that $N_{L/F}(\mu) = \lambda$, ind$(D \otimes L) < $ ind$(D)$ and $\alpha \cdot (\mu) = 0
\in H^3(L, \mu_{\ell^n}^{\otimes 2})$.

Let $K$ be a $p$-adic field and $F$ a function field of a curve over $K$.  
Let $\ell$ be a prime not equal to $p$, $D$  a central division algebra over $F$ of 
index $\ell^n$  and $\alpha$ the class of $D$ in $H^2(F, \mu_{\ell^n})$.
Let $\lambda \in F^*$ with $\alpha \cdot (\lambda) = 0 \in H^3(F, \mu_{\ell^n}^{\otimes 2})$. 
Let $\XX$ be a normal  proper model of $F$ over the ring of integers in $K$ and
$X_0$ its reduced special fibre.
 In \S 7, we reduce the construction of $(L,  \mu)$ to the construction of local
 $(L_x, \mu_x)$ for all $x \in X_0$ with some compatible conditions along the ``branches''.
 
 Further assume that $\XX$ is regular and ram$_\XX(\alpha) \cup$ supp$_\XX(\lambda) \cup X_0$ 
 is a union of regular curves with normal crossings. We group the components of $X_0$ into 8 types
 depending on the valuation of $\lambda$, index of $D$   and
 the ramification type of $D$ along those components. 
We call some nodal points  of $X_0$ as special points depending on the type of components 
passing through the point.  We also say that two components of $X_0$ are type 2 connected 
if there is a sequence of type 2 curves connecting these two components. 
We prove that there is a regular proper model of $F$ with no special points 
and no type 2 connection between certain types of component (\ref{choice_of_model}).

Starting with a model  constructed in (\ref{choice_of_model}),  in \S 9, we construct $(L_P, \mu_P)$ for all nodal points 
of $X_0$ (\ref{choice_at_P}) with the required properties. In \S 10, using the class field results of \S 3, we construct 
$(L_\eta, \mu_\eta)$  for each  of the components  $\eta$ of $X_0$ which are compatible with $(L_P, \mu_P)$ 
when $P$ is in the component $\eta$.

Finally in \S 11, we prove the main results by piecing together all the 
constructions of \S 7, \S 9 and  \S10.

\section{Preliminaries} 
\label{Preliminaries}

In this section we recall a few definitions and facts about Brauer groups, Galois cohomology 
groups, residue homomorphisms and   unramified Galois cohomology groups.
We refer the reader to (\cite{CT-Santa-Barbara}) and (\cite{GS}).

Let $K$ be a   field  and $n \geq 1$.  Let $\Brn(K)$ be the $n$-torsion
subgroup of the Brauer group $\Br(K)$. Assume that $n$ is coprime to
the  characteristic of  $K$.  Let $\mu_n$ be the group 
of $n^{\rm th}$ roots of unity. For  $d \geq 1$ and $m \geq 0$,  let
$H^d(K, \mu_n^{\otimes m})$ denote the $d^{\rm th}$ Galois cohomology group of
$K$ with values in $\mu_n^{\otimes m}$.  We have $H^1(K, \mu_n) \simeq K^*/K^{*n}$ and $H^2(K,
\mu_n) \simeq \Brn(K)$.   For $a \in K^*$, let $(a)_n \in H^1(K,
\mu_n)$ denote the image of the class of $a$ in $K^*/K^{*n}$. 
 When there is no ambiguity of $n$, we drop $n$ and 
denote $(a)_n$ by $(a)$.  If $K$ is a product of finitely many fields $K_i$, we denote 
$\prod H^d(K_i, \mu_n^{\otimes m})$ by $H^d(K, \mu_n^{\otimes m})$.

 Every element of $H^1(K, \Z/n\Z)$ is represented by a pair 
 $(E, \sigma)$, where $E/F$ is a cyclic extension of degree dividing $n$ and
 $\sigma$ a generator of  Gal$(E/F)$. Let $r \geq 1$. Then $(E, \sigma)^r \in H^1(K, \Z/n\Z)$
 is represented by the pair $(E', \sigma')$ where $E'$ is the fixed field of
 the subgroup of Gal$(E/F)$ generated by $\sigma^{n/d}$, where $d = $ gcd$(n, r)$ 
 and $\sigma' = \sigma^{r}$. In particular if $r$ is coprime to $n$, then 
 $(E, \sigma)^r = (E, \sigma^r)$.
  Let $(E, \sigma) \in H^1(K, \Z/n\Z)$ and  $\lambda \in K^*$. 
 Let   $(E, \sigma, \lambda) = (E/F, \sigma, \lambda)$ denote the cyclic algebra over $K$
$$(E, \sigma, \lambda) = E \oplus Ly  \oplus \cdots \oplus Ey^{n-1}$$
with $y^n = \lambda$ and $ya = \sigma(a) y$.
The cyclic algebra $(E, \sigma, \lambda)$ is a central simple algebra and
its index is the order of $\lambda$ in $K^*/N_{E/K}(E^*)$ (\cite[Theorem 18, p. 98]{Albert}).
The pair $(E, \sigma)$  represents an element in $H^1(K, \Z/n\Z)$
and the element $(E, \sigma) \cdot (\lambda) \in H^2(K, \mu_n)$ is 
represented by the central simple algebra $(E, \sigma, \lambda)$. 
In particular $(E, \sigma, \lambda) \otimes E $ is a matrix algebra and hence ind$(E, \sigma, \lambda) \leq [E : F]$.

For $\lambda, \mu \in K^*$ we have (\cite[p. 97]{Albert})
$$(E, \sigma, \lambda) + (E, \sigma, \mu) = (E, \sigma, \lambda \mu) \in H^2(K, \mu_n).$$
In particular $(E, \sigma, \lambda^{-1}) = - (E, \sigma, \lambda)$.

Let $(E,  \sigma, \lambda)$ be a cyclic algebra over a field $K$ and $L/K$ be a
field extension. Since $E/K$ is  separable,  $E \otimes_KL$ is a product of field extensions 
 $E_i$, $1 \leq i \leq t$,   of $L$
with $E_i $ and  $ E_j$ isomorphic over $L$ and $E_i /L$ is cyclic with Galois group a subgroup of the Galois group of 
$E/K$.  Then $(E, \sigma, \lambda) \otimes L$ is Brauer equivalent to $(E_i, \sigma_i, \lambda)$ for any $i$,  with a suitable $\sigma_i$. 
In particular if  $L$ is a finite extension of $K$ and $EL$ is the 
composite of $E$ and $L$ in an algebraic closure of $K$,
then  $EL/L$ is cyclic with Galois group isomorphic to a subgroup of  the Galois group of $E/K$ and
$(E, \sigma, \lambda) \otimes L$ is Brauer equivalent to $(EL, \sigma', \lambda)$ for a suitable 
$\sigma'$.

\begin{lemma}
\label{cyclic_algebra_extn} Let $E/F$ be a cyclic extension of degree $n$, $\sigma$ a generator of 
Gal$(E/F)$ and $\lambda \in F^*$. Let $m$ be a factor of $n$ and $d = n/m$.  Let  
$M/F$ be the  subextension of  $E/F$ with $[M : F] = m$. 
Then $(E/F, \sigma, \lambda) \otimes F(\sqrt[d]{\lambda}) = 
(M(\sqrt[d]{\lambda})/F(\sqrt[d]{\lambda}), \sigma\otimes 1,   \sqrt[d]{\lambda})$.
\end{lemma}

\begin{proof}
We have $(E, \sigma)^d = (M, \sigma) \in H^1(F, \Z/n\Z)$ and hence   
$$
\begin{array}{lcl}
(E, \sigma, \lambda) \otimes F(\sqrt[d]{\lambda})  & = & (E(\sqrt[d]{\lambda})/F(\sqrt[d]{\lambda}), \sigma \otimes 1, \lambda) \\
&  = & (E(\sqrt[d]{\lambda})/F(\sqrt[d]{\lambda}), \sigma \otimes 1, ( \sqrt[d]{\lambda})^d)  \\
& =  & (E(\sqrt[d]{\lambda})/F(\sqrt[d]{\lambda}), \sigma\otimes 1)^d  \cdot  (\sqrt[d]{\lambda}) \\
& =  & (M(\sqrt[d]{\lambda})/F(\sqrt[d]{\lambda}), \sigma\otimes 1,   \sqrt[d]{\lambda}). 
\end{array}
$$
\end{proof}

Let $K$ be a field with a 
discrete valuation $\nu$, residue field $\kappa$ and valuation ring $R$. 
Suppose that $n$
is coprime   to the characteristic of $\kappa$. For any $d \geq 1$,
we have the
residue map $\partial_K : H^d(K, \mu_n^{\otimes i} ) \to H^{d-1}(\kappa, \mu_n^{\otimes i-1})$.
We also denote $\partial_K$ by $\partial$. 
An element $\alpha$ 
 in $H^d(K, \mu_n^{\otimes i})$ is called {\it unramified } at  $\nu$ or $R$ if $\partial(\alpha) = 0$. 
 The subgroup of all unramified elements is denoted by $H^d_{nr}(K/R, \mu_n^{\otimes i})$
 or simply $H^d_{nr}(K, \mu_n^{\otimes i})$. 
 Suppose that $K$ is complete with respect to $\nu$. Then we have an isomorphism 
 $H^d(\kappa, \mu_n^{\otimes i}) \buildrel{\sim}\over{\to}  H^d_{nr}(K, \mu_n^{\otimes i})$
 and the composition $H^d(\kappa, \mu_n^{\otimes i})
 \buildrel{\sim}\over{\to}   H^d_{nr}(K, \mu_n^{\otimes i}) 
 \hookrightarrow H^d(K, 
 \mu_n^{\otimes i})$ is denoted by $\iota_\kappa$ or simply $\iota$.

Let $K$ be a complete discretely  valued field with residue field $\kappa$, 
$\nu$ the discrete valuation on $K$
 and $\pi \in K^*$ a parameter. 
Suppose that $n$ is coprime to the characteristic of $\kappa$. 
Let $\partial: H^2(K, \mu_n) \to H^1(\kappa, \Z/n\Z)$ be the residue homomorphism.
Let $E/K$ be a cyclic unramified extension of  degree $n$ with residue field $E_0$
and $\sigma$ a generator  of Gal$(E/K)$ with $\sigma_0 \in $ Gal$(E_0/\kappa)$ induced by 
$\sigma$. Then  $(E, \sigma,  \pi)$ is a division algebra  over $K$ of degree $n$. 
For any $\lambda \in K^*$, 
we have  
$$\partial(E, \sigma, \lambda) = (E_0, \sigma_0)^{\nu(\lambda)}.$$
For $\lambda, \mu \in K^*$,  we have 
$$ \partial ((E, \sigma, \lambda) \cdot (\mu) ) = (E_0, \sigma_0) \cdot ((-1)^{\nu(\lambda)\nu(\mu)}
\theta),$$ 
where $\theta$ is the   image of $\frac{\lambda^{\nu(\mu)}}{\mu^{\nu(\lambda)}}$ in the residue field.

Suppose  $E_0$ is  a cyclic extension of $\kappa$ of degree $n$.
Then there is a unique unramified cyclic extension  
$E$ of $K$ of degree $n$ with residue field $E_0$. 
Let $\sigma_0$ be a generator
of Gal$(E_0/\kappa)$ and $\sigma \in $ Gal$(E/K)$ be the
lift of $\sigma_0$. Then $\sigma$ is a generator of Gal$(E/K)$.
We call the pair $(E, \sigma)$ the {\it lift of}  $(E_0,\sigma_0)$.

Let $X$ be an integral regular scheme with function field $F$. 
For every point $x$ of
$X$, let ${\OO}_{X, x}$ be the regular local ring at $x$ and $\kappa(x)$ the residue field
at $x$. Let  $\hat{{\OO}}_{X, x}$ be the
completion of ${\OO}_{X, x}$ at its maximal ideal $m_x$  and $F_x$ the field of 
fractions of $\hat{\OO}_{X, x}$.  Then every codimension  one  point $x$ of $X$ gives a discrete
valuation $\nu_x$ on $F$.  Let $n\geq 1$ be an integer   which is a unit on
$X$.  For any $d\geq 1$,  
the residue homomorphism  $ H^d(F,
  \mu_n^{\otimes j}) \to H^{d-1}(\kappa(x), \mu_n^{\otimes (j-1)})$ at the discrete
valuation $\nu_x$ is denoted by $\partial_x$.  An element $\alpha \in
H^d(F, \mu_n^{\otimes m})$ is said to be {\it ramified} at $x$ if
$\partial_x(\alpha ) \neq 0$ and {\it unramified } at $x$ if
$\partial_x(\alpha) = 0$.  If $X = $ Spec$(A)$ and $x$ a point of $X$ given by 
$(\pi)$, $\pi$ s prime element, we also denote $F_x$ by $F_\pi$ and $\kappa(x)$ by 
$\kappa(\pi)$.

\begin{lemma}
\label{generator}
Let $K$ be a complete discretely valued field  and $\ell$  a prime not equal to 
the characteristic of the residue field of $K$. Suppose that $K$ contains a primitive $\ell^{\rm th}$ root of
unity. Let $L/K$ be a cyclic field extension or the split extension  of degree $\ell$.
Let $\mu \in L$ and $ \lambda  = N_{L/K}(\mu) \in K$.
Then   there exists $\theta \in L$ with $N_{L/K}(\theta) = 1$
 such that  $L = K(\mu \theta)$ and $\theta$ is sufficiently close to 1. 
 \end{lemma}
 
 \begin{proof}    Since $[L : K] $ is a prime, 
 if $\mu \not\in K$, then $L = K(\mu)$. In this case $\theta = 1$ has the required properties. 
 Suppose that $\mu  \in K$.  If $L = \prod K$, let $\theta_0 \in K^*  \setminus \{ \pm 1 \}$ sufficiently close to 1
 and $\theta  = (\theta_0, \theta_0^{-1}, 1, \cdots , 1)$.
 Suppose $L$ is a field. Let $\sigma$
 be a generator of Gal$(L/K)$. Since $L/K$ is cyclic, we have $L = K(\sqrt[\ell]{a})$
 for some $a \in K^*$.   For any sufficiently large $n$,   
  $\theta =  (1+ \pi^n \sqrt[\ell]{a})^{-1}\sigma(1 + \pi^n \sqrt[\ell]{a})  \in L$ has the required properties.
 \end{proof}
 
\begin{lemma}
\label{galois_extension} 
Let $K$ be a  field and   $E/K$ be a finite extension of degree coprime
to char$(K)$. Let $L/K$ be a sub-extension of $E/K$ and $e = [E : L]$. 
Suppose  $L/K$ is Galois and $E = L(\sqrt[e]{\pi})$ for some $\pi \in L^*$. 
 Then $E/K$ is  Galois  if and only if  $E$ contains a primitive $e^{\rm th}$ root of 
 unity and  for  every $\tau \in $ Gal$(L/K)$, $\tau( \pi) \in E^{*e}$.
 \end{lemma}

\begin{proof}   
Suppose that  $E/K$ is Galois.  Let $f(X) = X^e - \pi \in L[X]$. Since 
$[E : L] = e$ and $E = L(\sqrt[e]{\pi})$, $f(X)$ is irreducible in $L[X]$.
Since $f(X)$ has one root in $E$ and $E/L$ is Galois, $f(X)$ has all the roots
in $E$.  Hence $E$ contains a primitive $e^{\rm th}$ root of unity. 
Let $\tau \in $ Gal$(L/K)$. Then $\tau$ can be extended to an automorphism  $\tilde{\tau}$ of $E$. 
We have $\tau(\pi) =  \tilde{\tau}(\pi) = (\tilde{\tau}(\sqrt[e]{\pi}))^e \in E^{*e}$. 

Conversely, suppose that $E$ contains a primitive $e^{\rm th}$ root of unity and 
$\tau(\pi) \in E^{*e}$ for every $\tau \in $ Gal$(L/K)$. 
Let $$g(X)  = \prod_{\tau ~\in ~Gal(L/K)} (X^e - \tau(\pi)).$$
Then  $g(X) \in K[X]$ and  $g(X)$ splits completely in $E$.
Since $e$ is coprime to char$(K)$,  the splitting field $E_0$ of 
$g(X)$ over $K$ is Galois. Since $L/K$ is Galois and $E$ is the composite of 
$L$ and $E_0$, $E/K$ is Galois. 
 \end{proof}

\begin{lemma}
\label{ramified_extension}
Let $F$ be a complete discretely valued field with residue field $\kappa$
and $\pi \in F$ a parameter.
Let $e$ be a natural number coprime   to the characteristic of $\kappa$.
If $L/F$ is a totally  ramified extension of degree $e$, then $L =  F(\sqrt[e]{v\pi})$
for some $v \in F$ which is a unit in the valuation ring of $F$.  Further if $e$ is a power of a prime $\ell$
and $\theta \in F^* \setminus F^{*\ell}$ is a norm form $L$, then $L = F(\sqrt[e]{\theta})$.
\end{lemma}

\begin{proof}  Since $F$ is a complete discretely valued field,  there is a unique
extension of the valuation $\nu$ on $F$ to a valuation $\nu_L$ on $L$.
Since   $L/F$ is totally ramified extension of degree $e$ and  $e$ is coprime to char$(\kappa)$, 
the residue field of $L$ is $\kappa$ and $\nu_L(\pi) = e$.
Let $\pi_L \in L$ with $\nu_L(\pi_L) = 1$.  Then $\pi = w\pi_L^e$ for 
some $w \in L$ with $\nu_L(w) = 0$. Since the residue field of $L$ is same as 
the residue field of $F$, there exists $w_1 \in F$ with $\nu(w_1) = 0$ and the image 
of $w_1$ is same as the image of $w$ in the residue field $\kappa$.  Since $L$ is complete 
and $e$ is coprime to char$(\kappa)$, by Hensel's Lemma, there exists $u \in L$ such that 
$w = w_1u^e$. Thus $\pi = w\pi_L^e= w_1u^e \pi_L^e  = w_1  (u\pi_L)^{e}$.
In particular $w_1^{-1}\pi \in L^{*e}$ and hence $L = F(\sqrt[e]{v\pi})$ with $v = w_1^{-1}$.

Let $\theta \in F^* \setminus F^{*e}$. Suppose that $\theta$ is a norm from $L$.
Let $\mu \in L$ with $N_{L/F}(\mu) = \theta$. 
Since $L = F(\sqrt[e]{v\pi})$ with $v \in F$ a unit in the valuation ring of $F$ and
$\pi \in F$ a parameter, $\sqrt[e]{v\pi} \in L$ is a parameter at the valuation of $L$.
Write $\mu = w_0(\sqrt[e]{v\pi})^s$ for some $w_0 \in L$ a unit at the valuation of $L$ and 
$s \in \Z$.   As above, we have  $w_0 = v_1u_1^e$ for some $v_1 \in K$ and $u_1 \in L$.
 Since $v_1 \in F$,  we have 
$$\theta = N_{L/F}(\mu) =  N_{L/F}(w_0(\sqrt[e]{v\pi})^s) = N_{L/F}(v_1u_1^e (\sqrt[e]{v\pi})^s) =
v_1^e N_{L/F}(u_1)^e (v\pi)^s.$$
Since  $e$ is a power of a prime $\ell$ and $\theta \not\in F^{*\ell}$, $s$ is coprime to $\ell$ and hence 
$L = F(\sqrt[e]{\theta})$.
\end{proof}

\begin{lemma}
\label{local_field_norm} Let $k$ be a local field and  $\ell$ a prime not
equal to the characteristic of the residue field of $k$.  Let $L_0/k$ 
be a an extension of degree $\ell$ and $\theta_0 \in k^*$.
If $\theta_0 \not \in k^{*\ell}$ and $\theta_0$ is a norm from $L_0$, then
$L_0 = k(\sqrt[\ell]{\theta_0})$.
\end{lemma}

\begin{proof}  Suppose that $L_0/k$ is ramified. Since 
$\theta_0 \not\in k^{*\ell}$, by (\ref{ramified_extension}), 
$L_0 = k(\sqrt[\ell]{\theta_0})$.

Suppose that $L_0/k$ is  unramified. 
Let $\pi$ be a parameter in $k$ and write $\theta_0 = u \pi^r$ with
$u$ a unit in the valuation ring of $k$. 
Since $\theta_0 $ is a norm from $L_0$,  $\ell$ divides $r$. 
Since $\theta_0$ not an $\ell^{\rm th}$ power in $k$, $u$ is not an $\ell^{\rm th}$ power
in $k$ and $k(\sqrt[\ell]{\theta_0}) = k(\sqrt[\ell]{u})$ is an unramified extension of $k$ of degree $\ell$. 
Since $k$ is a local field, there is only one unramified field extension of $k$ of degree $\ell$ and hence 
 $L_0 = k(\sqrt[\ell]{u}) = k(\sqrt[\ell]{\theta_0})$. 
\end{proof}

\begin{lemma}
\label{dvr_local_field_norm}   Suppose $F$ is a complete discretely valued field with residue feld 
$\kappa$  a local field.  
Let $\ell$ be  prime not   equal to  char$(\kappa)$.  
Let $L /F $ be a degree $\ell$ field extension  with $\theta$ a norm from $L$. 
If $\theta \not\in F ^{*\ell}$, 
then $L \simeq F (\sqrt[\ell]{\theta})$. 
\end{lemma}

\begin{proof}  If  $L/F$ is a ramified extension, then by (\ref{ramified_extension}),   $L \simeq F(\sqrt[\ell]{\theta})$. 
Suppose that   $L /F $ is an unramified extension.  
 Let $L_0 $ be the residue field 
of  $L$. Then $L_0 /\kappa $ is a field extension of degree $\ell$ and the image $\overline{\theta}$ 
 of  $\theta$ in $\kappa$ is a norm from $L_0 $.   Since $\theta \not\in F ^{*\ell}$, 
  $\overline{\theta}$  is not an $\ell^{\rm th}$ power in $\kappa$.  
 Since $\kappa $ is a local field,  $L_0 \simeq \kappa (\sqrt[\ell]{\overline{\theta}})$ (\ref{local_field_norm})
  and hence 
 $L \simeq F(\sqrt[\ell]{\theta})$.
 \end{proof}

\begin{lemma}
\label{matching_at_branch3}   
 Let $F$ be a complete discretely valued field with residue field $k$ a global field.
Let  $L/F$ be an unramified  cyclic  extension of degree coprime to char$(k)$ and
  $L_0$  the residue field of $L$. 
Let $\theta \in F$ be a unit in   the   valuation ring of $F$ and $\overline{\theta}$
be the image of $\theta$ in $k$. Suppose that $\theta$ is a norm from $L$.
If  $\mu_0 \in L_0$  with $N_{L_0/k}(\mu_0) = \overline{\theta}$, then there exists $\mu \in L$ 
such that $N_{L/F}(\mu) = \theta$ and the image of $\mu$ in $L_0$ is $\mu_0$. 
\end{lemma}

\begin{proof}  Let $\sigma$ be a generator  of the Galois group of $L/F$ and $\sigma_0$ be the induced
automorphism of $L_0/k$. 
Since $\theta \in F$ is a norm from $L$,  there exists  
$\mu' \in L$ with  $N_{L/F}(\mu') = \theta$.  Since $\theta$ is a unit at the discrete valuation of $F$, 
$\mu' \in L$ is a unit at the discrete valuation of $L$.
Let $\overline{\mu}'$ be the image of 
$\mu'$ in $L_0$. Then $N_{L_0/k}(\overline{\mu'}) = \overline{\theta}$
and hence $\overline{\mu}'\mu_0^{-1} \in L_0$ is a norm one element. 
Thus there exist  $a \in L_0$ such that  $\overline{\mu}'\mu_0^{-1} = a^{-1}\sigma_0(a)$.
Let $b \in L$ be a lift of $a$ and $\mu = \mu' b \sigma(b)^{-1}$. Then $N_{L/F}(\mu) = N_{L/F}(\mu') = 
\theta$ and the image of $\mu$ in $L_0$ is $\mu_0$.
\end{proof}

For $L = \prod_1^\ell F$, let  $\sigma$ be the automorphism of $L$ 
given by $\sigma(a_1, \cdots , a_\ell) = (a_2, \cdots , a_\ell, a_1)$. 
Then  any $\sigma^i$, $1 \leq i \leq \ell -1 $ is called  a {\it generator} of Gal$(L/F)$. 

\begin{lemma}
\label{hilbert90}
Let $F$ be a  field and   $\ell$  a prime not equal to 
the characteristic of $F$.  Let  $L$  be  a cyclic  extension 
of $F$ or the split extension of degree $\ell$  and $\sigma$ a generator of the Galois group of $L/F$. 
Suppose that there exists an integer $t \geq 1$ such that $F$ does not contain
a primitive $\ell^{{t}^{\rm th}}$ root of unity. 
Let  $\mu \in L$ with $N_{L/F}(\mu) = 1$ and $m \geq t$.  If   $\mu  \in L^{*\ell^{2m}},$ 
then there  exists $b \in L^*$ such that $\mu = b^{-\ell^m}\sigma(b^{\ell^m})$. 
\end{lemma}

\begin{proof}   Suppose $L = \prod F$ and $\mu \in L^{*\ell^{s}}$ for some $s \geq 1$ with $N_{L/F}(\mu) =1$.
Then  $\mu = (\theta_1^{s}, \cdots, \theta_\ell^s) \in L$ with  $ \theta_1^s \cdots \theta_\ell^s = 1$.
Let $b = (1, \theta_1, \cdots , \theta_{\ell-1}) \in L^*$. Then $\mu = b^{-s}\sigma(b^{s})$.

Suppose $L/F$ is a cyclic field extension.
Write $\mu = \mu_0^{\ell^{2m}}$ for some $\mu_0 \in L$.
Let $\mu_1 = \mu_0^{\ell^m}$.  Then $\mu = \mu_1^{\ell^m}$.
Let $\theta_0 = N_{L/F}(\mu_0)$  and $\theta_1 = N_{L/F}(\mu_1)$.
Then $\theta_1 = \theta_0^{\ell^m}$.
Since  $N_{L/F}(\mu) = 1$, we have $\theta_1^{\ell^m} = N_{L/F}(\mu_1^{\ell^m}) = 1$.
If $\theta_1 \neq 1$, then $F$ contains a primitive $\ell^{m^{\rm th}}$ root of unity.
Since $m \geq t$ and $F$ has no primitive $\ell^{t^{\rm th}}$ root of unity, $\theta_1 = 1$.
Hence $N_{L/F}(\mu_1) = 1$ and by Hilbert 90, $\mu_1 = b^{-1}\sigma(b)$ for some
$b \in L$.  Thus $\mu = \mu_1^{\ell^m} = b^{-\ell^m} \sigma(b^{\ell^m})$.
\end{proof}

\section{Global fields}

In this a section we prove a few technical results concerning Brauer group of global fields and 
reduced norms.
 We begin with the following.

   \begin{lemma}
\label{uglylemma} Let $k$ be a global field, $\ell$ a prime not equal char$(k)$,
$n, d \geq 2$  and $ r \geq 1$ be  integers. 
   Let $E_0$ be a cyclic extension of $k$, $\sigma_0$ a generator of the Galois group of $E_0/k$ and
   $\theta_0  \in k^*$. Let $\beta \in H^2(k, \mu_{\ell^n})$   be  such that 
  $r\ell \beta = (E_0, \sigma_0, \theta_0) \in H^2(k, \mu_{\ell^n})$. Let  $S$ be  a finite set of places 
  of $k$ containing all the places of $\kappa$ with $\beta \otimes k_\nu \neq 0$.  
Suppose for each $\nu \in S$,  there is a field extension   $L_\nu$ of $k_{\nu}$ of degree $\ell$ or
   $L_\nu $ is the split extension of $k_{\nu}$ of degree $\ell$  and $\mu_\nu \in L_\nu^*$ 
   such that  
   
   1) $N_{L_\nu/k_\nu}(\mu_\nu) = \theta_0$ 
   
   2)    $r\beta \otimes L_{\nu}  =  (E_0\otimes L_\nu, \sigma_0\otimes 1, \mu_\nu) $ 
   
   3)  ind$(\beta \otimes E_0 \otimes L_\nu) <  d $.
   
 \noindent
 Then there exists a field extension $L_0/k$ of
  degree $\ell$ and $\mu_0 \in L_0$ such that 
  
   1) $N_{L_0/k}(\mu_0) = \theta_0$ 
   
   2) $r\beta \otimes L_0  =  (E_0\otimes L_0, \sigma_0 \otimes 1, \mu_0 )  $ 
   
   3)  ind$(\beta \otimes E_0 \otimes L_0) <   d $
   
   4) $L_0 \otimes k_\nu \simeq L_\nu$ for all $\nu \in S$.
   
   5) $\mu_0$ is close to $\mu_\nu$ for all $\nu \in S$.

\end{lemma}
 
 \begin{proof} Let $\Omega_k$ be the  set of all places of $k$ and 
 $$ S' = S  \cup \{ \nu \in \Omega_k \mid   
 \theta_0   {~\rm is ~not ~a~ unit~at~} \nu    {~\rm  or ~} E_0/k  {~\rm is ~ramified ~at~} \nu  \}
  $$
Let $\nu \in S' \setminus  S$.
Then  $\beta \otimes k_\nu = 0$. 
  Let $L_\nu$ be a field 
 extension of $k_\nu$  of degree $\ell$ such that  $\theta_0 \in N(L_\nu^*)$.
 Let $\mu_\nu \in L_\nu$ with $N_{L_\nu/k_\nu}(\mu_\nu)
 = \theta_0$.  Since $\beta \otimes k_\nu = 0$, ind$(\beta \otimes E_0 \otimes L_\nu) = 1 < d$. 
 Since the corestriction  map cor $: H^2(L_\nu, \mu_{\ell^n}) \to H^2(k_\nu, \mu_{\ell^n})$ is
 injective (cf. \cite[Theorem 10, p. 237]{LF}) and  cor$(E_0 \otimes L_\nu,  \sigma_0 \otimes 1, \mu_\nu) = 
 (E_0 \otimes k_\nu, \sigma_0 \otimes 1, \theta_0) = r\ell \beta \otimes k_\nu = 0$, 
 $(E_0 \otimes L_\nu,  \sigma_0 \otimes 1, \mu_\nu) = 0  =  r \beta  \otimes L_\nu$. 
 Thus,  if necessary, by enlarging $S$, we assume that $S$ contains all those places $\nu$ of 
 $k$ with either $\theta_0$ is not a unit  at $\nu$ or $E_0/k$ is ramified at $\nu$ and that
there is at least one $\nu \in S$
 such that $L_\nu$ is a field extension of $k_\nu$ of degree $\ell$. 
 
 Let $\nu \in S$.  By (\ref{generator}),  there exists $\theta_\nu \in L_\nu$ such that 
 $N_{L_\nu/k_\nu}(\theta_\nu) = 1$,  $L_\nu = k_\nu(\theta_\nu\mu_\nu)$ and $\theta_\nu$ is sufficiently 
 close to 1.  In particular $\theta_\nu \in L_\nu^{\ell^n}$ and hence 
 $r\beta \otimes L_\nu = (E_0 \otimes L_\nu, \sigma_0 \otimes 1, \mu_\nu) 
 = (E_0 \otimes L_\nu, \sigma_0 \otimes 1, \mu_\nu \theta_\nu)$.
 Thus, replacing $\mu_\nu$ by $\mu_\nu \theta_\nu$,  
 we assume that  
 $L_\nu = k_\nu(\mu_\nu)$.  
 Let   $f_\nu(X) = X^\ell + b_{\ell-1, \nu}X^{\ell-1} + \cdots + b_{1, \nu}X +
(-1)^\ell \theta_0 \in k_\nu[X]$  be the minimal polynomial of $\mu_\nu$
over $k_{\nu}$.  

 By  Chebotarev density  theorem (\cite[Theorem 6.3.1]{FA}), there exists $\nu_0  \in \Omega_k
 \setminus S$ such that $E_0 \otimes k_{\nu_0}$ is the split extension of
 $k_{\nu_0}$.     
 By the strong  approximation  theorem (\cite[p. 67]{CFANT}), choose    
$b_j \in k$, $0 \leq j \leq \ell-1$ such that each $b_j$ is  sufficiently close enough  to
$b_{j, \nu}$ for all $\nu \in S$  and each $b_j$ is an integer at all $\nu
\not\in S \cup \{ \nu_0 \}$.    
 Let $L_0 = k[X]/ (X^\ell + b_{\ell-1}X^{\ell-1} + \cdots + b_{1}X +
 (-1)^\ell \theta_0)$ and $\mu_0 \in L_0$ be the image of $X$. 
 We now show that $L_0$ and $\mu_0$ have the required properties. 

Since each $b_j$ is sufficiently close enough to $b_{j, \nu}$ at each  $\nu \in S$,
 it follows from Krasner's lemma  that $L_0 \otimes k_{\nu} \simeq L_\nu$  and the image of $\mu_0 \otimes 1$ 
 in $L_\nu$ is close to $\mu_\nu$ for all
 $\nu \in S$ (cf. \cite[Ch. II, \S 2]{Se}).   Since $L_\nu$ is a field extension of $k_\nu$ of degree $\ell$ 
 for at least one  $\nu \in S$,  $L_0$ is a field extension of degree  $\ell$ over $k$. 
Since $X^\ell + b_{\ell-1}X^{\ell-1} + \cdots + (-1)^\ell\theta_0$ is
the minimal polynomial of $\mu_0$, we have $N(\mu_0)  = \theta_0$. 
 
To show that ind$(\beta \otimes E_0 \otimes L_0) <  d$ and
$r\beta = (E_0, \sigma_0, \mu_0)  \in H^2(L_0, \mu_{\ell^n})$,  by
Hasse-Brauer-Noether theorem (cf. \cite[p. 187]{CFANT}), 
it is enough to show that for every place $w$ of $L_0$,
ind$(\beta \otimes E_0 \otimes L_w) <  d$ and
$r\beta \otimes L_w = (E_0, \sigma_0, \mu_0) \otimes L_w \in H^2(L_w, \mu_{\ell^n})$. 

Let $w$ be a place of $L_0$ and $\nu$ a place of  $k$ lying below $w$.
 Suppose that $\nu \in S$.  Then   $L_0 \otimes k_\nu \simeq L_\nu$.
 Suppose $L_\nu  =  \prod k_\nu$ is the split extension. Then $L_w \simeq k_\nu$.
  By the assumption on $L_\nu$, we have 
 ind$(\beta \otimes E_0 \otimes k_\nu) <  d$. Since $\mu_\nu$ is close to $\mu_0$, 
 we have $r\beta \otimes L_\nu =   (E_0 \otimes L_\nu, \sigma_0, \mu_\nu)  =  
 (E_0 \otimes L \otimes k_\nu, \sigma_0, \mu_0)$. 
 
   Suppose that $L_\nu$ is a field extension of $k_\nu$ of degree $\ell$.
 Then $L_w \simeq L_0 \otimes k_\nu \simeq L_\nu$ and 
 by the assumption on $L_\nu$, we have 
   $r\beta \otimes L_\nu=  (E_0,\sigma_0, \mu_\nu) \otimes L_\nu$ and 
 ind$(\beta \otimes E_0 \otimes L_\nu) <  d $.  Since $\mu_0$ is close to $\mu_\nu$, 
 we have $r\beta \otimes L_\nu =   (E_0 \otimes L_\nu, \sigma_0, \mu_\nu)  =  
 (E_0 \otimes L \otimes k_\nu, \sigma_0, \mu_0)$.

Suppose that $\nu \not\in S$ and $\nu \neq \nu_0$.  Then $\theta_0$ is a unit at $\nu$, 
$E_0/k$ is unramified at $\nu$ and $\beta \otimes k_\nu = 0$.
Since each $b_j$ is an integer at $\nu$ and $\mu_0$ is  a root  of the polynomial
$X^\ell + b_{\ell -1}X^{\ell-1} + \cdots + b_1 X + (-1)^\ell\theta_0$, $\mu_0$ is an integer at $w$.
Since $\theta_0$ is a unit at $\nu$, $\mu_0$ is a unit at $w$.
In particular  $(E_0\otimes L_w, \sigma_0, \mu_0) = 0 = 
r\beta \otimes L_w$.  If $\nu = \nu_0$, then by the choice of $\nu_0$,  $\beta \otimes k_\nu = 0$, 
 $E_0 \otimes k_\nu$ is the split extension of $k_\nu$  
and hence $(E_0, \sigma_0, \mu_0) \otimes L_w= 0 = r\beta \otimes L_w$. 
\end{proof}

\begin{cor}
\label{globalfield_unramified}
 Let $k$ be a global field, $\ell$ a prime not equal char$(k)$,
$n$  and $ r \geq 1$ be  integers. Let $\theta_0  \in k^*$, $r \geq 1$ and  
  $\beta \in H^2(k, \mu_{\ell^n})$. Suppose that 
    $r\ell \beta = 0   \in H^2(k, \mu_{\ell^n})$ and $\beta \neq 0$. 
    Then there exists a field extension $L_0/k$ of
  degree $\ell$ and $\mu_0 \in L_0$ such that 
  $N_{L_0/k}(\mu_0) = \theta_0$, $r\beta \otimes L_0 = 0$    and 
   ind$(\beta \otimes L_0) <   $ ind$(\beta)$.
   \end{cor}

\begin{proof} Let $S$ be a finite set of places of $k$ containing 
all the places of $k$ with $\beta \neq 0$.  Let $\nu \in S$.
If  $\theta_0 \not \in k_\nu^{*\ell}$, then, let $L_\nu = k_\nu(\sqrt[\ell]{\theta_0})$
and $\mu_v = \sqrt[\ell]{\theta_0} \in L_\nu$. If
$\theta_0 \in k_\nu^{*\ell}$, then, let $L_\nu/k_\nu$ be any field extension of 
 degree $\ell$ and $\mu_ \nu = \sqrt[\ell]{\theta_0} \in k_\nu \subset L_\nu$. 
 In both the cases, we have $N_{L_\nu/k_\nu}(\mu_v) =\theta_0$. 
 Since $L_\nu/k_\nu$ is a degree $\ell$ field  extension, $\ell$ divides ind$(\beta)$
 and $k_\nu$ is a local field,   ind$(\beta \otimes L_\nu) < $ ind$(\beta)$ ( \cite[p. 131]{CFANT}).
 Since $r\ell \beta = 0$ and $L_\nu/k_\nu$ is a field extension of degree $\ell$, 
 $r\beta \otimes L_\nu = 0$. 
 Let $E_0 = k$. Then,   by (\ref{uglylemma}), there exist a field extension 
 $L_0/k$ of degree $\ell$ and $\mu \in L_0$ with required properties.
\end{proof}

 \begin{lemma}
 \label{local_index_cal} Let $k$ be a global field and $\ell$ a prime not equal to 
 char$(k)$.  Let $E_0/k$ be a cyclic extension of degree a power of $\ell$ and $\sigma_0$ a generator
 of Gal$(E_0/k)$. Let $n \geq 1$, 
 $\theta_0 \in k^*$  and $\beta \in H^2(k, \mu_{\ell^n})$ be such that 
 $r \ell \beta = (E_0, \sigma_0, \theta_0)$ for some $r \geq 1$. 
 Suppose that  $r\beta \otimes E_0 \neq 0$.  If $\nu$ is a place of $k$ such that 
 $\sqrt[\ell]{\theta_0} \not\in k_\nu$,  then 
 ind$(\beta \otimes E_0 \otimes k_\nu(\sqrt[\ell]{\theta_0})) < $ ind$(\beta \otimes E_0)$. 
 \end{lemma}
 
 \begin{proof}  Write $r \ell = m \ell^d$ with $m$ coprime to $\ell$. Then $ d \geq 1$.
 Since $m\ell^d \beta = r \ell \beta  = (E_0, \sigma_0, \theta_0)$, we have $m\ell^d \beta \otimes E_0 = 0$.
 Since $m$ is coprime to $\ell$ and the period of $\beta$ is a power  of $\ell$, 
 it follows that $\ell^d\beta \otimes E_0 = 0$.  Since $r\beta \otimes E_0 \neq 0$,  $\ell^{d-1}\beta \otimes E_0 \neq 0$
 and  per$(\beta \otimes E_0) = \ell^d$.
 
 Let $\nu$ be  a place of $k$. 
 Suppose that  $\sqrt[\ell]{\theta_0} \not\in E_0 \otimes k_\nu$.
 Then $[E_0\otimes k_\nu(\sqrt[\ell]{\theta_0}) : E_0\otimes k_\nu] = \ell$
 and hence  ind$(\beta \otimes E_0 \otimes k_\nu(\sqrt[\ell]{\theta_0})) < $ ind$(\beta \otimes E_0)$ (\cite[p. 131]{CFANT}).
 Suppose that  $\sqrt[\ell]{\theta_0} \in E_0 \otimes k_\nu$.
 Then $E_0 \otimes k_\nu(\sqrt[\ell]{\theta_0}) = E_0 \otimes k_\nu$.
  Write $E_0 \otimes k_\nu = \prod E_i$  with each $E_i$ a cyclic field extension of 
 $k_\nu$. Since 
 $E_0/k$ is a Galois extension, $E_i \simeq E_j$ for all $i$ and $j$
 and $m\ell^d\beta \otimes k_\nu = (E_0, \sigma_0, \theta_0) \otimes k_\nu = 
 (E_i, \sigma_i, \theta_0)$ for all $i$,  for  suitable    generators $\sigma_i$  of Gal$(E_i/k_\nu)$. 
 Since   $\sqrt[\ell]{\theta_0} \in E_0 \otimes k_\nu$, $\sqrt[\ell]{\theta_0} \in E_i$ for all $i$ and
 hence $\theta_0^{[E_i : \kappa_\nu]/\ell} \in N_{E_i/k_\nu}(E_i^*)$. 
 Since the period of $(E_i, \sigma_i, \theta_0)$ is  equal to the order of 
 the class of $\theta_0$  in the group $k_\nu^*/N_{E_i/k_\nu}(E_i^*)$ (\cite[p. 75]{Albert}),  
  per$(E_i, \sigma_i, \theta_0) \leq  [E_i : k_\nu]/\ell < [E_i : k_\nu]$.  
 
 Suppose that  per$(\beta \otimes k_\nu) \leq 
 [E_i : k_\nu]$. Since $k_\nu$ is a local field, per$(\beta \otimes E_i ) = 1$. 
 Thus per$(\beta \otimes E_0 \otimes k_\nu ) = $ per$(\beta \otimes E_i) = 1 < \ell^d = $ per$(\beta \otimes E_0)$.
  
  Suppose that  per$(\beta \otimes k_\nu) >
 [E_i : k_\nu]$.  Since $m \ell^d\beta \otimes  k_\nu = (E_i, \sigma_i, \theta_0)$ and $m$ is coprime to 
 $\ell$, 
 we have  per$(\beta \otimes k_\nu) \leq \ell^d$per$(E_i, \sigma_i, \theta_0)$.
  Since $k_\nu$ is a local-field, 
  $${\rm per} (\beta \otimes E_0 \otimes k_\nu) =  {\rm per} (\beta \otimes E_i) =  \frac{{\rm  per}(\beta \otimes k_\nu) }{ [E_i : k_\nu]}
  \leq \frac{\ell^d{\rm per}(E_i, \sigma_i, \theta_0)}{[E_i : k_\nu]} < \ell^d = {\rm per}(\beta \otimes E_0).$$
  Since $k_\nu$ is a local field, period equals index and hence the lemma follows.
 \end{proof}

\begin{prop}
\label{global_field_non_zero} Let $k$ be a global field and $\ell$ a prime not equal to 
 char$(k)$.  Let $E_0/k$ be a cyclic extension of degree a power of $\ell$ and $\sigma_0$ a generator
 of  Gal$(E_0/k)$. Let 
 $\theta_0 \in k^*$  and $\beta \in H^2(k, \mu_{\ell^n})$ be such that 
 $r \ell \beta = (E_0, \sigma_0, \theta_0)$ for some $r \geq 1$. Suppose that $r \beta \otimes E_0  \neq  0$.  
  Then there exist a field extension   $L_0/\kappa$ of degree $\ell$ and  $\mu_0 \in L_0$ such that \\
 1)  $N_{L_0/k}(\mu_0 ) = \theta_0$ \\
 2) ind$(\beta \otimes E_0 \otimes L_0) < $ ind$(\beta \otimes E_0)$ \\
 2) $r\beta \otimes L_0 = (E_0 \otimes L_0, \sigma_0 \otimes 1, \mu_0)$. \\
 \end{prop}

\begin{proof} Let $S$ be the finite set of places of $k$ consisting of all 
those places $\nu$ with $\beta \otimes k_\nu \neq 0$. Let $\nu \in S$.
Suppose that $\theta_0 \not\in k_\nu^\ell$. Let $L_\nu = k_\nu(\sqrt[\ell]{\theta_0})$
and $\mu_\nu = \sqrt[\ell]{\theta_0} \in L_\nu$. Then $N_{L_\nu/k_\nu}(\mu_\nu) = \theta_0$.
By (\ref{local_index_cal}), ind$(\beta \otimes E_0 \otimes k_\nu)
< $ ind$(\beta \otimes E_0)$.  In particular ind$(\beta \otimes E_0 \otimes L_\nu) \leq $
ind$(\beta \otimes E_0 \otimes k_\nu)
< $ ind$(\beta \otimes E_0)$.
Since cor$_{L_\nu/k_\nu}(r\beta \otimes L_\nu) = r\ell \beta = (E_0 \otimes k_\nu, 
\sigma_0, \theta_0) = $  cor$_{L_\nu/k_\nu}(E_0 \otimes L_\nu, \sigma_0 \otimes 1, \mu_\nu)$
and corestriction is injective (cf. \cite[Theorem 10, p. 237]{LF}), 
we have $r \beta \otimes L_\nu = (E_0 \otimes L_\nu, \sigma_0 \otimes 1, 
\mu_\nu)$.    

Suppose that $\theta_0 = \mu_\nu^\ell$ for some $\mu _\nu \in k_\nu $.  
Since $k_\nu$ is local field containing a primitive $\ell^{\rm th}$ root of unity
and $E_0 \otimes k_\nu$ is a cyclic extension, there exists a    cyclic field extension
$L_\nu/k_\nu$  of degree $\ell$ which is not contained in 
$E_0 \otimes k_\nu$. Then $N_{L_\nu/k_\nu}(\mu_\nu) = \mu_\nu^\ell = \theta_0$.
Since $L_\nu $ is not a subfield of $E_0 \otimes k_\nu$, ind$(\beta \otimes E_0 \otimes L_\nu) <  $
ind$(\beta \otimes E_0 \otimes k_\nu) \leq$ ind$(\beta \otimes E_0)$ (\cite[p. 131]{CFANT}).  Since cor$_{L_\nu/k_\nu}(r\beta \otimes L_\nu) = r\ell \beta \otimes k_\nu
 =  (E_0 \otimes k_\nu, \sigma_0 \otimes 1,  \mu_\nu^\ell) = $ 
 cor$_{L_\nu/k_\nu}(E_0 \otimes L_\nu, \sigma_0 \otimes 1, \mu_\nu)$, by (cf. \cite[Theorem 10, p. 237]{LF}), 
    we have $r\beta \otimes L_\nu = (E_0 \otimes L_\nu, \sigma_0 \otimes 1, \mu_\nu)$.
    
  By (\ref{uglylemma}), we have the required $L_0$ and $\mu_0$.
\end{proof}

\begin{prop}
\label{global_field_zero} Let $k$ be a global field and $\ell$ a prime not equal to 
 char$(k)$.  Let $E_0/k$ be a cyclic extension of degree a  positive  
 power of $\ell$ and $\sigma_0$ a generator
 of  Gal$(E_0/k)$. Let 
 $\theta_0 \in k^*$  and $\beta \in H^2(k, \mu_{\ell^n})$ be such that 
 $r \ell \beta = (E_0, \sigma_0, \theta_0)$ for some $r \geq 1$. Suppose that $r \beta \otimes E_0  =  0$.  
  Let $L_0$ be the unique subfield of $E_0$ of degree $\ell$ over 
 $k$. 
 Then there exists $\mu_0 \in L_0$ such that \\
 1)  $N_{L_0/k}(\mu_0 ) = \theta_0$ \\
 2) $r\beta \otimes L_0 = (E_0 \otimes L_0, \sigma_0 \otimes 1, \mu_0)$. \\
 \end{prop}

\begin{proof}  Since $r \beta\otimes  E_0 = 0$ and $E_0/k$ is a cyclic extension, we have 
$r\beta = (E_0, \sigma_0, \mu')$ for some $\mu' \in k$.  We have $(E_0, \sigma_0, \mu'{^\ell}) = \ell r\beta = 
(E_0, \sigma_0, \theta_0)$. Thus $\theta_0 = N_{E_0/k}(y) \mu'{^\ell}$. Let $\mu_0 = N_{E_0/L_0}(y) \mu' \in L_0$.
Since $L_0 \subset E_0$, we have $r\beta \otimes L_0 = (E_0/L_0, \sigma_0^\ell, \mu') = (E_0/L_0, \sigma_0^\ell, 
N_{E_0/L_0}(y) \mu') = (E_0/L_0, \sigma_0^\ell, \mu_0)$  (cf.  \S \ref{Preliminaries}) 
and $$N_{L_0/k}(\mu_0)  = N_{E_0/k}(N_{E_0/L_0}(y))  \mu'{^\ell} = \theta_0.$$ 
\end{proof}

\begin{cor}
\label{global_field_zero_cor} Let $k$ be a global field and $\ell$ a prime not equal to 
 char$(k)$.  Let $E_0/k$ be a cyclic extension of degree a  power of $\ell$ and $\sigma_0$ a generator
 of Gal$(E_0/k)$. Let 
 $\theta_0 \in k^*$  and $\beta \in H^2(k, \mu_{\ell^n})$ be such that 
 $r \ell \beta = (E_0, \sigma_0, \theta_0)$ for some $r \geq 1$. Suppose that $r \beta \otimes E_0  =  0$.  
  Let $L_0$ be the unique subfield of $E_0$ of degree $\ell$ over 
 $k$.  Let $S$ be a finite set of places of $k$. Suppose for each $\nu \in S$ there exists
  $\mu_\nu \in L_0 \otimes k_\nu$ such that  \\
 $\bullet$  $N_{L_0 \otimes k_\nu/k_\nu}(\mu_\nu ) = \theta_0$ \\
$\bullet$ $r\beta \otimes L_0 \otimes k_\nu = (E_0 \otimes L_0 \otimes k_\nu, \sigma_0 \otimes 1, \mu_\nu)$.\\ 
 Then there exists $\mu \in L_0$ such that \\
 1)  $N_{L_0/k}(\mu  ) = \theta_0$ \\
 2) $r\beta \otimes L_0 = (E_0 \otimes L_0, \sigma_0 \otimes 1, \mu)$ \\
 3)  $\mu$ is close to $\mu_\nu$ for all $\nu \in S$. 
 \end{cor}

\begin{proof} By (\ref{global_field_zero}), there exists  $\mu_0 \in L_0$ such that \\
$\bullet$  $N_{L_0/k}(\mu_0 ) = \theta_0$ \\
$\bullet$  $r\beta \otimes L_0 = (E_0 \otimes L_0, \sigma_0 \otimes 1, \mu_0)$. \\
 Let $\nu \in S$. Then we have \\
$\bullet$  $N_{L_0/k}(\mu_0 ) = \theta_0 = N_{L_0 \otimes k_\nu/k_\nu}(\mu_\nu)$ \\
$\bullet$  $ (E_0 \otimes L_0 \otimes k_\nu, \sigma_0 \otimes 1, \mu_0) = 
(E_0 \otimes L_0 \otimes k_\nu, \sigma_0 \otimes 1, \mu_\nu)$.\\
Let $b_\nu = \mu_0 \mu_\nu^{-1} \in L_0 \otimes k_\nu$.
Then $N_{L_0 \otimes k_\nu/k_\nu}(b_\nu) = 1$ and $(E_0 \otimes L_0 \otimes k_\nu, \sigma_0 \otimes 1, 
b_\nu) = 1$. Thus, there exists $a_\nu \in E_0 \otimes L_0 \otimes k_\nu$ with $N_{E_0\otimes L_0 \otimes k_\nu 
/ L_0 \otimes k_\nu }(a_\nu) = b_\nu$. We have $N_{E_0\otimes L_0 \otimes k_\nu 
/ k_\nu }(a_\nu) = N_{ L_0 \otimes k_\nu 
/  \otimes k_\nu }(b_\nu) = 1$. Since $E_0/k$ is a cyclic extension with $\sigma_0$ a generator 
of Gal$(E_0/k)$,  for each $\nu \in S$, there exists $c_\nu \in E_0 \otimes L_0 \otimes k_\nu$
such that $a_\nu = c_\nu^{-1}(\sigma_0\otimes 1)(c_\nu)$. 
By the weak approximation, there exists $c \in E_0 \otimes L_0$ such that $c$ is close to $c_\nu$ for 
all $\nu \in S$. Let $a = c^{-1}(\sigma\otimes 1)(c) \in E_0 \otimes L_0$ and $\mu = 
\mu_0 N_{E_0 \otimes L_0/L_0 }(c) \in L_0$.  Then $\mu$ has all the required properties. 
\end{proof}

\section{complete discretely valued fields}

Let $F$ be a complete discretely valued field with residue field $\kappa$.
Let $D$ be a central simple algebra over $F$ of period  $n$ coprime  to char$(\kappa)$.
Let  $\lambda \in F^*$ and $\alpha \in H^2(F, \mu_n)$ be the class of $D$. 
In this section we analyze the condition $\alpha \cdot (\lambda) = 0$ and
we use this analysis in the proof of our main result (\S 10).
As a consequence, we also deduce that  if $\kappa$ is either a local field or a global field and 
$\alpha \cdot (\lambda) = 0$ in $H^3(F, \mu_n^{\otimes 2})$, 
then $\lambda$  is a reduced norm from $D$. 

We use the following notation throughout this section:\\
$\bullet$ $F$ a complete discretely valued field.\\
$\bullet$ $\kappa$ the  residue field of $F$.\\
$\bullet$ $\nu$ the discrete valuation on $F$.\\
$\bullet$ $\pi \in F^*$ a parameter at $\nu$. \\
$\bullet$ $n \geq 2$ an integer coprime to char$(\kappa)$\\
$\bullet$ $D$ a central simple algebra over $F$ of period $n$.\\
$\bullet$ $\alpha \in H^2(F, \mu_n)$ the class representing $D$.\\
 Let $E_0$ be the cyclic extension of 
$\kappa$ and $\sigma_0 \in Gal(E_0/\kappa)$
be such that  $\partial(\alpha) = (E_0, \sigma_0)$.
Let $(E, \sigma)$ be the lift of $(E_0, \sigma_0)$ (cf. \S 2).
The pair  $(E, \sigma)$ or $E$  is called the {\it lift of the residue } of $\alpha$.
The following is well known.

\begin{lemma} 
\label{rbc} Let $\alpha \in H^2(F, \mu_n)$, $(E, \sigma)$ the lift of the  residue  of 
$\alpha$.  Then $\alpha = \alpha' + (E, \sigma, \pi)$ for some $\alpha' \in H^2_{nr}(F, \mu_n)$.  
Further  $\alpha' \otimes E =  \alpha \otimes E$ is independent of the choice of $\pi$. 
\end{lemma}

\begin{proof} Since $\partial(E,\sigma, \pi) = \partial(\alpha)$, 
$\alpha'  = \alpha - (E,\sigma, \pi) \in H^2_{nr}(F, \mu_n)$ and  
 $\alpha = \alpha' + (E, \sigma, \pi)$.
\end{proof}

\begin{lemma}
\label{dvr_ind} Let $n \geq 2$ be coprime to char$(\kappa)$ and 
 $\alpha \in H^2(F, \mu_n)$. If $\alpha = \alpha' + (E, \sigma, \pi)$ 
 as in (\ref{rbc}), then ind$(\alpha) = $ ind$(\alpha' \otimes E)[E : F] 
 = $ ind$(\alpha \otimes E) [E : F]$.
\end{lemma}

\begin{proof}
Cf. (\cite[Proposition 1(3)]{FS}  and \cite[5.15]{JW}).  
\end{proof}

\begin{lemma}
\label{not_type_6} Let $E$ be the lift of the residue of $\alpha$.
Suppose there exists a totally ramified extension $M/F$ which splits $\alpha$, 
then $\alpha \otimes E = 0$. 
\end{lemma}

\begin{proof} Write $\alpha = \alpha' + (E, \sigma,\pi)$ as in (\ref{rbc}).
Since $\alpha' \otimes E = \alpha \otimes E$, we have  $\alpha' \otimes E \otimes M = 0$. 
Since  $E\otimes M/E$ is totally ramified, the residue field of $E\otimes M$ is same as the residue field of 
$E$. Since  $\alpha'\otimes E \otimes M = 0$ and $\alpha' \otimes E$ is unramified, it follows from 
(\cite[7.9 and 8.4]{serrecohoinvariants}) that  $\alpha \otimes E =  \alpha' \otimes E = 0$.
\end{proof}

\begin{lemma} 
\label{dvr_ind_per} Let $ n \geq 2$ be coprime to char$(\kappa)$.
Let $\alpha \in H^2(F, \mu_n)$ and $(E, \sigma)$  be the lift of the
residue of $\alpha$. If  $\alpha \otimes E = 0$,  
then $\alpha = (E, \sigma, u\pi)$ for some $u \in F^*$
which is a unit at the discrete valuation and  per$(\alpha) = $ ind$(\alpha)$.
\end{lemma}

\begin{proof}  
We have  $\alpha = \alpha' + (E, \sigma, \pi)$ as in (\ref{rbc}).
Since $\alpha' \otimes E = \alpha \otimes E = 0$, we have $\alpha' =
(E, \sigma, u )$ for  $u \in F^*$.  
Since $E/F$  and $\alpha'$ are  unramified at the discrete valuation of $F$, $u$ 
is a unit at the discrete valuation of $F$.
We have  $\alpha = (E, \sigma, u) + (E, \sigma, \pi) = (E, \sigma, u\pi)$.
Since $E/F$ is an unramified extension   
and $u\pi$ is a parameter, $(E, \sigma, u\pi)$ is a division algebra and its period
is  $ [E : F]$. In particular  ind$(\alpha) = $ per$(\alpha)$.
\end{proof}

\begin{theorem} 
\label{local_local_period_index} Let $F$ be a complete discretely valued field with   
residue field $\kappa$. Suppose that $\kappa$ is a local field.  
Let $\ell$ be a  prime not equal to the characteristic of
 $\kappa$, $n = \ell^d$ and $\alpha \in H^2(F, \mu_n)$. 
Then  per$( \alpha) = $ ind$(\alpha)$. 
 \end{theorem}

\begin{proof}  Write $\alpha = \alpha' + (E, \sigma, \pi)$ as in (\ref{rbc}).
Then $E$ is an unramified cyclic extension of $F$ with $\partial(\alpha) = 
(E_0, \sigma_0)$ and $\alpha'$ is unramified at the discrete valuation of $F$.
Let $\overline{\alpha}'$ be the image of $\alpha'$ in $H^2(\kappa, \mu_n)$

Suppose that  per$(\partial(\alpha)) = $ per$(\alpha)$.
Then per$(\partial(\alpha)) =  [E_0 : \kappa]$. Since 
$F$ is complete discretely valued field and $E/F$ unramified extension, we have
$[E_0 : \kappa] = [E : F]$. Thus, 
$$
\begin{array}{rcl}
 0 & = &  per(\alpha) \alpha \\
 & = &  per(\alpha) (\alpha' + (E, \sigma, \pi))\\
 & = & per(\alpha) \alpha' + per(\alpha) (E, \sigma, \pi) \\
 & = & per(\alpha) \alpha'  +  [E : F](E, \sigma, \pi) \\
 & = & per(\alpha) \alpha'.
 \end{array}
 $$
 In particular, per$(\alpha')$ divides per$(\alpha) = [E_0, \kappa ] = [E : F] $.
 Since $\kappa$ is a 
local field, $\overline{\alpha}' \otimes E_0$ is zero (\cite[p.131]{CFANT})  and hence 
$\alpha' \otimes E$ is zero. 
By (\ref{dvr_ind_per}), we have   $\alpha = (E, \sigma, \theta\pi)$ 
for some $\theta \in F$ which is a unit in the valuation ring. In particular, 
$\alpha$ is cyclic and 
ind$(\alpha) = $  per$(\alpha) =  [E : F]$. 

Suppose that per$(\partial(\alpha)) \neq $ per$(\alpha)$. Then per$(\partial(\alpha)) < $ per$(\alpha)$.
   Since  per$(\partial(\alpha)) = $ per$(E,\sigma, \pi)$, 
we have   per$(\alpha) = $ per$(\alpha')$.  Since $\kappa$ is a local field,  per$(\overline{\alpha}')
 = $ ind$(\overline{\alpha}')$. 
Let $E_0$ be the residue field of $E$. Since per$(\overline{\alpha}') = $ per$(\alpha')$
and per$(\partial(\alpha)) = [E_0 : \kappa]$, we have $[E_0 : \kappa] < $ per$(\overline{\alpha}')$.
Since $\kappa$ is a local field, 
$$ {\rm ind}(\overline{\alpha}' \otimes E_0) = 
   \frac{{\rm per}(\overline{\alpha}')}{[E_0 : \kappa]}.$$ 
Since $E$ is a complete discrete valued field
with residue field $E_0$ and $\alpha'$ is unramified at the discrete valuation of $E$, we have 
 ind$(\alpha' \otimes E) = $ ind$(\overline{\alpha}' \otimes E_0)$.
 Thus,  we have 
 $$
 \begin{array}{rcll}
{\rm  ind}(\alpha) &  =  &  ind(\alpha' \otimes E) [E : F]  & ({\rm by} (\ref{dvr_ind} )) \\ 
 & =  &  ind(\overline{\alpha}' \otimes E_0)[E_0 : \kappa] \\
 & = & \frac{{\rm per}(\overline{\alpha}')}{[E_0 : \kappa]} [E_0 : \kappa] \\
 & = & {\rm per}(\overline{\alpha}') = {\rm per}(\alpha). 
 \end{array}
 $$
 \end{proof}

  \begin{prop}
 \label{corestriction}
 Suppose that $\kappa$ is a local field. Let $n \geq 2$ be coprime to
 char$(\kappa)$. 
 If  $L/ F$ is  a finite field extension, then  the corestriction homomorphism 
 $H^3(L, \mu_n^{\otimes 2}) \to H^3(F,
  \mu_n^{\otimes 2})$ is bijective.  
  \end{prop}
  
  \begin{proof} Let $k'$ be the residue field of $L$.
 Since $k$ and $k'$ are   local fields,  $H^3(k, \mu_n^{\otimes 2}) 
 = H^3(k', \mu_n^{\otimes 2}) = 0$ (\cite[p. 86]{Serre_Gal_coh}). Since $F$ and $L$ are 
 complete discrete valued fields, the residue homomorphisms 
$H^{3}(F, \mu_n^{\otimes 2}) \buildrel{\partial_F}\over{\to} H^2(k, \mu_n)$ and
$H^{3 }(L, \mu_n^{\otimes 2}) \buildrel{\partial_L}\over{\to} H^2(k', \mu_n)$ 
are isomorphisms (cf. \cite[7.9]{serrecohoinvariants}). The proposition  follows from 
the  commutative diagram 
$$
\begin{array}{ccc}
H^{3 }(L, \mu_n^{2}) & \buildrel{\partial_L}\over{\to}   & H^2(k', \mu_n) \\
\downarrow & & \downarrow \\
H^{3 }(F, \mu_n^{\otimes 2}) & \buildrel{\partial_F}\over{\to} &  H^2(k, \mu_n),
\end{array}
$$
where the vertical arrows are the corestriction maps (\cite[8.6]{serrecohoinvariants}).  
    \end{proof}

\begin{lemma} 
\label{dot-zero}Let  $\ell$ be
a prime not equal to char$(\kappa)$  and $n = \ell^d$ for some $d \geq 1$.
  Let $\alpha \in H^2(F, \mu_n)$  and  $\lambda \in F^*$.  Write $\lambda = \theta \pi^r$
  for some $\theta, \pi \in F$ with $\nu(\theta) = 0$ and $\nu(\pi) = 1$.
  Let $(E, \sigma)$ be the lift of the residue of $\alpha$ and $\alpha = \alpha' + (E, \sigma, \pi)$ as in (\ref{rbc}).   
  Then 
  $$
  \partial( \alpha \cdot (\lambda)) = 0 \iff  r\alpha' = (E, \sigma, \theta) \iff 
 r \alpha = (E, \sigma, \lambda).$$
   In particular, if $\partial( \alpha \cdot (\lambda)) = 0$ and
   $ r = \nu(\lambda)$  is coprime to $\ell$,  then 
 ind$(\alpha \otimes
  F(\sqrt[\ell]{\lambda})) < $ ind$(\alpha)$ and $\alpha 
  \cdot (\sqrt[\ell]\lambda) = 0 \in H^3(F(\sqrt[\ell]{\lambda}),\mu_n^{\otimes 2})$.  
\end{lemma}

\begin{proof}  Since $r\alpha = r\alpha' + (E, \sigma, \pi^r)$ and $\lambda = \theta \pi^r$, 
$r\alpha = (E, \sigma, \lambda)$ if and only if $r\alpha' = (E, \sigma, \theta)$.

  We have 
 $$ \partial(\alpha \cdot (\lambda) )  =   \partial((\alpha' + 
  (E, \sigma, \pi)) \cdot (\theta \pi^r) )
   =  r\overline{\alpha}' + (E_0, \sigma_0, \overline{\theta}^{-1}),$$
   where $\partial(\alpha) = (E_0, \sigma_0)$. 
   
 Thus  $\partial(\alpha \cdot (\lambda) ) = 0$ if and only if $r\overline{\alpha}' + (E_0, \sigma_0, \overline{\theta}^{-1}) 
 = 0$ if and only if $r\overline{\alpha}' =  (E_0, \sigma_0, \overline{\theta})$ if and only if 
 $r\alpha  =  (E, \sigma,  \theta)$ ( $F$ being  complete).
  
%
%
%
 \end{proof}

\begin{lemma} 
\label{type-other} Let $ n \geq 2$ be coprime to char$(\kappa)$ and $\ell$
a prime which divides $n$.
Let $\alpha \in H^2(F, \mu_n)$, $\lambda = \theta \pi^{\ell r}
 \in F^*$ with $\theta$ a unit in the valuation ring of $F$, $\pi$ a parameter and $\alpha  =  \alpha' + (E, \sigma, \pi)$
 be as in (\ref{rbc}). 
 Suppose that     
 $\alpha \cdot (\lambda) = 0 \in H^3(F, \mu_n^{\otimes 2})$ and there exist an extension 
 $L_0$ of $\kappa$ of degree $\ell$ and $\mu_0 \in L_0$ 
 such that \\
$\bullet$   $N_{L_0/\kappa}(\mu_0) = \overline{\theta}$, \\
$\bullet$  $r\overline{\alpha}' \otimes
 L_0  = (E_0\otimes L_0,  \sigma_0 \otimes 1, \mu_0) .$\\
  Then, there exist an unramified 
 extension $L$ of $F$ of degree $\ell$ and $\mu \in L$ such that  \\
 $\bullet$ residue field of $L$ is  $L_0$, \\
 $\bullet$ $\mu$ a unit in the valuation ring of $L$, \\
 $\bullet$ $\overline{\mu} = \mu_0$, \\
 $\bullet$    $N_{L/F} (\mu) = \theta$,  \\
 $\bullet$  $\alpha \cdot (\mu \pi^r)  \in H^3(L, \mu_n^{\otimes 2})$ is unramified. 
\end{lemma}

\begin{proof}  Since $\ell$ is a prime and  $[L_0 : \kappa] = \ell$,  $L_0 = \kappa(\mu_0')$ for
any  $\mu_0' \in L_0 \setminus \kappa$. 
Let $g(X) = X^\ell + b_{\ell-1} X^{\ell-1} + \cdot + b_1X + b_0  \in \kappa[X]$ be the minimal polynomial 
of $\mu_0'$ over $\kappa$.  Let $a_i $ be   in the valuation ring of $F$ mapping to $b_i$
and $f(X) = X^\ell + a_{\ell-1}X^{\ell-1} + \cdots + a_1X + a_0 \in F[X]$.
If  $\mu_0 \not\in \kappa$, then we take $\mu_0' = \mu_0$. Since $N_{L_0/\kappa}(\mu_0) = \overline{\theta}$, 
we have $b_0 = (-1)^\ell \overline{\theta}$. In this case we take $a_0 = (-1)^\ell \theta$.
Since $g(X)$ is irreducible in $\kappa[X]$, $f(X) \in F[X]$ is irreducible.
 Let $L = F[X]/(f)$. Then $L/F$  is the unramified extension  with 
residue field $L_0$.  If $\mu_0 \in \kappa$, then $\overline{\theta} = N_{L_0/\kappa}(\mu_0) = \mu_0^\ell$.
Since $F$ is a complete discretely valued field and  $\ell$ is coprime to char$(\kappa)$, 
there exists $\mu \in F$ which is a unit in the valuation ring of $F$ which maps to $\mu_0$
and $\mu^\ell = \theta$. If $\mu_0 \not\in \kappa$, then let $\mu \in L$ be the image of $X$.
Then the image of $\mu$ is $\mu_0$ and $N_{L/F}(\mu) = \theta$.

 Since $L/F$, $E/F$ and $\alpha'$ are   unramified  at the discrete valuation of $F$, 
 we have $\partial_L(  \alpha'   \cdot (\mu \pi^r))= r \overline{\alpha}'  \otimes L_0$
 and  $\partial_L((E, \sigma, \pi)   \cdot (\mu \pi^r)) = \partial_L((E \otimes L, \sigma \otimes 1, \mu^{-1}) 
 \cdot (\pi))  =  (E_0 \otimes L_0, \sigma_0 \otimes 1,  \mu_0^{-1})$.
 Since $\alpha = \alpha' + (E, \sigma, \pi)$, we have
$$
\begin{array}{rcl}
 \partial_L( \alpha \cdot (\mu\pi^r))  & = & \partial_L( (\alpha' \otimes L) \cdot (\mu \pi^r))  + 
  \partial_L((E, \sigma, \pi ) \cdot (\mu\pi^r)) \\
& = &  r\overline{\alpha}' \otimes {L_0} + (E_0 \otimes L_0, \sigma_0 \otimes 1, \mu_0^{-1})   \\
& = & 0
\end{array}
$$
\end{proof}

 \begin{lemma}
 \label{lambda_non_square_local}
   Suppose that $\kappa$ is a local field. Let $\ell$ be  prime not 
 equal to  char$(k)$ and $n =\ell^d $. 
 Let $\alpha \in H^2(F, \mu_n)$ and $\lambda \in F^*$.
 Suppose  $\lambda  \not \in F^{*\ell}$, $\alpha \neq 0$ and $\alpha \cdot (\lambda) = 0$.
 Then ind$(\alpha \otimes F(\sqrt[\ell]{\lambda})) <  $ ind$(\alpha)$ and $\alpha \cdot (\sqrt[\ell]{\lambda}) 
 = 0 \in H^3(F(\sqrt[\ell]{\lambda}), \mu_n^{\otimes 2})$. 
 \end{lemma}
 
\begin{proof} Suppose $\nu(\lambda)$ is coprime to $\ell$. Then, by (\ref{dot-zero}), 
we have ind$(\alpha \otimes F(\sqrt[\ell]{\lambda})) <  $ ind$(\alpha)$ and 
$\alpha \cdot (\sqrt[\ell]{\lambda})  = 0 \in H^3(F(\sqrt[\ell]{\lambda}), \mu_n^{\otimes 2})$.
 
 Suppose that $\nu(\lambda) $ is divisible by $\ell$. Write $\lambda = \theta \pi^{\ell d}$
 with $\theta \in F$ a unit in the valuation ring of $F$.

Write $\alpha = \alpha' + (E, \sigma, \pi)$ as in (\ref{rbc}). Then 
ind$(\alpha) = $ ind$(\alpha' \otimes E) [E : F]$ and
ind$(\alpha \otimes F(\sqrt[\ell]{\theta})) \leq $ ind$(\alpha' \otimes E(\sqrt[\ell]
{\theta})) [E(\sqrt[\ell]{\theta}) : F(\sqrt[\ell]{\theta})]$ (cf. \ref{dvr_ind}).
If $\sqrt[\ell]{\theta} \in E$, then $F(\sqrt[\ell]{\theta}) \subset  E = E(\sqrt[\ell]{\theta})$.
In particular $[E(\sqrt[\ell]{\theta}) : F(\sqrt[\ell]{\theta})] = [E : F(\sqrt[\ell]{\theta})] < [E: F]$.
Since  ind$(\alpha' \otimes E(\sqrt[\ell]{\theta})) \leq $  
ind$(\alpha' \otimes E)$, it follows that  ind$(\alpha \otimes F(\sqrt[\ell]{\theta})) < $ ind$(\alpha)$.

Suppose  $\sqrt[\ell]{\theta} \not\in E$.  Since $E$ is unramified extension of $F$ and
$\theta$ is a unit in the valuation ring of $E$, $E(\sqrt[\ell]{\theta})$ is an unramified extension of 
$F$ with residue field $E_0(\sqrt[\ell]{\overline{\theta}})$, where $E_0$ is the residue field of
$E$ and $\overline{\theta}$ is the image of $\theta$ in the residue field. Since $F$ is a complete
discretely valued field and $\theta$ is not an $\ell^{\rm th}$ power in $E$, 
$\overline{\theta}$ is not an $\ell^{\rm th}$ power in $E_0$ and
$[E_0(\sqrt[\ell]{\overline{\theta}}) : E_0] = \ell$. 

 Suppose $\alpha' \otimes E \neq 0$. Then $\overline{\alpha}' 
  \otimes E_0  \neq 0$. Since $E_0$ is a local field and ind$(\overline{\alpha}')$ is a 
  power of $\ell$,   ind$(\overline{\alpha}' \otimes E_0(\sqrt[\ell]{\overline{\theta}})) < 
  $ ind$(\overline{\alpha}' \otimes E_0)$ (\cite[p. 131]{CFANT}). Hence ind$( \alpha' \otimes E(\sqrt[\ell]{ \theta})) 
  <  $ ind$( \alpha' \otimes E)$ and ind$(\alpha \otimes F(\sqrt[\ell]{\theta})) < $ ind$(\alpha)$. 

Suppose that $\alpha' \otimes E = 0$. Then, by (\ref{dvr_ind_per}),  $\alpha
= (E, \sigma, u\pi)$ for some unit $u$ in the valuation ring of $F$.
Since $\alpha \cdot (\lambda) = 0$, $(E, \sigma, u\pi) \cdot (\lambda ) = 0$.
Since $E/F$ is unramified with residue field $E_0$, $u, \theta$ are units in the 
valuation ring of $F$ and $\pi$ a parameter, by taking the residue of
$\alpha \cdot (\lambda) = 0$, we see that $(E_0, \sigma_0,  \overline{\theta}^{-1} \overline{u}^{\ell d}) = 0 \in H^2(\kappa, \mu_n)$
(cf. \ref{dot-zero}).
In particular, $\overline{\theta} \overline{u}^{-\ell d}$ is a norm from $E_0$. Since $[E_0 : k]$ is a power of $\ell$
and $E_0/\kappa$ is cyclic, there exists a sub extension $L$ of $E_0$ such that $[L : \kappa] = \ell$.
Then $\overline{\theta} \overline{u}^{-\ell d}$ is a norm from $L$ and hence  
  $\overline{\theta}$ is a norm from $L$.
Since $\overline{\theta}$ is not in $\kappa^{*\ell}$, by (\ref{local_field_norm}), $L = \kappa(\sqrt[\ell]{
\overline{\theta}})$. In particular $\sqrt[\ell]{\overline{\theta}} \in E_0$ and hence 
$\sqrt[\ell]{\theta} \in E$.  Thus ind$(\alpha \otimes F(\sqrt[\ell]{\theta})) 
= $ ind$((E, \sigma, u\pi) \otimes F(\sqrt[\ell]{\theta})) < $ ind$(E,\sigma, u \pi) = $ ind$(\alpha)$. 
  \end{proof}

\begin{lemma} 
\label{local-local}
Suppose $\kappa$ is a local field. Let $\ell$ be a prime not equal to 
char$(\kappa)$ and $n = \ell^d$. Let $\alpha \in H^2(F, \mu_n)$ and $\lambda \in F^*$.
Suppose that $\kappa$ contains a primitive $\ell^{\rm th}$ root of unity. 
If $\alpha \neq 0$ and $\alpha \cdot (\lambda) = 0  \in H^3(F,\mu_n^{\otimes 2})$, then 
there exist a cyclic field extension $L/F$ of degree $\ell$ and $\mu \in L^*$ such that 
$N_{L/F}(\mu) = \lambda$, ind$(\alpha \otimes L) < $ ind$(\alpha)$
 and $\alpha \cdot (\mu) = 0 \in H^3(L, \mu_n^{\otimes 2})$.  Further, if $\lambda \in F^{*\ell}$, then 
 $L/F$ is  unramified and $\mu \in F$.  
\end{lemma}

\begin{proof} Suppose $\lambda$ is not an $\ell^{\rm th}$ power in $F$. Let $L = F(\sqrt[\ell]{\lambda})$
and $\mu =\sqrt[\ell]{\lambda}$. Then, by (\ref{lambda_non_square_local}),  ind$(\alpha \otimes L) < $ ind$(\alpha)$
and  $\alpha \cdot (\mu) = 0 \in H^3(L, \mu_n^{\otimes 2})$.

Suppose $\lambda = \mu^\ell$ for some $\mu \in F^*$. 
Write $\alpha = \alpha' + (E, \sigma, \pi)$
as in (\ref{rbc}). 

Suppose that $\alpha' \otimes E = 0$. 
Then, by (\ref{dvr_ind_per}), 
$\alpha = (E, \sigma, u \pi)$ for some $u \in F^*$ which is a unit in the valuation ring of $F$.
Let $L$ be the unique subfield of $E$ with $L/F$ of degree $\ell$. 
Then ind$(\alpha \otimes L) < $ ind$(\alpha)$. 
Since cor$_{L/F}(\alpha \cdot (\mu) ) 
= \alpha \cdot (\mu^{\ell}) = \alpha \cdot (\lambda)  = 0$, by (\ref{corestriction}), 
$\alpha \cdot (\mu) = 0$ in $H^3(L, \mu_n)$.  We also have  $\lambda = \mu^\ell = N_{L/F}(\mu)$.

Suppose that $\alpha' \otimes E \neq 0$.
Let $E_0$ be the residue field of $E$.  Then $E_0 / \kappa$ is a cyclic field extension of $\kappa$
of degree 
equal to the degree of $E/F$. Since $\kappa$ is a local field and contains a primitive $\ell^{\rm th}$ root of
unity, there are at  least three distinct  cyclic field extensions of $\kappa$ of degree $\ell$. Since $E_0/\kappa$ 
is a cyclic extension, there is at most  one sub extension of $E_0$ of degree $\ell$ over $\kappa$.
Thus there exists a cyclic field extension $L_0/\kappa$ of degree $\ell$ such that 
$E_0 \otimes L_0$ is a field. Let $L/F$ be the unramified extension  with residue field $L_0$.
Then $E \otimes L$ is a field. 
Let $\overline{\alpha}'$ be the image of $\alpha'$ in $H^2(\kappa, \mu_n)$. Since $E$ is a complete discretely 
valued field,   $\overline{\alpha}' \otimes E_0 \neq 0$. 
Since   $E_0 \otimes L_0/E_0$ is a field extension of degree $\ell$ and  
$\kappa$ is a  local field,  ind$(\overline{\alpha}' \otimes E_0 \otimes L_0) < $ ind$(
\overline{\alpha}' \otimes E_0)$ (\cite[p. 131]{CFANT}).  Since $E$ is a complete discretely valued field,
ind$(\alpha' \otimes E \otimes L) < $ ind$(
 \alpha' \otimes E)$. Since $L/F$ is unramified,  $\partial(\alpha \otimes L) = \partial(\alpha) \otimes L_0$
 (cf. \cite[Proposition 3.3.1]{CT-Santa-Barbara}) 
 and hence  the decomposition $\alpha \otimes L = \alpha' \otimes L + (E \otimes L, \sigma \otimes 1, \pi)$ is 
 as in (\ref{rbc}).
Thus, by (\ref{dvr_ind}), ind$(\alpha \otimes L) < $ ind$(\alpha)$.  As above, we also have 
 $\lambda = N_{L/F}(\mu)$ and $\alpha \cdot (\mu) = 0 \in H^3(L, \mu_n^{\otimes 2})$. 
\end{proof}

\begin{lemma} 
\label{local-global}
Suppose $\kappa$ is a global  field. Let $\ell$ be a prime not equal to 
char$(\kappa)$ and $n = \ell^d$. Let $\alpha \in H^2(F, \mu_n)$ and $\lambda \in F^*$.
If $\alpha \neq 0$ and $\alpha \cdot (\lambda) = 0  \in H^3(F,\mu_n^{\otimes 2})$, then 
there exist a field extension $L/F$ of degree $\ell$ and $\mu \in L^*$ such that 
$N_{L/F}(\mu) = \lambda$, ind$(\alpha \otimes L) < $ ind$(\alpha)$
 and $\alpha \cdot (\mu) = 0 \in H^3(L, \mu_n^{\otimes 2})$. 
\end{lemma}

\begin{proof} Suppose that $\nu(\lambda)$ is coprime to $\ell$. Then, by (\ref{dot-zero}), 
 $L = F(\sqrt[\ell]{\lambda})$ and $\mu = \sqrt[\ell]{\lambda}$
has the required properties.

 Suppose that $\nu(\lambda)$ is divisible by $\ell$. Let $\pi$ be a parameter in $F$.
Then   $\lambda = \theta \pi^{r\ell}$ with $\nu(\theta) = 0$. 
 Write $\alpha = \alpha' + (E, \sigma, \pi)$ as in (\ref{rbc}). 
Let $\overline{\alpha}'$ be the image of 
$\alpha'$  in $H^2(\kappa, \mu_n)$  and $\theta_0$ the image of $\theta$ in $\kappa$.
Since $\alpha \cdot (\lambda)= 0$, by (\ref{dot-zero}), we have 
$r\ell \overline{\alpha}' = (E_0, \sigma_0, \theta_0)$, where $E_0$ is the residue field of $E$ and 
$\sigma_0$ induced by $\sigma$.
 
 Suppose that $r\overline{\alpha}' \otimes E_0 \neq 0$.
Then, by (\ref{global_field_non_zero}) , there exist a  extension $L_0/\kappa$
of degree $\ell$ and $\mu_0 \in L_0$ such that $N_{L_0/\kappa}(\mu_0) = \theta_0$,
ind$(\overline{\alpha}' \otimes E_0 \otimes L_0) < $ ind$(\overline{\alpha}' \otimes E_0)$
 and $r\overline{\alpha}' \otimes L_0 = (E_0\otimes L_0, \sigma_0, \mu_0)$.

 Suppose that  $r \overline{\alpha}' \otimes E_0 = 0$. Suppose that  $E_0 \neq \kappa$.
 Let $L_0$ be the unique subfield field of $E_0$ of degree $\ell$ over $\kappa$. 
Then, by (\ref{global_field_zero}),  there exists $\mu_0 \in L_0$
such that $N_{L_0/\kappa}(\mu_0) = \theta_0$ and $r\overline{\alpha}' \otimes L_0
= (E_0, \sigma_0, \mu_0)$.
 
 Suppose that   $E_0 =  \kappa$.  Then, by (\ref{globalfield_unramified}), there exist
 a field extension $L_0/\kappa$ of degree $\ell$ and $\mu_0 \in L_0$ such that 
 $N_{L_0/k}(\mu_0) = \theta_0$ and ind$(\overline{\alpha}' \otimes L_0) < $ ind$(\overline{\alpha}')$.

By (\ref{type-other}), we have the required $L$ and $\mu$. 
\end{proof}

\begin{theorem}
\label{thm_dvr} Let $F$  be a complete discrete valued field with residue  field $\kappa$.
Suppose that $\kappa$ is a local field or a global field. Let $D$ be a central simple algebra 
over $F$ of period $n$. Suppose that $n$ is coprime to char$(\kappa)$. Let $\alpha \in H^2(F,\mu_n)$ be the 
class of $D$ and  $\lambda \in F^*$.
If $\alpha \cdot (\lambda) = 0 \in H^3(F, \mu_n^{\otimes 2})$, then $\lambda$ is a reduced norm 
from $D$. 
\end{theorem}

\begin{proof}  Write $n = \ell_1^{d_1} \cdots \ell_r^{d_r}$, $\ell_i$ distinct  primes, $d_i > 0$, 
  $D = D_1 \otimes \cdots \otimes  D_r$ with each $D_i$ a central simple algebra 
over $F$ of period  power of $\ell_i$ (\cite[Ch. V, Theorem 18]{Albert}). 
 Let $\alpha_i$ be the corresponding cohomology class of 
$D_i$. Since  $\ell_i$'s are distinct primes, $\alpha \cdot (\lambda) = 0$ if and only if $\alpha_i \cdot (\lambda) = 0$ 
and $\lambda$ is a 
reduced norm from $D$ if and only if $\lambda$ is a reduced norm from each $D_i$.  Thus without loss of 
generality we assume that per$(D) = \ell^d$ for some prime $\ell$. 

We prove the theorem by the induction on the index of $D$. Suppose that deg$(D) = 1$.
Then every element of $F^*$ is a reduced norm from $D$. We assume that deg$(D) = n = \ell^d \geq 2$. 

Let $\lambda \in F^*$ with $\alpha \cdot (\lambda) = 0 \in H^3(F, \mu_n^{\otimes 2})$.
Let $\rho$ be a primitive $\ell^{\rm th}$ root of unity. Since $[F(\rho) :  F]$ is coprime to $n$, 
$\lambda$ is a reduced norm from $F$ is and only if $\lambda$ is a reduced from $D\otimes F(\rho)$.
Thus, replacing $F$ by $F(\rho)$, we assume that $\rho \in F$. 

Since $\kappa$ is either a local field or a global field, by (\ref{local-local}, \ref{local-global}), 
there exist an extension $L/F$ of degree $\ell$ and $\mu \in L^*$ such that $N_{L/F}(\mu) = 
\lambda$, $\alpha \cdot (\mu) = 0$ and  ind$(\alpha \otimes L) < $ ind$(\alpha)$. 
Thus, by induction, $\mu$ is a reduced norm from $D \otimes L$. 
Since $N_{L/F}(\mu)  = \lambda$, $\lambda$ is a reduced norm from $D$.
\end{proof}

The following technical lemma is used in \S 6. 

\begin{lemma}
\label{dropping_ind_dvr}
 Let $\kappa$ be a finite field and $K$ a function field of a curve over $\kappa$.
Let $u, v, w \in \kappa^*$ and $\theta \in K^*$. Let $\ell$ a prime not equal to char$(\kappa)$ and
 $\theta = wu\lambda$.  If $\kappa$ contains a primitive $\ell^{\rm th}$ root 
 of unity and $w \not\in \kappa^{*\ell}$,  
 then for $ r\geq 1$,    the  element  $(  v,  \sqrt[\ell^r]{\theta})_\ell$ is $H^2(K(\sqrt[\ell^e]{\theta})$
 is trivial  over $K(\sqrt[\ell^r]{\theta}, \sqrt[\ell]{v + u\lambda})$.

\end{lemma} 

\begin{proof}
Let $L = K(\sqrt[\ell^r]{\theta}, \sqrt[\ell]{v + u\lambda})$ and $\beta = (v, \sqrt[\ell^r]{\theta})_\ell$. Since $L$ is a global field, to show that 
$\beta \otimes L$ is trivial, it is enough to show that $\beta \otimes L_\nu$ is trivial for every discrete valuation $\nu$ of 
$L$. Let $\nu$ be a discrete valuation of $L$.  Since $v \in \kappa^*$, $v$ is a unit at $\nu$. If $\theta$ is a unit at 
$\nu$, then $\beta \otimes L$ is unramified at $\nu$ and hence $\beta \otimes L_\nu$ is trivial. Suppose that 
$\theta$ is  not a unit at $\nu$. Since $u$ and $v$ are units at $\nu$, $\lambda$ is not a unit. 
Suppose that $\nu(\lambda) > 0$.  Then $v \in L_\nu^{*\ell}$ and hence $\beta \otimes L_\nu$ is trivial.
Suppose that $\nu(\lambda) < 0$. Then $ \sqrt[\ell]{u\lambda} \in L_\nu$. Since $r \geq 1$, $\theta = uw\lambda$ and 
$\sqrt[\ell^r]{\theta} \in L_\nu$, we have $\sqrt[\ell]{\theta}  = \sqrt[\ell]{wu\lambda} \in L_\nu$. Hence $\sqrt[\ell]{w} \in L_\nu$.
Since $w\in \kappa^* \setminus \kappa^{*\ell}$, $v \in \kappa^*$ and $\kappa$ is a finite field, $\sqrt[\ell]{v} \in \kappa(\sqrt[\ell]{w})$.
Since $\kappa(\sqrt[\ell]{w}) \subset L_\nu$, $\beta \otimes L_\nu$ is trivial.
\end{proof}

We end this section with the following well known facts.  

\begin{lemma}
\label{ur_ext}
 Let $F$ be a complete discrete valued field with  the residue field  $\kappa$.
Let $\alpha \in \Br(F)$ and $L/F$ an unramified extension with residue field $L_0$. 
Suppose that   per$(\alpha)$ is coprime to char$(\kappa)$. Let  $\partial(\alpha) = (E_0, \sigma_0)$.
If $\partial(\alpha \otimes L)$ is trivial,  then 
$E_0$ is isomorphic to a subfield of  $L_0$. 
\end{lemma}

\begin{proof} Let $L_0$ be the residue field of $L$. 
Since $L/F$ is unramified, $ (E_0, \sigma_0) \otimes L_0 = \partial(\alpha) \otimes L_0 = 
\partial(\alpha \otimes L)$ (cf. \cite[Proposition 3.3.1]{CT-Santa-Barbara}).
Since $\alpha \otimes L = 0$, $\partial(\alpha \otimes L) = 0$ and hence $E_0$ is isomorphic to a subfield of 
$L_0$.
\end{proof}

\begin{cor}
\label{residue}
 Let  $L/F$ be a cyclic extension of degree $n$,
$\tau$ a generator of Gal$(L/F)$ and $\theta \in F^*$. If $\nu(\theta)$ is coprime to $n$ and 
ind$(L/F, \tau, \theta) =  [ L : F]$,
then  $[L : F ] = $ per$(\partial(L/F, \tau, \theta))$.
\end{cor}

\begin{proof} Let $\beta = (L/F, \tau, \theta)$ and $m =  $ per$(\partial(\beta))$. 
Since $n =[L : F] = $ ind$(\beta)$,   $m$ divides $n$. Since 
$\nu(\theta)$ is coprime to $n$,  $F(\sqrt[m]{\theta})/F$  is a totally ramified extension of
degree $m$ with residue field equal to the residue field  $\kappa$ of $F$. 
Since $\partial(\beta \otimes F(\sqrt[m]{\theta})) =  m \partial(\beta)$, $\beta \otimes F(\sqrt[m]{\theta})$
is unramified. Since  $F(\sqrt[n]{\theta})/F(\sqrt[m]{\theta})$ is totally ramified and 
 $\beta \otimes F(\sqrt[n]{\theta})$ is trivial,  $\beta  \otimes F(\sqrt[m]{\theta})$ is trivial (cf. \ref{not_type_6}).  Hence 
$n = m$.
\end{proof}

\section{Brauer group -  complete two dimensional  regular local rings}

Through out this section  $A$ denotes   a complete regular local ring of dimension 2 with residue
field $\kappa$ and $F$ its  field of fractions. Let $\ell $ be a prime not
equal   to the characteristic of $\kappa$ and $n  = \ell^d$ for some $d \geq
 1$.  Let $m = (\pi, \delta)$ be the maximal ideal of $A$.  For any prime $p \in A$,
 let $F_p$ be the completion of the field of fractions of the completion of 
 the local ring $A_{(p)}$ at $p$ and $\kappa(p)$ the residue field at $p$.

\begin{lemma} 
\label{lifting_extension} 
  Let $E_\pi$ be a Galois   extension of $F_\pi$ of 
degree coprime to char$(\kappa)$. Then there exists a Galois  
extension $E$ of $F$  of degree $[E_\pi : F_\pi]$  which is 
unramified on $A$ except possibly at $\delta$ and Gal$(E/F) \simeq $Gal$(E_\pi/F_\pi)$. 
\end{lemma}

\begin{proof}  Since  $A$ is complete and 
 $ m = (\pi, \delta)$, $\kappa(\pi)$ is a complete discretely 
 valued field with residue field $\kappa$ and  
 the image $\overline{\delta}$ of $\delta$  as a parameter.  
 Let $E_0$ be the residue field of $E_\pi$.
 Then $E_0/\kappa(\pi)$ is a Galois extension with Gal$(E_0/\kappa(\pi)) 
 \simeq $ Gal$(E_\pi/F_\pi)$.
 Let  $L_0$ be   the maximal unramified extension of $\kappa(\pi)$
contained in $E_0$.  Then $L_0$ is also a complete discretely valued field
with $\overline{\delta}$ as  a parameter.
Since $E_0/L_0$ is a totally ramified extension 
of degree coprime to char$(\kappa)$, we have $E_0 = L_0(\sqrt[e]{
v\overline{\delta}})$ for some  $v \in L_0$ which is a unit at the discrete valuation of
$L_0$ (cf. \ref{ramified_extension}).  

Since $E_0/ \kappa(\pi)$ is a Galois  extension, $E_0/L_0$ and $L_0/\kappa(\pi)$
are Galois  extensions.  Let $\kappa_0$ be the residue field of $E_0$. Then the residue field of $L_0$ 
is also $\kappa_0$. Since $\kappa_0$ is a Galois   extension of $\kappa$ and
$A$ is complete, there exists a Galois  extension $L$ of $F$ which is unramified
on $A$ with residue field $\kappa_0$. Let $B$ be the integral closure of $A$ in 
$L$.  Then $B$ is a  regular local ring with residue field $\kappa_0$ (cf. \cite[Lemma 3.1]{PS1}).  
Let $u \in B$ be a lift of  $v$. 

Let $E = L(\sqrt[e]{u\delta})$.  
Since $L/F$ is unramified on $A$, $E/F$ is unramified on $A$ except possibly  at $\delta$. 
In particular $E/F$ is unramified at $\pi$ with residue field $E_0$. By the construction 
$[E : F] = [E_0  : \kappa(\pi)]$.  Hence $E \otimes F_\pi \simeq E_\pi$.

Since $L/F$  is a Galois extension which is unramified  at $\pi$,  we have 
Gal$(L/F) \simeq $ Gal$(L_0/\kappa(\pi))$.  Let $\tau \in $ Gal$(L/F)$ and 
 $\overline{\tau}  \in $Gal$(L_0/\kappa(\pi))$ be the image of $\tau$.  Since $E_0/\kappa(\pi)$
is Galois and $E_0 = L_0(\sqrt[e]{v\overline{\delta}})$,  by (\ref{galois_extension}),  $E_0$ 
contains a primitive $e^{\rm th}$ root of unity $\rho$ and 
$ \overline{\tau}(v\overline{\delta}))  \in E_0^e$.  In particular  $\rho \in \kappa_0$. 
Since $B$ is complete with residue field $\kappa_0$, $\rho \in B$ and hence $\rho \in L \subseteq E$. 
Since $\overline{\tau}(v\overline{\delta}) =\overline{\tau}(v) \overline{\delta}$ and 
$v\overline{\delta}$, $\overline{\tau}(v\overline{\delta}) \in E_0^e$,  $\overline{\tau}(v)/v \in E_0^e$.
Since $\overline{\tau}(v)$ and $v$ are units at the discrete valuation of $L_0$ and $E_0/L_0$ is 
totally ramified, $\overline{\tau}(v)/v \in L_0^e$. 
Since $B$ is complete and the image of  $\tau(u)/u$ in $L_0$ is $\overline{\tau}(v)/v$,
$\tau(u)/u \in   L^e$. 
Since $E = L(\sqrt{u\delta})$,  $\tau(u\delta) \in E^e$.  
Thus, by (\ref{galois_extension}),  $E/F$ is  Galois.  
 Since $E \otimes F_\pi \simeq E_\pi$, 
   Gal$(E/F) \simeq $Gal$(E_\pi/F_\pi)$. 
\end{proof}

Since  $A$ is complete and $(\pi, \delta)$ is the maximal ideal of $A$, $A/(\pi)$ is a complete 
discrete valuation ring with $\overline{\delta}$ is a parameter and $A/(\delta)$ is a complete 
discrete valuation ring with $\overline{\pi}$.  The following follows from (\cite[Proposition 1.7]{Ka}).

\begin{lemma}(\cite[Proposition 1.7]{Ka})
\label{2dimcomplex} 
Let  $m \geq 1$ and $\alpha \in H^m(F, \mu_n^{\otimes (m-1)})$. Suppose that $\alpha$ is unramified on $A$ except 
possibly at $\pi$ and $\delta$. Then $$\partial_{\overline{\delta}}(\partial_\pi(\alpha)) = -
\partial_{\overline{\pi}}(\partial_\delta(\alpha)).$$ 
\end{lemma} 

Let $H^m_{nr}(F, \mu_n^{\otimes (m-1)})$  be the intersections of the kernels of the residue 
homomorphisms $\partial_\theta : H^m(F, \mu_n^{\otimes (m-1)}) \to H^{m-1}(\kappa(\theta), \mu_n^{\otimes (m-2)})$ 
for all primes $\theta \in A$.  The following lemma follows from the purity theorem of Gabber. 

\begin{lemma}
\label{purity}  For $m \geq 1$,  $H^m_{nr}(F, \mu_n^{\otimes (m-1)}) \simeq H^m(\kappa, \mu_n^{\otimes (m-1)})$.
\end{lemma}

\begin{proof} By the purity theorem of Gabber (cf. \cite[CH. XVI]{Riou}), we have 
$H^m_{nr}(F, \mu_n^{\otimes (m-1)}) \simeq H^m_{\acute{e}t}(A, \mu_n^{\otimes (m-1)})$.
Since $A$ is complete, we have $H^m_{\acute{e}t}(A, \mu_n^{\otimes (m-1)}) \simeq 
H^m(\kappa, \mu_n^{\otimes (m-1)})$ (cf. \cite[Corollary 2.7, p.224]{Milne}).
\end{proof}

\begin{lemma} 
\label{saltman1}
Let   $ m \geq 1$   and $\alpha \in H^m(F, \mu_n^{\otimes (m-1)})$.
Suppose that $\alpha$ is unramified except possibly at $\pi$.
Then there exist $\alpha_0 \in H^m(F,\mu_n^{\otimes (m-1)})$ and $\beta \in H^{{m-1}}(F, \mu_n^{\otimes (m-2)})$ 
which are unramified on $A$ such that 
$$\alpha = \alpha_0 + \beta \cdot (\pi).$$
\end{lemma}

\begin{proof}  
 Let $\beta_0 = \partial_\pi(\alpha)$. 
 By (\ref{2dimcomplex}), $\beta_0 \in H^{m-1}(\kappa(\pi), \mu_n^{ \otimes (m-2)})$  is unramified on $A/(\pi)$.
 Since $A/(\pi)$ is a complete discrete valuation ring with residue field $\kappa$, 
 we have $H^{m-1}_{nr}(\kappa(\pi), \mu_n^{\otimes (m-2)}) \simeq H^{m-1}(\kappa, \mu_n^{\otimes (m-2)})$.
 Since $A$ is a complete regular local ring of dimension 2, $H^{m-1}_{nr}(F, \mu_n^{\otimes (m-2)}) \simeq 
 H^{m-1}(\kappa, \mu_n^{\otimes (m-2)})$ (\ref{purity}). 
 Thus, there exists $\beta \in H^{m-1}_{nr}(F, \mu^{\otimes (m-1)})$
 which is a lift of $\beta_0$.
 Then $\alpha_0 = \alpha - \beta  \cdot (\pi)$ is unramified on $A$.
Hence $\alpha = \alpha_0 + \beta \cdot (\pi)$.
\end{proof}

\begin{cor} 
\label{zero}
Let $m  \geq 1$  and   $\alpha \in H^m(F, \mu_n^{\otimes (m-1)})$ is 
unramified on $A$ except possibly at $\pi$ and $\delta$.  If 
$\alpha \otimes F_\delta = 0$, then $\alpha = 0$.  In particular if $\alpha_1, \alpha_2 
\in H^m(F, \mu_n^{\otimes (m-1)})$ unramified on $A$ except possibly at $\pi$ and $\delta$ and   
$\alpha_1 \otimes F_\delta = \alpha_2 \otimes F_\delta$, then $\alpha_1  = \alpha_2$.  
\end{cor}

\begin{proof} Since $\alpha \otimes F_\delta = 0$, $\alpha$ is unramified at $\delta$.
Thus $\alpha$ is unramified on $A$ except possibly at $\pi$.  By (\ref{saltman1}),
we have $\alpha = \alpha_0 + \beta \cdot (\pi)$ for some $\alpha_0 \in H^m(F, \mu_n^{\otimes (m-1)})$
and $\beta \in H^{m-1}(F, \mu_n^{\otimes (m-2)})$ which are unramified on $A$.
Since $\alpha \otimes F_\delta = 0$, we have $(\beta \cdot (\pi)) \otimes F_\delta  
= - \alpha_0 \otimes F_\delta$. Since $\beta \cdot (\pi)$ and $\alpha_0$ are unramified at $\delta$,
we have $\overline{\beta} \cdot (\overline{\pi}) = - \overline{\alpha}_0$, where $\overline{~}$ denotes the
image over $\kappa(\delta)$. Since $\kappa(\delta)$ is a complete discrete valued field with 
$\overline{\pi}$ as  a parameter, by taking the residues, we see that the image of $\beta$ is 0 
in  $H^{m-1}(\kappa, \mu_n^{\otimes (m-2)})$. Since $A$ is a complete regular local ring, $\beta = 0$
(\ref{purity}).  Hence $\alpha = \alpha_0$ is unramified on $A$.  Since  $\alpha \otimes F_\delta = 0$, 
$\overline{\alpha} = 0 \in H^m(\kappa(\delta), \mu_n^{\otimes (m-1)})$.
 In particular the image of $\alpha$   in $H^m(\kappa, \mu_n^{\otimes (m-1)})$ is zero.
Since $A$ is a complete regular local ring, $\alpha = 0$ (\ref{purity}).
\end{proof}

\begin{cor} 
\label{local-period}
Let $ m\geq 1$ and $\alpha \in H^m(F, \mu_n^{m-1})$. If $\alpha$ is unramified
on $A$ except possibly at $\pi$ and $\delta$, 
then per$(\alpha) = $ per$(\alpha \otimes F_\pi) = $ per$(\alpha \otimes F_\delta)$.
\end{cor}

\begin{proof} Suppose $t = $ per$(\alpha \otimes F_\delta)$.
Then $t \alpha \otimes F_\delta = 0$ and hence, by (\ref{zero}), $t\alpha = 0$.
Since per$(\alpha \otimes F_\delta) \leq $ per$(\alpha)$, it follows
that per$(\alpha) = $ per$(\alpha \otimes F_\delta)$. Similarly,
  per$(\alpha) = $ per$(\alpha \otimes F_\pi)$. 
\end{proof}

\begin{cor}
\label{local-cyclic} Suppose that $\kappa$ is a finite field. 
Let $\alpha \in H^2(F, \mu_n)$. If $\alpha$ is unramified except at $\pi$ and $\delta$, then 
there exist a cyclic extension $E/F$ and $\sigma \in {\rm Gal}(E/F)$ a generator,  
 $u \in A$ a unit, and $0 \leq i, j < n $ such that 
 $\alpha = (E, \sigma, u\pi^i\delta^j)$  with 
 $E/F$ is unramified  on $A$ except at $\delta$ and $i = 1$ or 
 $E/F$ is unramified on $A$ except at $\pi$ and  $j = 1$.
\end{cor}
  
 \begin{proof}  
 Since $n$ is a power of  the prime $\ell$ and $n \alpha = 0$,  
   per$(\partial_\pi(\alpha))$ and per$(\partial_\delta(\alpha))$ are   powers  of $\ell$.
  Let $d' $ be the maximum of per$(\partial_\pi(\alpha))$ and
 per$(\partial_\delta(\alpha))$. Then   
 $\partial_\pi(d'\alpha) =  d'\partial_\pi(\alpha) = 0$ and 
$\partial_\delta(d'\alpha) = d'\partial_\delta(\alpha) = 0$.
In particular $d'\alpha$ is unramified on $A$. Since $\kappa$ is a finite field,  $d'\alpha = 0$.
Hence  per$(\alpha)$ divides $d'$ and $d' = $ per$(\alpha)$. 
Thus per$(\alpha) =  $ per$(\partial_\pi(\alpha))$ or per$(\partial_\delta(\alpha))$.

Suppose that   per$(\alpha) =  $ per$(\partial_\pi(\alpha))$.
Since 
$\partial_\pi(\alpha \otimes F_\pi) = \partial_\pi(\alpha)$,  we have 
per$(\partial_\pi(\alpha)) \leq$  per$(\alpha \otimes F_\pi)  \leq $ per$(\alpha)$.
Thus  per$(\alpha \otimes F_\pi) =$ per$(\partial_\pi(\alpha \otimes F_\pi))$.
Thus,  by (\ref{dvr_ind_per}), we have 
$\alpha \otimes F_\pi = (E_\pi/F_\pi, \sigma, \theta \pi)$ for some 
cyclic unramified extension $E_\pi/F_\pi$ and 
$\theta \in F_\pi$ a unit in the valuation ring of   $F_\pi$. 

By (\ref{lifting_extension}), there exists a  Galois extension  $E/F$ 
which is unramified on $A$ except possibly at 
$(\delta)$ such that $E \otimes F_\pi \simeq E_\pi$.
Since $E_\pi/F_\pi$ is cyclic, $E/F$ is cyclic.
 Since $\theta \in F_\pi$ is a unit in the valuation ring of $F_\pi$ and 
the residue field of $F_\pi$ is a complete discrete valued field with $\overline{\delta}$ 
as parameter, 
we can write $\theta = u \delta^j \theta_1^n$ for some unit $u \in A$, $\theta_1 \in F_\pi$ and 
$0 \leq j \leq n-1$.
Then $\alpha \otimes F_\pi \simeq   (E, \sigma, u\delta^j\pi) \otimes F_\pi$.
 Thus, by (\ref{zero}), we have $\alpha = (E, \sigma, u\delta^i\pi)$.
 
 If per$(\alpha) =  $ per$(\partial_\delta(\alpha))$, then, as above, we get
 $\alpha = (E, \sigma, u\pi^j\delta)$ for some cyclic extension $E/F$ which is unramified
 on $A$ except possibly at $\pi$.
 \end{proof}

The following is proved  in (\cite[2.4]{RS}) under the assumption that F contains a primitive
$n^{\rm th}$  root of unity.

\begin{prop}
\label{local_index} Suppose that $\kappa$ is a finite field.
Let $\alpha \in H^2(F, \mu_n)$. If $\alpha$ is unramified on $A$ except possibly at
$(\pi)$ and $(\delta)$. Then  ind$(\alpha) = $ ind$(\alpha \otimes F_\pi)
= $ ind$(\alpha \otimes F_\delta)$. 
\end{prop}

\begin{proof}  Suppose that $\alpha$ is unramified on $A$ except possibly at
$(\pi)$ and $(\delta)$.  Then, by (\ref{local-cyclic}),  we assume without loss of generality  
that $\alpha = (E/F, \sigma, \pi\delta^j)$  with $E/F$ unramified on $A$ 
except possibly at $\delta$.  Then ind$(\alpha) \leq [E : F]$.
Since $E/F$ is unramified on $A$ expect possibly at $\delta$, we have 
$[E : F] = [E_\pi : F_\pi]$ and ind$(\alpha \otimes F_\pi) = [E_\pi : F_\pi]$.
Thus  $[E : F] = [E_\pi : F_\pi]  = $ ind$(\alpha \otimes F_\pi) \leq $ ind$(\alpha) \leq [ E : F]$
and hence $[E : F] = $ ind$(\alpha \otimes F_\pi) = $ ind$(\alpha)$.

 \end{proof}

\begin{cor}
\label{period=index} Suppose that $\kappa$ is a finite field. 
Let $\alpha \in H^2(F, \mu_n)$. If $\alpha$ is unramified on $A$ except possibly at
$(\pi)$ and $(\delta)$. Then  ind$(\alpha) = $ per$(\alpha)$. 
\end{cor}

\begin{proof} 
By (\ref{local-period}), per$(\alpha) = $ per$(\alpha \otimes F_\pi) $ 
 and by (\ref{local_local_period_index}),  ind$(\alpha \otimes F_\pi) = $ per$(\alpha \otimes F_\pi)$.
 Thus   per$(\alpha) = $ ind$(\alpha \otimes F_\pi)$. By (\ref{local_index}), we have
 ind$(\alpha) = $ per$(\alpha)$.
 \end{proof}

 The following follows from  (\cite{HHK3} and \cite{HHK5}).
 
 \begin{prop} 
 \label{hhk}
 Let  $\alpha \in H^2(F, \mu_n)$.  Let 
 $ \phi : \XX \to {\rm Spec}(A)$ be a sequence of blow-ups and 
 $V = \phi^{-1}(m)$.  Then ind$(\alpha) = l.c.m\{ {\rm ind}(\alpha \otimes F_x) \mid x \in V \}$. 
 \end{prop}
 
 \begin{proof}  Follows from  similar arguments as in the proof of  (\cite[Theorem 9.11]{HHK3})
 and using (\cite[Theorem 4.2.1]{HHK5}).
 \end{proof}

 We end this section with the following well known result

\begin{lemma}
\label{extns_2_local} 
Let $E/F$ be a cyclic extension of degree $\ell^d$ for some $d \geq 1$. 
  If $E/F$ is unramified on $A$ except possibly at $\delta$, then there exist a subextension $E_{nr}$ of $E/F$ and 
  $w \in E_{nr}$ which is a unit  in the integral closure of $A$ in $E_{nr}$ such that $E_{nr}/F$ is unramified on $A$
  and  $E= E_{nr}(\sqrt[\ell^e]{w\delta})$. Further if   $\kappa$ is a  finite  field, $\kappa$ contains a primitive 
  $\ell^{\rm th}$ root of unity and $0 < e < d$, then  $N_{E_{nr}/F}(w) \in A$ is not an $\ell^{\rm th}$ power 
in $A$.
\end{lemma}

\begin{proof}  Let $E(\pi)$ be the residue field of $E$ at $\pi$. Since $E/F$ is unramified at 
$A$ except possibly at $\delta$, by (\ref{local-period}), $[E(\pi) : \kappa(\pi)] = [E : F]$. Since 
$E/F$ is cyclic, $E(\pi)/\kappa(\pi)$ is cyclic.  As in the proof of  (\ref{lifting_extension}), there exist a cyclic  extension 
$E_0/F$ unramified on $A$ and a unit $w$ in the integral closure of $A$ in $E_0$ such that the residue field  of
$E_0(\sqrt[\ell^e]{w\delta})$  at $\pi$ is $E(\pi)$. 
By (\ref{zero}), we have $E \simeq E_0(\sqrt[\ell^e]{w\delta})$.  Let $E_{nr} = E_0$. Then $E_{nr}$ has the required properties.

Suppose that $\kappa$ is a  finite field and  contains a primitive 
  $\ell^{\rm th}$ root of unity.  Let $B$ be the integral closure of $A$ in $E_{nr}$. 
Then $B$ is a complete regular local ring with residue field $\kappa'$ a finite extension of $\kappa$.

Let $w_0 = N_{E_{nr}/F}(w) \in A^*$ and $\overline{w}_0 \in \kappa^*$. 
Suppose that $w_0 \in A^{*\ell}$.  Then $\overline{w}_0 \in \kappa^{*\ell}$.
Since $\kappa$ contains a primitive 
  $\ell^{\rm th}$ root of unity, we have $\mid\!\! \kappa'^*/\kappa'^{*\ell}\! \!\mid  = \mid\!\! \kappa^*/\kappa^{*\ell} \! \!\mid = \ell$.
  Since norm map is surjective from $\kappa'$ to $\kappa$, the norm map induces an isomorphism  from 
  $ \kappa'^*/\kappa'^{*\ell}   \to  \kappa^*/\kappa^{*\ell}$. 
Thus the image of $w$ in $\kappa'$ is an $\ell^{\rm th}$ power. 
Since  $B$ is a  complete regular local 
ring,  $w \in B^{*\ell}$. 
Suppose $0 < e < f$. Then $\sqrt[\ell]{\delta} \in E$. Since $E_{nr}/F$ is nontrivial unramified extension and 
$F(\sqrt[\ell]{\delta})/F$ is a nontrivial  totally ramified extension of $F$,  we have two distinct 
degree $\ell$ subextensions of $E/F$, which is a contradiction to the fact that $E/F$ is cyclic.
Hence $w_0 \not\in A^{*\ell}$.
\end{proof}

   \section{Reduced norms -  complete two dimensional  regular local rings}

Throughout this section we fix  the following notation: \\ 
$\bullet$ $A$ a complete two dimensional regular local ring \\
$\bullet$ $F$ the field of fractions of $A$ \\
$\bullet$ $m = (\pi, \delta)$ the maximal ideal of $A$ \\
$\bullet$ $\kappa = A/m$  a finite field\\
$\bullet$ $\ell$ a prime not equal to char$(\kappa)$ \\
$\bullet$ $n = \ell^d$ \\
$\bullet$ $\alpha \in H^2(F, \mu_n)$ is unramified on $A$ except possibly at $(\pi)$ and $(\delta)$\\
$\bullet$ $\lambda = w\pi^s\delta^t$, $w \in A$ a unit and $s, t \in \Z$ with $1 \leq s, t < n$.

  The aim of  this section is to prove that 
  if  $\alpha \neq 0$ and $\alpha \cdot (\lambda) = 0$, 
  then  there exist an extension $L/F$ of degree $\ell$
  and $\mu \in L$ such that  ind$(\alpha \otimes L) < $ ind$(\alpha)$
  and $N_{L/F}(\mu) = \lambda$.  We assume that \\
 $\bullet$ $F$ contains a primitive $\ell^{\rm th}$ root of unity. \\

  We begin with the following
  
\begin{lemma}
\label{local_coprime}
  If $\alpha \cdot (\lambda) = 0$, then $s\alpha = (E, \sigma, \lambda)  $ for some cyclic 
 extension $E$ of $F$ which is unramified on $A$ except possibly at  $\delta$.
 In particular, if $s$  is coprime to $\ell$, then $\alpha = 
 (E', \sigma', \lambda)$ for  some cyclic 
 extension $E'$ of $F$ which is unramified on $A$ except possibly at  $\delta$.
\end{lemma}

\begin{proof}  By (\ref{dot-zero}), there exists an unramified cyclic
extension $E_\pi$ of $F_\pi$ such that $s\alpha \otimes F_\pi = (E_\pi, \sigma, \lambda)$.
Let $E(\pi)$ be the residue field of $E_\pi$. Then $E(\pi)$ is
a cyclic extension of $\kappa(\pi)$. By (\ref{lifting_extension}), there exists
a cyclic extension $E$ of $F$ which is unramified on $A$ except possibly at
$\delta$ with  $E \otimes F_\pi \simeq E_\pi$.  Since $E/F$ is unramified on $A$ except possibly  at 
$\delta$ and $\lambda = w \pi^s \delta^t$ with $w$ a unit in $A$, $(E, \sigma, \lambda)$
is unramified on $A$ except possibly  at $(\pi)$ and $(\delta)$.  Since $\alpha$ is unramified 
on $A$ except possibly at $(\pi)$ and $(\delta)$, $s\alpha - (E, \sigma, \lambda)$
is unramified  on $A$ except possibly at $(\pi)$ and $(\delta)$. Since $s\alpha \otimes F_\pi = 
(E_\pi, \sigma, \lambda) = (E, \sigma, \lambda) \otimes F_\pi$,   by (\ref{zero}), $s\alpha = (E, \sigma, \lambda)$. 
\end{proof}

\begin{lemma}
\label{local_nonsquare}  
Suppose that $\alpha \cdot (\lambda)  = 0$ and $\lambda \not\in F^{*\ell}$. If $\alpha \neq 0$, 
then   ind$(\alpha \otimes F(\sqrt[\ell]{\lambda })) < $ ind$(\alpha)$ and $\alpha \cdot (\sqrt[\ell]{\lambda})
= 0 \in H^3(F(\sqrt[\ell]{\lambda}), \mu_n^{\otimes 2})$. 
\end{lemma}

\begin{proof} Suppose that $s $ is  coprime to $\ell$.
Then, by (\ref{local_coprime}),  $\alpha = 
 (E', \sigma', \lambda)$ for  some cyclic 
 extension $E'$ of $F$ which is unramified on $A$ except possibly at  $\delta$. 
 Since $\nu_\pi(\lambda) = s$ is coprime to $\ell$ and $E'/F$ is unramified at $\pi$, 
 it follows that  ind$(\alpha) = [ E'   :  F]$. In particular, ind$(\alpha \otimes F(\sqrt[\ell]{\lambda}))
 \leq [E' : F]/\ell < $ ind$(\alpha)$. Similarly, if $t$ is coprime to $\ell$, 
 then ind$(\alpha \otimes F(\sqrt[\ell]{\lambda})) < $ ind$(\alpha)$. 
 Further $\alpha \cdot (\sqrt[\ell]{\lambda}) = (E', \sigma',  \lambda) \cdot (\sqrt[\ell]{\lambda}) = 0$.
  
Suppose that $s$ and $t$ are divisible by $\ell$.  Since $\lambda = w\pi^s\delta^t$, 
we have  $ F(\sqrt[\ell]{\lambda}) = F(\sqrt[\ell]{w})$. Let $L = F(\sqrt[\ell]{\lambda}) = F(\sqrt[\ell]{w})$ 
and $B$  be the integral  closure of $A$ in $L$.  Since $w$ is a unit in $A$, by (\cite[Lemma 3.1]{PS1}), 
 $B$ is a complete regular local ring  with maximal ideal generated by
 $\pi$ and $\delta$.
   Since $w $ is not an $\ell^{\rm th}$ power
 in $F$ and $A$ is a complete regular local ring,  the image 
 of $w$ in $A/m$ is not an $\ell^{\rm th}$ power. 
 Since  $A/(\pi)$  is also a complete regular local ring with 
 residue field $A/m$,  the image of $w$ in $A/(\pi)$ 
   is not an $\ell^{\rm th}$ power.  Since  $F_\pi$   is a 
  complete discrete valued field with residue field the field of fractions of $A/(\pi)$,  
   $w$ is not an $\ell^{\rm th}$ power in $F_\pi$.
  Since $\alpha \cdot (\lambda) = 0$ and the residue field of $F_\pi$ 
  is a  local field,   by (\ref{lambda_non_square_local}), ind$(\alpha \otimes L_\pi) < $ ind$(\alpha )$.   
 Hence, by (\ref{local_index}), ind$(\alpha \otimes L) < $ ind$(\alpha)$. 
 
 Since   $L_\pi = 
  L\otimes F_\pi$ 
 and $L_\delta = L \otimes F_\delta$ are field extension of degree $\ell$ over $F_\pi$ 
 and $F_\delta$ respectively and  cores$(\alpha \cdot
 (\sqrt[\ell]{\lambda})) = \alpha \cdot (\lambda) = 0$, by 
 (\ref{corestriction}), $(\alpha \cdot (\sqrt[\ell]{\lambda}) ) \otimes L_\pi = 0 $ 
 and $(\alpha \cdot (\sqrt[\ell]{\lambda}) ) \otimes L_\delta= 0 $.
 Hence, by (\ref{zero}),   $\alpha \cdot (\sqrt[\ell]{\lambda}) = 0$. 
 \end{proof}

\begin{lemma}
\label{dropping_index} 
Suppose   $\alpha = (E/F, \sigma, u\pi\delta^{\ell m})$ for some $m \geq 0$,  
$u$ a unit in $A$, $E/F$ a cyclic extension of degree $\ell^d$ which is unramified on $A$ except possibly at 
$\delta$ and $\sigma$ a generator of Gal$(E/F)$.
Let $\ell^e$ be the ramification index of $E/F$ at $\delta$ and 
$f = d - e$. Let $i \geq 0$ be such that $\ell^f + \ell^{di} > \ell m$. Let $v \in A$ be a unit which is not in $F^{*\ell}$
and $L = F(\sqrt[\ell]{v\delta^{\ell^f + \ell^{di} - \ell m}+ u\pi})$.  If  $f > 0$, 
then ind$(\alpha \otimes L) <  $ ind$(\alpha)$
\end{lemma}

\begin{proof} 
Let $B$ be the integral closure of $A$ in $L$ and $r = \ell^f + \ell^{di} - \ell m$. 
Since $\ell^f + \ell^{di} > \ell m$,  $L = F(\sqrt[\ell]{v\delta^r + u \pi})$ and 
$v\delta^r  + u \pi$ is a regular prime in $A$. Thus $B$ is a complete regular local ring 
(cf. \cite[Lemma 3.2]{PS1}) and $\pi$,  $\delta$ remain primes in $B$.   Note that $\pi$ and $\delta$
may not generate the maximal ideal of $B$.
Let $L_\pi$ and $L_\delta$ be the completions of $L$ at the  discrete valuations given by $\pi$ and $\delta$ respectively.
Since $v \not \in F^{*\ell}$, $F(\sqrt[\ell]{v})$ is
the unique degree $\ell$ extension of $F_\pi$ which is unramified on $A$. Since $f > 0$, 
there is  a subextension $E$ of degree $\ell$ over $F$ which is unramified on $A$ and hence 
$F(\sqrt[\ell]{v}) \subset E$. 

Since $E/F$ is unramified on $A$ except possibly at $\delta$, by (\ref{local_index}),
$[E : F] = [E_\pi : F_\pi]$ and hence ind$(\alpha) = $ per$(\alpha) = [E : F]$.

Since $r$ is divisible by $\ell$,   $L_\pi \simeq F_\pi(\sqrt[\ell]{v})$ and  hence 
$L_\pi \subset E_\pi$. Thus ind$(\alpha \otimes L_\pi) < $ ind$(\alpha)$.
Since $r  > 0$,  $L_\delta  \simeq F_\delta(\sqrt[\ell]{u\pi})$. Since 
$\alpha = (E/F, \sigma, u\pi\delta^{\ell m})$,  ind$(\alpha \otimes L_\delta) < 
 [E \otimes L_\delta : L_\delta]  \leq [E : F]$. 
In particular by (\ref{local_local_period_index}), per$(\alpha \otimes F_\pi) < $ ind$(\alpha)$
and per$(\alpha \otimes F_\delta) < $ ind$(\alpha)$. Since $\alpha \otimes L$ is unramified on $B$
except possibly at $\pi$ and $\delta$ and $H^2(B, \mu_\ell) = 0$, 
per$(\alpha \otimes L) < $ ind$(\alpha)$. If $d = 1$, then per$(\alpha \otimes L) < $ ind$(\alpha) = \ell$
and hence per$(\alpha \otimes L) = $ ind$(\alpha \otimes L) = 1 < $ ind$(\alpha)$.  Suppose that $d \geq 2$.

Let $ \phi: \XX \to {\rm Spec}(B)$ be a sequence of blow-ups such that  the ramification locus of 
$\alpha \otimes L$ is a union of regular curves with normal crossings. Let $V = \phi^{-1}(P)$.
 To show that ind$(\alpha \otimes L) < $ ind$(\alpha)$, 
by (\ref{hhk}), it is enough to show that for every point $x$ of $V$, ind$(\alpha \otimes L_x) < $ ind$(\alpha)$.

Let $x \in V$  be a closed point. Then, by (\ref{period=index}), ind$(\alpha \otimes L_x) = $ per$(\alpha \otimes L_x)$. 
Since per$(\alpha \otimes L_x) < $ ind$(\alpha)$, ind$(\alpha \otimes L_x)  < $ ind$(\alpha)$.

Let $x \in V$ be a  codimension zero point.  Then $\phi(x) $ is the closed point of Spec$(B)$.
Let $\tilde{\nu}$ be the   discrete valuation of $L$ given by $x$.
Then $\kappa(\tilde{\nu}) \simeq \kappa'(t)$ for some finite extension $\kappa'$ over $\kappa$ and a  variable $t$ over $\kappa$. 
Let  $\nu$ be the restriction of $\tilde{\nu}$ to $F$.  

Suppose that $\nu(\delta^r) < \nu(\pi)$.  Then $L\otimes F_\nu = F_\nu(\sqrt[\ell]{v\delta^r})$.
Since $\ell$ divides $r$, $L \otimes F_\nu = F_\nu(\sqrt[\ell]{v})$.  Since $F(\sqrt[\ell]{v}) \subset E$,    
ind$(\alpha \otimes L\otimes F_\nu) < $ ind$(\alpha)$.
Suppose that $\nu(\delta^r) >  \nu(\pi)$.  Then $L\otimes F_\nu = F_\nu(\sqrt[\ell]{u\pi})$
and as above   ind$(\alpha \otimes L\otimes F_\nu) < $ ind$(\alpha)$. 
Suppose that $\nu(\delta^r) = \nu(\pi)$. Let $\lambda = \pi/\delta^r$.
Then $\lambda$ is a unit at $\nu$ and $L_{\tilde{\nu}} = F_\nu(\sqrt[\ell]{v + u\lambda})$.
We have $u\pi\delta^{\ell m} = u \lambda \delta^{r + \ell m} = u\lambda \delta^{\ell^f + \ell^{di}}$ and 
$$\alpha \otimes F_\nu = (E\otimes F_\nu/F_\nu, \sigma \otimes 1,  u\pi\delta^{\ell m}) = 
(E\otimes F_\nu/F_\nu, \sigma \otimes 1,  u\lambda \delta^{\ell^f + \ell^{di}}). $$
Since $[E : F] = \ell^d$, $\alpha \otimes F_\nu = 
(E\otimes F_\nu/F_\nu, \sigma \otimes 1,  u\lambda \delta^{\ell^f}).$
Suppose that $f = d$. Then $E/F$ is unramified and hence every element of 
$A^*$ is a norm from $E$. Thus $(E \otimes F_\nu/F_\nu, \sigma \otimes 1, w_0u\lambda)$ with 
$w_0 \in A^* \setminus A^{*\ell}$.
Suppose that $f < d$. Then $e = d - f > 0$ and hence by (\ref{extns_2_local}), 
we have $E = E_{nr}(\sqrt[\ell^e]{w\delta})$, for some unit $w$ in the integral  closure of $A$ in $E_{nr}$,
with $N(\sqrt[\ell^d]{w\delta}) =  w_1\delta^{\ell^f}$ with $w_1 \in A^* \setminus A^{*\ell}$. 
Thus $$\alpha \otimes F_\nu =  (E\otimes F_\nu/F_\nu, \sigma \otimes 1,  u\lambda \delta^{\ell^f}) = (E\otimes F_\nu/F_\nu, \sigma \otimes 1,  w_0u\lambda). $$ with $w_0 = w_1^{-1}$.

If $E \otimes F_\nu$ is not  a field, then ind$(\alpha \otimes F_\nu) <  [E : F]$.
Suppose $E \otimes F_\nu$ is a field.  Let $\theta = w_0u\lambda$.
Since $\alpha \otimes F_\nu = (E \otimes F_\nu/F_\nu, \sigma \otimes 1, \theta)$,  
 ind$(\alpha \otimes L \otimes F_\nu ) \leq $ ind$(\alpha \otimes L \otimes  F_\nu(\sqrt[\ell^{d-1}]{\theta}) ) \cdot
  [ L\otimes F_\nu(\sqrt[\ell^{d-1}]{\theta}) : L\otimes F_\nu]$. 
  Since $[ L\otimes F_\nu(\sqrt[\ell^{d-1}]{\theta}) : L\otimes F_\nu] \leq \ell^{d-1} < [E : F]$,  it is enough to show that 
$ \alpha \otimes L \otimes  F_\nu(\sqrt[\ell^{d-1}]{\theta})$ is trivial. 

Since $F(\sqrt[\ell]{v})/F$ is the unique subextension of $E/F$ degree $\ell$ and $[E : F] = \ell^d$,
we have  $\alpha \otimes F_\nu(\sqrt[\ell^{d-1}]{\theta}) =  (F_\nu (\sqrt[\ell^{d-1}]{\theta}, \sqrt[\ell]{v}) /F(\sqrt[\ell^{d-1}]{\theta}), \sigma, 
\sqrt[\ell^{d-1}]{\theta}) $ (cf. \ref{cyclic_algebra_extn}).  Let $M = F_\nu(\sqrt[\ell^{d-1}]{\theta})$.
Since $\kappa$ contains  a primitive $\ell^{\rm th}$ root of unity, we have $\alpha \otimes M = (v, \sqrt[\ell^{d-1}]{\theta})_\ell$.
Then $M$ is a complete discrete valuation field. Since $\lambda$ is a unit at $\nu$, $\theta$ is a unit at $\nu$.
Hence the residue field of $M$ is $\kappa(\nu)(\sqrt[\ell^{d-1}]{\overline{\theta}})$.
Since $\theta$ and $v$ are units at $\nu$,
$\alpha \otimes M =  (v, \sqrt[\ell^{d-1}]{\theta})$ is unramified at the discrete valuation of $M$.
Hence it is enough to show that the specialization $\beta$ of $\alpha \otimes M$  is trivial 
over $\kappa(\nu)(\sqrt[\ell^{d-1}]{\overline{\theta}}) \otimes L_0$, where $L_0$ is the residue field of 
$L \otimes F_\nu$ at $\nu$.  

Suppose that $L_{\tilde{\nu}}/F_\nu$ is ramified. Since $L_{\tilde{\nu}}  = F_\nu(\sqrt[\ell]{u + v\lambda})$, 
$v + u\lambda$ is not a unit at $\nu$. Thus $v = -u\lambda$ modulo $F_\nu^{*\ell^{d}}$ and 
$\theta = w_0u\lambda = -w_0v$ modulo $F_\nu^{*\ell^{d}}$.  In particular 
$\sqrt[\ell^{d-1}]{\theta } = \sqrt[\ell^{d-1}]{-w_0v}$ modulo $M^{*\ell}$. 
Since $\overline{v}, \overline{w_0} \in \kappa$ and $\kappa$ a finite field, 
$\beta = (\sqrt[\ell]{\overline{v}}, \sqrt[\ell^{d-1}]{\theta}) = (\sqrt[\ell]{\overline{v}}, \sqrt[\ell^{d-1}]{-\overline{w}_0\overline{v}})$ is trivial.

Suppose that $L_{\tilde{\nu}}/F_\nu$ is unramified. Then $L_0 = \kappa(\pi)(\sqrt[\ell]{\overline{v} + \overline{u\lambda}})$.
Since $\kappa(\pi)$ is a global field and $d - 1 \geq 1$, by (\ref{dropping_ind_dvr}), $\beta \otimes L_0(\sqrt[\ell^{d-1}]{\overline{\theta}}) = 0$. 
 \end{proof}

 \begin{lemma}  
\label{patching_at_P} Suppose  $L_\pi/F_\pi$ and $L_\delta/F_\delta$ are unramified cyclic field extensions 
of degree $\ell$  and $\mu_\pi \in L_\pi$, $\mu_\delta \in L_\delta$ such that \\
$\bullet$ ind$(\alpha \otimes L_\pi) < d_0$ for some $d_0$, \\
$\bullet$   $\lambda  =   N_{L_\pi/F_\pi}(\mu_\pi)$ and $ \lambda  =  N_{L_\delta/F_\delta}(\mu_\delta)$,\\
$\bullet$ $\alpha \cdot (\mu_\pi) = 0 \in H^3(L_\pi, \mu_n^{\otimes 2})$,
$\alpha \cdot (\mu_\delta) = 0 \in H^3(L_\delta, \mu_n^{\otimes 2})$,\\
$\bullet$ if  $\lambda \in F_P^{*\ell}$  and $\alpha = (E/F, \sigma, v\pi)$
for some cyclic extension $E/F$ which is unramified on $A$ except
possibly  at $\delta$, then $L_\delta/F_\delta = F_\delta(\sqrt[\ell]{v\pi}).$\\
Then there exists a cyclic extension $L/F$
of degree $\ell$  and $\mu \in L$ such that \\ 
$\bullet$ ind$(\alpha \otimes L) <  d_0$, \\
$\bullet$ $\lambda = N_{L/F}(\mu)$,  \\
$\bullet$ $\alpha \cdot (\mu) = 0 \in H^3(L, \mu_n^{\otimes 2})$, \\
$\bullet$  $L \otimes F_\pi \simeq L_\pi$ and $L  \otimes F_\delta \simeq L_\delta$. \\
\end{lemma}

\begin{proof}  Since  $\alpha \cdot (\mu_\pi) = 0 \in H^3(L_\pi, \mu_n^{\otimes 2}) $ and  
$\lambda  =   N_{L_\pi/F_\pi}(\mu_\pi)$, by taking the corestriction, we see that 
$\alpha \cdot (\lambda) = 0 \in H^3(F_\pi, \mu_n^{\otimes 2})$.
Since $\alpha \cdot (\lambda)$ is unramified on $A$ except possibly at $\pi$ and $\delta$, 
 by (\ref{zero}), $\alpha \cdot (\lambda) = 0$.

 Suppose that $\lambda \not\in F^{*\ell}$.
Then, by (\ref{dvr_local_field_norm}) and (\ref{local_nonsquare}), $L = F(\sqrt[\ell]{\lambda})$ and
$\mu = \sqrt[\ell]{\lambda}$ have the required properties.

Suppose that $\lambda \in F^{*\ell}$.  
Let $L(\pi)$ and $L(\delta)$ be the residue fields of $L_\pi$ and $L_\delta$
respectively. Since $L_\pi/F_\pi$ and $L_\delta/F_\delta$ are unramified cyclic extensions of
degree $\ell$, $L(\pi)/\kappa(\pi)$ and $L(\delta)/\kappa(\delta)$ are cyclic extensions of degree 
$\ell$.  Since $F$ contains a primitive $\ell^{\rm th}$ root of unity, we have 
$L(\pi)  = \kappa(\pi)[X]/(X^\ell - a)$ and $L(\delta)  = \kappa(\delta)[X]/(X^\ell - b)$
for some $a \in \kappa(\pi)$ and $b \in \kappa(\delta)$.  Since $\kappa(\pi)$ is a complete 
discretely valued field with $\overline{\delta}$ a parameter, without loss of generality we assume 
that $a = \overline{u_1} \overline{\delta}^{\epsilon}$ for some unit $u_1 \in A$ and $\epsilon = 0$ or 1.
Similarly we have $b = \overline{u_2} \overline{\pi}^{\epsilon'}$ for some unit $u_2 \in A$ and $\epsilon' = 0$ or 1.

By (\ref{local-cyclic}), we assume that $\alpha = (E/F, \sigma, u\pi\delta^j)$ for some cyclic extension 
$E/F$ which is unramified on $A$ except possibly at $\delta$, $u$ a unit in $A$ and $j \geq 0$.
Then ind$(\alpha) = [E : F]$.  Let $E_0$ be the residue field of $E$ at $\pi$.
Then $[E : F] = [E_0 : \kappa(\pi)]$. Since $\partial_\pi(\alpha) = (E_0/\kappa(\pi), \overline{\sigma})$, per$(\partial_\pi(\alpha)) =
[E : F] = $ ind$(\alpha)$. Since $L_\pi/F_\pi$ is an unramified cyclic extension of degree $\ell$
and ind$(\alpha \otimes L_\pi) < $ ind$(\alpha)$, the residue field $L(\pi)$ of $L_\pi$ 
 is the unique degree $\ell$ subextension of $E_0/\kappa(\pi)$.

Suppose that $\epsilon = \epsilon' = 0$.  Since $L_\pi$ and $L_\delta$ are fields, 
$u_1$ and $u_2$ are not $\ell^{\rm th}$ powers.   Let  $L/F$ be the unique cyclic field extension of degree 
$\ell$ which is unramified on $A$.  Then $L \otimes F_\pi \simeq L_\pi$ and $L \otimes F_\delta \simeq L_\delta$.
Let $B$ be the integral closure of $A$ in $L$. Then $B$ is a regular local ring with maximal ideal
$(\pi, \delta)$ and hence by (\ref{local_index}) ind$(\alpha \otimes L) < $ ind$(\alpha)$.

Suppose $\epsilon = 1$. Then $L_\pi = F_\pi(\sqrt[\ell]{u_1\delta})$ and
  $L(\pi) = \kappa(\pi)(\sqrt[\ell]{\overline{u_1}\overline{\delta}})$.
  Since  $E_0/\kappa(\pi)$ is a cyclic extension containing a totally ramified extension,  $E_0 /\kappa(\pi)$
is a totally ramified cyclic extension. Thus $\kappa(\pi)$ contains  a primitive $\ell^{d^{\rm th}}$ root of unity
and $E_0 = \kappa(\pi)(\sqrt[\ell^d]{\overline{u_1}\overline{\delta}})$. 
In particular $F$ contains a primitive $\ell^{d^{\rm th}}$ root of unity and 
$\alpha = (u_1\delta,  u\pi\delta^j) = (u_1\delta, u'\pi)$. Since $L_\delta/F_\delta$ is an unramified extension of 
degree $\ell$ with ind$(\alpha \otimes L_\delta) < $ ind$(\alpha)$, as above, we have $L_\delta = F_\delta(\sqrt{u'\pi})$
and hence $\alpha = (u_1\delta, u_2\pi)$. Let $L = F(\sqrt[\ell]{u_1\delta + u_2\pi })$. Then 
$L \otimes F_\pi \simeq L_\pi$ and $L\otimes F_\delta \simeq L_\delta$.
Since for any $a, b \in F^*$, 
$(a, b) = (a + b, -a^{-1}b)$, we have $\alpha = (u_1\pi + u_2\delta, -u_1^{-1}\pi^{-1}u_2\delta)$.
In particular ind$(\alpha \otimes L) < $ ind$(\alpha)$. 
 
Suppose that  $\epsilon = 0$ and $\epsilon' = 1$.  Suppose $j$ is coprime to $\ell$. 
Then, by (\ref{residue}),  ind$(\alpha) = $ per$(\partial_\delta(\alpha))$ and as in the proof of (\ref{local-cyclic}), 
we have $\alpha = (E'/F, \sigma', v\delta\pi^{j'})$ for some cyclic extension $E'/F$ which 
is unramified on $A$ except possibly at $\pi$.  Thus, we have the required extension as in the case 
$\epsilon = 1$.  

 Suppose $j$ is divisible by $\ell$.   Since $\epsilon =0$, $L_\pi = F_\pi(\sqrt[\ell]{u_1})$.
 Since the residue field $L_\pi(\pi)$ of $L_\pi$ is contained in the residue field $E_0$ of $E$ at $\pi$, 
 $F(\sqrt[\ell]{u_1}) \subset E$ and hence  $E/F$ is not totally ramified at $\delta$.
 Since $E/F$ is unramified on $A$ except possibly at $\delta$, by (\ref{extns_2_local}), 
 $E= E_{nr}(\sqrt[\ell^e]{w\delta})$ for some unit $w$ in the integral  closure of $A$ in $E_{nr}$.
 Suppose $ e = 0$. Then $E = E_{nr}/F$ is unramified on $A$. Since $\kappa$ is a finite field and
 $A$ is complete, every unit in $A$ is a norm from $E/F$. Thus multiplying $u\pi\delta^j$ by a norm from 
 $E/F$ we assume that $\alpha = (E/F, \sigma, u_2\pi\delta^j)$.  Suppose that $e > 0$. 
 Then, by (\ref{extns_2_local}), $N_{E/F}(w\delta) = w_1\delta^{\ell^f}$ with $w_1 \in A^* \setminus A^{*\ell}$.
 Since $A^*/A^{*\ell}$ is a cyclic group of order $\ell$, we have $\alpha = (E/F, \sigma, u_2\pi\delta^{j + j'\ell^{f}})$ for some 
 $j' $. Since $j$ is divisible by $\ell$ and $f \geq 1$, $j + j'\ell^f$ is divisible by $\ell$. 
 Hence, we assume that $\alpha = (E/F, \sigma, u_2\pi\delta^{\ell m})$ for some $m$.  
Thus, by (\ref{dropping_index}),  there exists $i \geq 0$ such that 
 ind$(\alpha \otimes L) < $ ind$(\alpha)$ for $L = F(\sqrt[\ell]{u_1\delta^{\ell^{f + di}} + u_2\pi\delta^{\ell m}})$ .

By the choice, we have  $L/F$ is   the unique unramified extension  or $L = F(\sqrt[\ell]{u_1\delta + u_2\pi})$ or 
$L = F(\sqrt[\ell]{u_1\delta^{\ell^{f + di}} + u_2 \pi\delta^{\ell m}})$ with $\ell^{f + di} > \ell m$.
Let $B$ be the integral closure of $A$ in $L$. Then $B$ is a complete regular local ring 
with $\pi$ and $\delta$ remain prime in $B$. 
 Since $\lambda = w\pi^s \delta^t$ and $\lambda \in F_P^{*\ell}$, 
we have $\lambda = w_0^\ell\pi^{\ell s_1}\delta^{\ell t_1}$
for some unit $w_0 \in A$. Let $\mu = w_0\pi^{s_1}\delta^{t_1} \in F$. Then $N_{L/F}(\mu) = \mu^\ell = \lambda$.
Since $\alpha \cdot (\lambda) = 0$, by (\ref{corestriction}), $\alpha \cdot (\mu) = 0$ in $H^3(L_\pi, \mu_n^{\otimes 2})$
and $H^3(L_\delta, \mu_n^{\otimes 2})$.  Hence $\alpha \cdot (\mu)$ is unramified at all height one
prime ideals of $B$.  Since $B$ is a complete regular local ring with residue field  $\kappa$ finite,  
 $\alpha \cdot (\mu) = 0$ (\ref{purity}).   
\end{proof}

\begin{lemma}
\label{curve_point}
Suppose that  $\nu_\pi(\lambda)$ is divisible by $\ell$, $\alpha$ is unramified
on $A$ except possibly at $\pi$ and $\delta$,  and $\alpha \cdot (\lambda) = 0$.
Let $L_\pi$  be a cyclic  unramified or split extension of $F_\pi$ of degree $\ell$, $\mu_\pi  \in L_\pi$  and $d_0 \geq 2$   such that \\
$\bullet$  $N_{L_ \pi/F_\pi}( \mu_\pi) =  \lambda $, \\
$\bullet$ ind$(\alpha \otimes L_\pi) <  d_0$, \\
$\bullet$ $\alpha \cdot (\mu_\pi) = 0$ in $H^3(L_\pi, \mu_n)$.\\
 Then there exists an extension $L$ over $F$ of degree $\ell$ and $\mu \in L$ 
 such that \\
 $\bullet$ $N_{L/F}(\mu) =  \lambda $, \\
 $\bullet$ ind$(\alpha \otimes L)  <  d_0 $, \\
 $\bullet$ $\alpha \cdot (\mu) = 0 \in H^3(L, \mu_n^{\otimes 2})$ and \\
 $\bullet$  there is an isomorphism $\phi : L_\pi  \to L   \otimes F_\pi$  with
$$\phi(\mu_\pi) (\mu_P \otimes 1)^{-1}  \in (L_P \otimes F_\pi)^{\ell^{m}},$$ for all $m \geq 1$. 
 \end{lemma}

\begin{proof}  Since $\nu_\pi(\lambda)$ is divisible by $\ell$, $\lambda = w\pi^{r\ell}\delta^s$
for some $w \in A$ a unit.

Suppose that $L_\pi = \prod F_\pi$ is a split extension. 
Let $L  = \prod F $ be the split extension of degree $\ell$. Since $\mu_\pi \in L_\pi$, 
we have $\mu_\pi  = (\mu_1, \cdots, \mu_\ell)$ with $\mu_i \in F_\pi$. Write $\mu_i = \theta_i \pi^{r_i}$
with $\theta_i \in F_\pi$ a unit at its discrete valuation. Since $N_{L_\pi/F_\pi}(\mu_\pi) = \lambda = w\pi^{r\ell}$,
 $\theta_1 \cdots \theta_\ell = w$ and $\pi^{r_1 + \cdots + r_\ell} = \pi^{r\ell}$.
 For $2 \leq i \leq \ell$, let  $\bar{\theta}_i$    be  the image of $\theta_i$ in the residue field
 $\kappa(\pi)$ of $F_\pi$. Since $\kappa(\pi)$ is the field of fractions of $A/(\delta)$ and
 $A/(\delta)$ is a  complete discrete valuation ring with $\bar{\delta}$ as a parameter,
 we have $\bar{\theta}_i =  \bar{u}_i\bar{\delta}^{s_i}$ for some unit $u_i \in A$.
 For $2 \leq i \leq \ell$, let  $\tilde{\theta}_i =  u_i\delta^{s_i} \in F$, 
 $\tilde{\theta}_1 = w\tilde{\theta}_2^{-1} \cdots \tilde{\theta}_{\ell}^{-1}$
 and $\mu = (\tilde{\theta}_1\pi^{r_1}, \cdots, \tilde{\theta}_{\ell}\pi^{r_\ell}) \in L = \prod F$.
 Then $N_{L/F}(\mu) = \lambda$. Since $\tilde{\theta}_i\pi^{r_i}  \mu_i^{-1}$ is a unit at $\pi$ 
 with image 1 in $\kappa(\pi)$,  $ \tilde{\theta}_i\pi^{r_i}  \mu_i^{-1}  \in F_\pi^{\ell^m}$ for any $m \geq 1$. 
 In particular $\alpha \cdot (\tilde{\theta}_i \pi^{r_i}) = 
 \alpha \cdot (\mu_i) = 
 0 \in H^3(F_\pi, \mu_n^{\otimes 2})$. Since $\alpha$ is unramified on $A$ except possibly at $\pi$ 
 and $\delta$ and $\tilde{\theta}_i = u_i\delta^{t_i}$ with $u_i \in A $ a unit,  
 $\alpha \cdot (\tilde{\theta}_i \pi^{r_i})$ is unramified on $A$ except possibly at $\pi$ and $\delta$.
 Thus, by  (\ref{zero}), $\alpha \cdot (\tilde{\theta}_i\pi^{r_i}) = 0 \in H^3(F, \mu_n^{\otimes 2})$.
  Since ind$(\alpha \otimes L_\pi) <  d_0$ and $L_\pi$ is the split extension, 
 ind$(\alpha \otimes F_\pi) < d_0$. Since $\alpha$ is unramified on $A$ except 
 possibly at $\pi$ and $\delta$, by (\ref{local_index}), ind$(\alpha) < d_0$. 
 Thus $L$ and $\mu = (\tilde{\theta}_1\pi^{r_1}, \cdots , \tilde{\theta}_\ell\pi^{r_\ell}) \in L$ 
have  the required properties. 
 
Suppose that $L_\pi$ is a field extension of $F_\pi$. 
By (\ref{lifting_extension}),  there exists a cyclic 
extension $L$ of $F$ of degree $\ell$ which is unramified 
on $A$ except possibly  at $\delta$ with   $L \otimes F_\pi \simeq L_\pi$.  
Let $B$ be the integral  closure of $A$ in $L$. 
By the construction of $L$, 
either $L/F$ is unramified $A$ or $L = F(\sqrt[\ell]{u\delta})$ for some 
unit $u \in A$.  Replacing $\delta$ by $u\delta$, we assume that $L/F$ is unramified on $A$
or $L = F(\sqrt[\ell]{\delta})$.  In particular, $B$ is a regular local ring with maximal ideal
generated  by $(\pi, \delta')$ for $\delta' = \delta$ or $\delta' = \sqrt[\ell]{\delta}$ (cf. \cite[Lemma 3.2]{PS1}).
Since $\alpha$ is unramified on $A$ except possibly at $\pi$ and $\delta$,  
 $\alpha \otimes L$ is unramified on  $B$ except possibly at $\pi$ and $\delta'$. 
 Since  ind$(\alpha \otimes L _\pi) <  d_0$,  
 by (\ref{local_index}), ind$(\alpha \otimes L) = $ ind$(\alpha \otimes L_\pi) <  d_0$.

Since $L_\pi/F_\pi$ is unramified and $N_{L_\pi/F_\pi}(\mu_\pi) = \lambda = w\pi^{r\ell}\delta^s$, 
we have $\mu_\pi = \theta_\pi \pi^r$ for some $\theta_\pi \in L_\pi$ which is a unit at its discrete valuation. 
Let $\overline{\theta}_\pi$ be the image of $\theta_\pi$ in $L(\pi)$.
Since  $L(\pi)$  is the field of fractions of the complete discrete valuation ring $B/(\pi)$
and $\overline{\delta'}$ is a parameter in $B/(\pi)$, we have 
$\overline{\theta}_\pi = \overline{v} \overline{\delta'}^t$ for some unit $v \in B$.
Since  $N_{L(\pi)/\kappa(\pi)}(\overline{\theta}_\pi) = \overline{w}\overline{\delta}^s$,
it follows that $N_{L(\pi)/\kappa(\pi)}(\overline{v})  = \overline{w}$. Since  $w$ is a unit in $A$, 
 there exists a unit $\tilde{v} \in B$  with $N_{L/F}(\tilde{v}) = w$ and $\overline{\tilde{v}}
  = \overline{v}$.  Let $\mu = \tilde{v} \pi^r{\delta'}^t \in L$.  Then $\mu\mu_\pi^{-1} \in L_\pi$ is
  a unit in the valuation ring at $\pi$ with the image 1 in the residue field $L(\pi)$ and hence $\mu\mu_\pi^{-1} \in L_\pi^{\ell^m}$ for all 
  $m \geq 1$. In particular 
  $\alpha \cdot (\mu) = \alpha \cdot (\mu_\pi) = 0 \in H^3(L_\pi, \mu_n^{\otimes 2})$.
  Since ind($\alpha \otimes L_\pi) <  d_0$,
  by (\ref{local_index}), ind$(\alpha \otimes L) < d_0$. Since 
  $ \alpha \cdot (\mu) = 0$ in $H^3(L_\pi, \mu_n^{\otimes 2})$,
  $\alpha$ is unramified on $A$ except possibly at $\pi$ and the support of $\mu$ on $A$ is at most 
  $ \pi$,   by (\ref{zero}), $\alpha \cdot (\mu) = 0$ in $H^3(L, \mu_n^{\otimes 2})$.
   \end{proof}

   \begin{lemma} 
  \label{norm}  Suppose that $\alpha \cdot (\lambda) = 0$ and 
  $\nu_\delta(\lambda) = s\ell$.  
  Suppose that $\alpha = (E/F, \sigma, \pi\delta^m)$ for some cyclic extension $E/F$
  which is unramified on $A$ except possibly at $\delta$.
  Let  $E_\delta$ be the lift of the residue of $\alpha$ at  
  $\delta$.  If $s\alpha \otimes   E_\delta = 0$,   then there exists  an integer $r_1 \geq 0$ such that  
  $w_1\delta^{mr_1 - s}$ is a norm from the extension $E/F$ for some unit $w_1 \in A$. 
  \end{lemma}
  
\begin{proof}  
Write $\alpha \otimes F_\delta = \alpha' + (E_\delta/F_\delta, \sigma_\delta, \delta)$ as in (\ref{rbc}).
Since $\alpha \otimes E_\delta = \alpha' \otimes E_\delta$,  $s\alpha' \otimes E_\delta = 0$.
Hence $s\alpha' =  (E_\delta, \sigma, \theta)$ for some $\theta \in F_\delta$.
Since $\alpha'$  and $E_\delta/F_\delta$ are  unramified at $\delta$, we assume that $\theta \in F_\delta$ is
a unit at $\delta$. Since the residue field  $\kappa(\delta)$ of $F_\delta$ is a complete discrete 
valued field with the image of $\pi$ as a parameter, without loss of generality we assume that 
$\theta = w_0\pi^{r_1}$ for unit $w_0 \in A$ and $r_1 \geq 0$.  Let $\lambda_1 = w_0\pi^{r_1}\delta^s$.
Since $s\alpha' = (E_\delta, \sigma_\delta, \theta)$, by (\ref{dot-zero}), 
$\alpha \cdot (\lambda_1) = 0 \in H^3(F_\delta, \mu_n^{\otimes 2})$.
Since $\alpha$ is unramified on $A$ except possibly at $\pi$,$\delta$ and
$\lambda_1 = w_0\pi^{r_1}\delta^s$ with $w_0 \in A$ a unit,  $\alpha \cdot (\lambda_1)$
is unramified in $A$ except possibly at $\pi$ and $\delta$.
Hence, by (\ref{zero}), $\alpha \cdot (\lambda_1) = 0 \in H^3(F, \mu_n^{\otimes 2})$.
We have 
$$ 0 = \partial_\pi(\alpha \cdot (\lambda_1))  = \partial_\pi((E/F, \sigma, u\pi\delta^m) \cdot (w_0\pi^{r_1}\delta^s)) = 
 (E(\pi)/\kappa(\pi), \overline{\sigma},  (-1)^{r_1} \overline{u}^{r_1}\overline{w}_1^{-1}\overline{\delta}^{mr_1 - s}).$$
Since $(E/F, \sigma, (-1)^{r_1}u^{r_1}w_0^{-1} \delta^{mr_1 - s})$ is unramified on $A$ except possibly at 
$\pi$ and $\delta$, by (\ref{zero}), $(E/F, \sigma, (-1)^{r_1}u^{r_1}w_0^{-1} \delta^{mr_1 - s}) = 0$.
In particular $(-1)^{r_1}u^{r_1} w_0^{-1} \delta^{mr_1 - s}$ is a norm from the extension 
$E/F$. 
\end{proof}

  \begin{lemma} 
  \label{type2_56}  
   Suppose that $\alpha \cdot (\lambda) = 0$ and $\lambda = w\pi^{r}\delta^{s\ell}$ for some
  unit $w \in A$ and $r$ coprime to $\ell$. Let  $E_\delta$ be the lift of the residue of $\alpha$ at  $\delta$.
  If  $s\alpha \otimes   E_\delta = 0$,   then there exists  $\theta \in A$ such that \\
$\bullet$ $\alpha \cdot (\theta) = 0$, \\ 
$\bullet$ $\nu_\pi(\theta) = 0$, \\
$\bullet$ $\nu_\delta(\theta) = s$.
  \end{lemma}
  
\begin{proof}  Since $r$ is coprime to $\ell$, by (\ref{local_coprime}),
$\alpha = (E/F, \sigma, \lambda)$ for some cyclic  extension $E/F$ which is
unramified on $A$ except possible at $\delta$. Let $t = [E : F]$.
Since $t$ is a power of $\ell$ and $r$ is coprime to $\ell$, there exists an integer $r' \geq 1$
such that  $rr' \equiv 1 $ modulo $t$. We have
$$\alpha = \alpha^{rr'} = (E/F, \sigma, w\pi^r\delta^{s\ell})^{rr'} = (E/F, \sigma)^r \cdot (w\pi^r\delta^{s\ell})^{r'} = 
(E/F, \sigma)^r \cdot (w^{r'}\pi\delta^{r's\ell}).$$
Since $r$ is coprime to $\ell$, we also have $(E/F, \sigma)^r = (E/F, \sigma^{r'})$ (cf. \S 2)
and hence $\alpha = (E/F, \sigma^r, \pi\delta^{r's\ell})$.
Thus, by (\ref{norm}),  there exist a unit $w_1 \in A$ and $r_1 \geq 0$ such that 
$w_1\delta^{r's\ell r_1 - s}$ is a norm from $E/F$. 
Since $r'\ell r_1 - 1$ is coprime to $\ell$, $r'\ell r_1 - 1$ is coprime to $t$ and hence
there exists an integer  $r_1 \geq 0$ such that $(r'\ell r_1 - 1)r_2 \equiv 1 $ modulo $t$.
In particular $w_1^{r_2}\delta^s \equiv (w_1\delta^{r's\ell r_1 -s})^{r_2}$ modulo $F^{*t}$
and hence $ w_1^{r_2}\delta^s $ is a norm from $E/F$. Thus $\theta = w_1^{r_2}\delta^s $
has the required properites. 
\end{proof}

  \begin{lemma} 
  \label{type56}  
  Let $E_\pi$ and $E_\delta$ be the lift of the residues of $\alpha$ at $\pi$ and 
  $\delta$ respectively. Suppose that $\alpha \cdot (\lambda) = 0$ and $\lambda = w\pi^{r\ell}\delta^{s\ell}$ for some
  unit $w \in A$. If   $\alpha \cdot (\lambda) = 0$, $r\alpha \otimes E_\pi = 0$ and
$s\alpha \otimes   E_\delta = 0$,   then there exists  $\theta \in A$ such that \\
$\bullet$ $\alpha \cdot (\theta) = 0$, \\ 
$\bullet$ $\nu_\pi(\theta) = r$, \\
$\bullet$ $\nu_\delta(\theta) = s$.
  \end{lemma}
  
\begin{proof}  By (\ref{local-cyclic}), we assume that $\alpha = (E/F, \sigma, u\pi\delta^m)$
for some extension $E/F$ which is unramified on $A$ except possible at $\delta$ and $m \geq 0$. 
Without loss of generality, we assume that $0 \leq m <  [E : F]$.
By (\ref{norm}), there exists an integer $r_1 \geq 0$ such that $w_1\delta^{  mr_1 -s}$ is a norm from 
$E/F$. Let $t = [E - F] $ and $\theta = (-u\pi  +  \delta^{t - m})^{r_1 - r}w_1^{-1}(-u)^{r}\pi^r\delta^s$.
Since $t - m > o$, we have $\nu_\pi(\theta) = r$ and $\nu_\delta(\theta)  = s$.

Now we show that $\alpha \cdot (\theta) = 0$.  Since 
$t-m > 0$, we have $(-u\pi + \delta^{t - m} )^{r_1 - r} = (-u\pi)^{r_1 - r}$ modulo $\delta$
and hence $\theta \equiv  (-u)^{r_1 - r}\pi^{r_1 - r}w_1^{-1} (-u)^{r}\pi^r\delta^s = w_1^{-1}(-u)^{r_1}\pi^{r_1}\delta^s$ modulo $F_\delta^{*t}$.
Since $w_1\delta^{ mr_1-s}$ is a norm from $E/F$, we have  
$$
\begin{array}{rcl}
(\alpha \cdot (\theta) ) \otimes F_\delta & = & (E/F, \sigma, u\pi\delta^m) \cdot (w_1^{-1}(-u)^{r_1}\pi^{r_1}\delta^s) \otimes F_\delta \\
& = &    (E/F, \sigma, u\pi\delta^m) \cdot (w_1^{-1}(-u)^{r_1}\pi^{r_1}\delta^s w_1\delta^{mr_1 - s})  \otimes F_\delta \\
& = & (E/F, \sigma, u\pi\delta^m) \cdot ((-u)^{r_1}\pi^{r_1}\delta^{mr_1})  \otimes F_\delta \\
& = &  (E/F, \sigma, u\pi\delta^m) \cdot ((-u\pi\delta^m)^{r_1})  \otimes F_\delta = 0.
\end{array}
$$
Thus $\alpha \cdot (\theta)$ is unramified at $\delta$.

We have   $(-u\pi + \delta^{t-m})^{r_1 - r} \equiv \delta^{t (r_1 - r) + m(r- r_1)}$ modulo $\pi$
and hence 
$$\theta \equiv \delta^{t (r_1 - r) + m(r- r_1)} w_1^{-1}(-u)^r\pi^r\delta^s \equiv 
(-u\pi \delta^{m})^r (w_1\delta^{mr_1-s})^{-1} ~~{\rm ~~ modulo~} F_\pi^{*t}.$$
Since $w_1\delta^{mr_1-s}$ is a norm from $E/F$ and $t = [E : F]$, we have  
$$
\begin{array}{rcl}
(\alpha \cdot (\theta) ) \otimes F_\pi  & = & (E/F, \sigma, u\pi\delta^m) \cdot ((-u\pi \delta^{m})^r (w_1\delta^{mr_1-s})^{-1} ) \otimes F_\pi \\
& = &  (E/F, \sigma, u\pi\delta^m) \cdot ((-u\pi\delta^m)^{r})  \otimes F_\pi  = 0.
\end{array}
$$
In particular $\alpha \cdot (\theta)$ is unramified at $\delta$. 

Let $\gamma$ be a prime in $A$ with $(\gamma) \neq (\pi)$ and $(\gamma) \neq (\delta)$.
Since $\alpha$ is unramified on $A$ except possibly at $\pi$ and $\delta$, if 
$\gamma$ does not divide $\theta$, then $\alpha \cdot (\theta)$ is unramified at $\gamma$.
Suppose $\gamma$ divides $\theta$. Then $\gamma = -u\pi + \delta^{t-m}$.
Thus $u\pi\delta^m \equiv \delta^t$ modulo $\gamma$. Since 
$\partial_\gamma(\alpha \cdot (\theta)) = (E(\theta), \overline{\sigma}, \overline{u\pi}\overline{\delta}^m)$,
where $E(\theta)$ is the residue field of $E$ at $\theta$ and $\bar{}$ denotes the image modulo $\gamma$,
we have $\partial_\gamma(\alpha \cdot (\theta)) = (E(\theta), \overline{\sigma}, \overline{u\pi}\overline{\delta}^m)
= (E(\theta), \overline{\sigma}, \overline{\delta}^t) = 0.$

Hence $\alpha \cdot (\theta)$ is unramified on $A$. Since  $\alpha \cdot (\theta) \otimes F_\pi = 0$,
by (\ref{zero}), we have   $\alpha \cdot (\theta) = 0$. 
  \end{proof}

\section{Patching}
\label{patching}

We fix the following  data: \\
$\bullet$  $R$ a complete discrete valuation ring, \\
$\bullet$  $K$ the field of fractions of $R$, \\
$\bullet$  $\kappa$  the residue field of $R$,\\
$\bullet$  $\ell$ a prime not equal to char$(\kappa)$ and 
$n = \ell^d$ for some $d \geq 1$. \\
$\bullet$ $X$ a smooth projective geometrically integral variety over $K$, \\
$\bullet$  $F$ the function field of $X$,\\
$\bullet$  $\alpha \in H^2(F, \mu_n)$, $\alpha \neq 0$, \\
$\bullet$   $\lambda \in F^*$ with $\alpha \cdot (\lambda) = 0$, \\
 $\bullet$ $\XX$  a  normal   proper model of
$X$ over $R$  and $X_0$  the reduced special fibre of $\XX$. \\
 $\bullet$ $\PP_0$ a finite set of closed points of $X_0$ containing all the points
 of intersection of irreducible components of $X_0$.

For $x \in \XX$, let   $\hat{A}_x$ be  the completion of the  regular local ring at $x$ on $\XX$,
$F_x$  the field of fractions of $\hat{A}_x$ and $\kappa(x)$ the residue field at $x$. 
Let $\eta \in X_0$ be a
codimension zero point and $P \in X_0$ be a closed point such that $P$
is in the closure of $\eta$.  For abuse of the notation we denote the 
closure of $\eta$ by $\eta$ and  say that $P$ is a point of $\eta$.
A pair $(P, \eta)$ of a closed point $P$ and a codimension zero point of $X_0$
 is called a {\bf branch} if $P$ is in $\eta$.
Let $(P, \eta)$ be a branch.   
Let $F_{P, \eta}$ be the completion of  $F_P$ at the  discrete valuation on $F_P$ 
associated  to $\eta$. Then $F_x$ and $F_P$ are subfields of $F_{P, x}$. 
 Since $\kappa(\eta)$ is the function field of the curve $\eta$, 
 any closed point of $\eta$  gives a discrete valuation on $\kappa(\eta)$.
 The residue field  $\kappa(\eta)_P$ of $F_{P, \eta}$ is the completion of
   $\kappa(\eta)$ at the discrete valuation on $\kappa(\eta)$ given by  $P$.
   Let $\eta$ be a codimension zero point of $X_0$ and $U \subset \eta$ be a non-empty 
   open subset. Let $A_{U}$ be the ring of all those functions in $F$ which are regular 
   at every closed point of  $U$. Let $t$ be parameter in $R$. Then $t \in R_{U_\eta}$.
   Let $\hat{A}_{U}$ be the $(t)$-adic  completion of $A_U$ and $F_U$ be the field of fractions 
   of $\hat{A}_U$.   Then $F \subseteq F_{U}  \subseteq F_\eta$.
 
We begin with  the following  result, which follows from  (\cite[Theorem 9.11]{HHK3}) ( cf. proof of 
\cite[Theorem 2.4]{PS2}).
 
 \begin{prop}
  \label{splitting_alpha}
  For each irreducible component $X_\eta$ of $X_0$, let $U_\eta$ be 
a non-empty  proper open subset of
$X_\eta$ and $\PP  = X_0 \setminus \cup_\eta U_\eta$, where $\eta$ runs over 
the codimension zero points of $X_0$.  Suppose that $\PP_0 \subseteq \PP$. 
Let $L$ be a finite extension of $F$.  
Suppose that there exists $N \geq 1$ such that  
for each codimension zero point $\eta$ of $X_0$, ind$(\alpha \otimes L\otimes F_{U_\eta}) \leq N$
and for every closed point $P \in \PP$, ind$(\alpha \otimes L \otimes F_P) \leq  N$. Then 
ind$(\alpha \otimes L) \leq N$. 
\end{prop}
 
 \begin{proof} Let $\YY$ be the integral closure of  $\XX$ in $L$
 and $\phi : \YY \to \XX$ be the induced map. 
 Let  $\PP'$ be a finite set if closed points of $\YY$ containing 
   points of the intersection  of distinct irreducible curves on 
 the special fibre $Y_0$  of $\YY$ and inverse image of  $\PP$ under $\phi$.  
 Let $U$ be an irreducible component of $Y_0 \setminus \PP_0'$.
 Then $\phi(U) \subset U_\eta$ for some $U_\eta$ and there is a homomorphism
 of algebras from  $L \otimes F_{U_\eta}$ to $L_U$. (Note that $L \otimes F_{U_\eta}$ may be  
 a product  of fields).  Since  ind$(\alpha \otimes L \otimes F_{U_\eta}) \leq d$, 
 we have ind$(\alpha \otimes L_U) \leq N$.  Let $Q \in \PP'$. Suppose $\phi(Q) =  P \in \PP$.
 Then there is a homomorphism of algebras from  $L \otimes F_P$ to $L_Q$.
 (Once again note that $L \otimes F_P$ may be a product of fields). 
 Since ind$(\alpha \otimes L \otimes F_P) \leq N$,  ind$(
 \alpha \otimes L_Q) \leq N$. Suppose that $\phi(Q) \in U_\eta$ for some $U_\eta$.
 Then there is a homomorphism of algebras from  $L \otimes F_{U_\eta}$ to $L_Q $.  
 Thus ind$(\alpha \otimes L_Q ) \leq N$.
Therefore, by (\cite[Theorem 9.11]{HHK3}),  ind$(\alpha \otimes L) \leq N$. 
 \end{proof}
  
  \begin{lemma} 
  \label{openest}
  Let $\eta$ be a codimension zero point of $X_0$.
  Suppose there exists a field  extension or split extension $L_\eta/F_\eta$ of degree 
  $\ell$ and $\mu_\eta \in L_\eta$ such that \\
  1) $N_{L_\eta/F_{\eta}}(\mu_\eta) = \lambda$ \\
  2) ind$(\alpha \otimes L_\eta) < $ ind$(\alpha)$ \\
  3) $\alpha \cdot (\mu_\eta) = 0 \in H^3(L_\eta, \mu_n^{\otimes 2} )$. \\
  Then there exists a non-empty open subset $U_\eta$ of $\eta$, a split or field extension $L_{U_\eta}/
  F_{U_\eta}$ of degree $\ell$ and $\mu_{U_\eta} \in L_{U_\eta}$ such that \\
  1) $N_{L_{U_\eta}/F_{U_\eta}}(\mu_{U_\eta}) = \lambda$ \\
  2) ind$(\alpha \otimes L_\eta) < $ ind$(\alpha)$ \\
  3) $\alpha \cdot (\mu_{U_\eta}) = 0 \in H^3(L_{U_\eta}, \mu_n^{\otimes 2} )$ \\
  4) there is an isomorphism $\phi_{U_\eta} : L_{U_\eta} \otimes F_{\eta} \to L_\eta$
  with $\phi_{U_\eta}(\mu_{U_\eta} \otimes 1) \mu_\eta^{-1} \equiv 1$ modulo the radical of 
  the integral closure of $\hat{R}_\eta$  in $L_{\eta}$.\\
  Further if $L_\eta/F_\eta$ is cyclic, then $L_{U_\eta}/F_{U_\eta}$ is cyclic. 
  \end{lemma}
  
  \begin{proof} Suppose $L_\eta = \prod F_\eta$ is the split extension of degree $\ell$.
  Write $\mu_\eta = (\mu_1, \cdots , \mu_\ell)$ with $\mu_i \in F_\eta$.
  Then $\lambda = N_{L_\eta/F_\eta}(\mu_\eta) = \mu_1\cdots \mu_\ell$.  
  Since ind$(\alpha \otimes L_\eta) = $ ind$(\alpha \otimes F_\eta) < $
 ind$(\alpha)$,  by (\cite[Proposition 5.8]{HHK3}, \cite[Proposition 1.17]{KMRT}), 
 there exists a non-empty open subset $U_\eta$ of $\eta$ such that 
 ind$(\alpha) \otimes F_{U_\eta} < $ ind$(\alpha)$.
  Since $F_\eta$ is the completion of $F$ at the discrete valuation given by $\eta$,
there exist $\theta_i\in F^*$, $1 \leq i \leq \ell$, such that $\theta_i\mu_i^{-1} \equiv 1$ modulo the 
maximal ideal of $\hat{R}_\eta$.   Let $L_{U_\eta} = \prod F_{U_\eta}$
and $\mu_{U_\eta}  = (\lambda (\theta_2 \cdots \theta_\ell)^{-1}, \theta_2,  \cdots , \theta_\ell)
\in L_{U_\eta}$. Then $N_{L_{U_\eta/F_{U_\eta}}}(\mu_{U_\eta}) = \lambda$.
Since $\alpha \cdot (\theta_i) \in H^3(F_{U_\eta},\mu_n^{\otimes 2})$ and
$\alpha \cdot (\theta_i) = 0  \in H^3(F_{\eta},\mu_n^{\otimes 2})$, by 
(\cite[Proposition 3.2.2]{HHK4}), there exists a non-empty open subset $V_{\eta} \subseteq U_\eta$
such that $\alpha \cdot (\theta_i)  = 0 \in H^3(F_{V_\eta},\mu_n^{\otimes 2})$. By replacing 
$U_\eta$ by $V_\eta$, we have the required $L_{U_\eta}$ and $\mu_{U_\eta} \in L_{U_\eta}$.

 Suppose that $L_\eta/F_\eta$ is a field extension  of degree $\ell$.
 Let $F^h_\eta$ be the henselization of $F$ at the discrete valuation $\eta$.
 Then there exists a field extension $L^h_\eta/F^h_\eta$ of degree $\ell$  with 
  an isomorphism $\phi^h _\eta : L^h_\eta \otimes_{F^h_\eta} F_\eta \to L_\eta$.
  We identify $L^h$ with a subfield of $L_\eta$ through $\phi^h$. 
  Further if $L_\eta/F_\eta$ is cyclic extension, then $L^h/F^h$ is also a cyclic extension.
 Let $\tilde{\pi}_\eta \in L^h$ be a parameter.
 Then $\tilde{\pi}_\eta $ is also a parameter in $L_\eta$.Write $\mu_\eta  = u_\eta\tilde{\pi}_\eta^r$ for some 
 $u_\eta \in L_\eta$ a unit at $\eta$. Since $N_{L_\eta/F_\eta}(\mu_\eta) = \lambda$, we have
 $\lambda = N_{L_\eta/F_\eta}(u_\eta) N_{L_\eta/F_\eta}(\tilde{\pi}_\eta)$. 
 Since $u_\eta \in L_\eta$ is a unit at $\eta$, $ N_{L_\eta/F_\eta}(u_\eta) \in F_\eta$ is a unit at $\eta$.
 By (\cite[Theorem 1.10]{Artin}), there exists  
$u^h \in L^h_\eta$ such that $N_{L^h_\eta/F^h_\eta}(u^h_\eta) = N_{L_\eta/F_\eta}(u_\eta)$.
Let $\mu^h_\eta = u^h_\eta \tilde{\pi}_\eta \in L^h_\eta$. 
 Since $F^h_\eta$ is the filtered  direct limit of the fields $F_V$, where $V$  ranges over   the 
 non-empty open subset of $\eta$ (\cite[Lemma 2.2.1]{HHK4}), there exists a non-empty open 
 subset  $U_\eta$ of $\eta$, a field extension $L_{U_\eta}/F_{U_\eta}$ of degree $\ell$
 and $\mu_{U_\eta} \in L_{U_\eta}$ such that $N_{L_{U_\eta}/F_{U_\eta}}(\mu_{U_\eta}) = \lambda$
and there is an isomorphism $\phi^h_{U_\eta} : L_{U_\eta} \otimes F_\eta \simeq L^h_\eta$ with 
$\phi^h_{U_\eta}(\mu_{U_\eta}) = \mu^h_\eta$. By  shrinking $U_\eta$, we assume that 
$\alpha \cdot  (\mu_{U_\eta}) = 0 \in H^3(L_{U_\eta}, \mu_n^{\otimes 2})$ (\cite[Proposition 3.2.2]{HHK4}).
  \end{proof}
  
  \begin{lemma}
  \label{lemma}
Suppose that  for each   codimension zero point $\eta$ of   $ X_0$ 
there exist a field (not necessarily cyclic) or split  extension $L_\eta/F_\eta$ of degree $\ell$, 
 $\mu_\eta \in F_\eta$  and for every closed point 
$P$ of $X_0$ there exist a  cyclic or split extension 
$L_P/F_P$ of degree $\ell$ and  $\mu_P \in L_P $  such that for every
point $x$ of $X_0$  \\
1)  $N_{L_x / F_x }(\mu_x ) = \lambda$ \\
2) $ \alpha \cdot (\mu_x )= 0 \in H^3(L _ x , \mu_x ^{\otimes 2})$ \\
3) ind$(\alpha \otimes L_x) < $ ind$(\alpha)$\\
4)  for any branch $(P, \eta)$  there is an isomorphism $\phi_{P, \eta} :  L_\eta \otimes
F_{P, \eta} \to L_P \otimes F_{P, \eta}$ such that 
for a generator  $\sigma$ of Gal$(L_P \otimes F_{P, \eta}/F_{P,\eta})$ 
there exists $\theta_{P, \eta} \in L_P\otimes F_{P, \eta}$
such that $\phi_{P, \eta}(\mu_\eta)\mu_P^{-1} = \theta_{P,\eta}^{-\ell^d} \sigma(\theta_{P,\eta})^{\ell^d}$. \\
 Then there exist \\
 $\bullet$  a field   extension  $L/F$ of degree $\ell$ \\
$\bullet$ a non-empty open subset $U_\eta$ of $\eta$ for every codimension zero point $\eta$ of $X_0$ \\
with   $\theta_{U_\eta} \in L \otimes F_{U_\eta}$ \\
$\bullet$ for every $P \in \PP = X_0 \setminus \cup U_\eta$, $\theta_P \in L \otimes F_P$ \\ 
 such that   \\
 $1)$  ind$(\alpha \otimes L ) < $ ind$(\alpha)$ \\  
$2)$ $N_{L \otimes F_{U_\eta} / F_{U_\eta} }(\theta_{U_\eta} ) = \lambda$  and
$\alpha \cdot (\theta_{U_\eta}) = 0 \in H^3(L\otimes F_{U_\eta}, \mu_n^{\otimes 2})$ for all codimension 
zero points $\eta$ of $X_0$\\
$3)$ $N_{L \otimes F_{P } / F_P }(\theta_P ) = \lambda$
 and
$\alpha \cdot (\theta_P) = 0 \in H^3(L\otimes F_P, \mu_n^{\otimes 2})$ for all $P \in \PP$ \\
$4)$ for any branch $(P, \eta)$, $L\otimes F_{P, \eta}/F_{P, \eta}$ is cyclic or split and for a
 generator $\sigma$ of Gal$(L\otimes F_{P, \eta}/F_{P, \eta})$
there exists $\gamma_{P,\eta} \in L\otimes F_P$ such that 
$ \theta_{U_\eta} \theta _P^{-1}   = \gamma_{P, \eta}^{-\ell^d}\sigma(\gamma_{P, \eta})^{\ell^d}$. \\
Further if for each $x \in X_0$,  $L_x/F_x$ is cyclic or split, then $L/F$ is cyclic.
\end{lemma}

\begin{proof}
Let  $\eta$ be a codimension zero point of $X_0$.
By the assumption, there exist a cyclic  or split extension $L_\eta/F_\eta$ and $\mu_\eta \in L_\eta$ such that 
  $N_{L_x / F_x }(\mu_x ) = \lambda$,  $\alpha \cdot (\mu_x )= 0 \in H^3(L _ x , \mu_x ^{\otimes 2})$
  and ind$(\alpha \otimes L_\eta) < $ ind$(\alpha)$.
By (\ref{openest}), there exist a non-empty open set $U_\eta$ of   $\eta $, a cyclic  or split extension 
$L_{U_\eta}/F_{U_\eta}$ of degree $\ell$ and $\mu_{U_\eta} \in L_{U_\eta}$ such that 
  $N_{L_{U_\eta} / F_{U_\eta }}(\mu_\eta ) = \lambda$,   $\alpha \cdot (\mu_x )= 0 \in H^3(L _ x , \mu_x ^{\otimes 2})$,
  ind$(\alpha \otimes L_{U_\eta}) < $ ind$(\alpha)$, $\phi_\eta : L_{U_\eta} \otimes F_\eta \to L_\eta$ an 
  isomorphism and $\phi_\eta(\mu_{U_ \eta}) =  \mu_\eta$.
  By shrinking $U_\eta$, if necessary, we assume that $\PP_0 \cap U_\eta = \emptyset$.
 
 Let $\PP = X_0 \setminus \cup_\eta U_\eta$ and $P \in \PP$. Then, by the assumption 
 we have a cyclic or split extension $L_P/F_P$ of degree $\ell$ and  for every branch 
 $(P, \eta)$ there is an isomorphism 
 $\phi_{P, \eta} : L_\eta \otimes F_{P, \eta} \to L_P \otimes F_{P, \eta}$. 
 Thus  $\phi_{P, U_{\eta}}  = \phi_{P, \eta} (\phi_\eta \otimes 1) : L_{U_\eta} \otimes F_\eta \otimes 
 F_{P, \eta} \to L_P \otimes F_{P, \eta}$ is an isomorphism. Thus, by (\cite[Theorem 7.1]{HH}), there exists 
 an extension  $L/F$ of degree $\ell$ with isomorphisms 
 $\phi_{U_\eta} : L \otimes F_{U_\eta} \to L_{U_\eta}$ for all codimension zero points $\eta$ of $X_0$ 
and $\phi_P : L \otimes F_P \to L_P$ for all $P \in \PP$ such that the following commutative diagram 
$$
\begin{array}{ccc}
L \otimes F_{U_\eta} \otimes F_{P, \eta} & \buildrel{\phi_{U_\eta} \otimes 1} \over{\to} & 
L_{U_\eta} \otimes F_{\eta} \otimes F_{P, \eta} \\
\downarrow & & \downarrow \phi_{P, U_\eta} \\
L \otimes F_{P} \otimes F_{P, \eta} & \buildrel{\phi_{P} \otimes 1} \over{\to} & 
L_{P} \otimes   F_{P, \eta} 
\end{array}
$$
where the vertical arrow on the left side is the natural map. Further  if each $L_x/F_x$ is cyclic for 
all $x \in X_0$, then  $L/F$ is cyclic (\cite[Theorem 7.1]{HH}).

Since ind$(\alpha \otimes L  \otimes F_{U_\eta}) < $ ind$(\alpha)$  for all codimension zero points of 
$X_0$ and ind$(\alpha \otimes L \otimes F_P) < $ ind$(\alpha)$ for all $P \in \PP$, 
by (\ref{splitting_alpha}), ind$(\alpha \otimes L) < $ ind$(\alpha)$. 
In particular $L$ is a field.   

For every codimension zero point $\eta$ of $X_0$,
let $\theta_{U_\eta} = (\phi_{U_\eta})^{-1}(\mu_{U_\eta}) \in L \otimes F_{U_\eta}$
and for every $P \in \PP$, let $\theta_P = (\phi_P)^{-1}(\mu_P) \in L \otimes F_P$.
Since $\phi_{U_\eta}$ and $\phi_P$ are isomorphisms,  we have the required properties.
  \end{proof}
 
\begin{prop}
\label{patching_L}
 Suppose that   for each    point $x $ of $ X_0$  there exist a cyclic or split extension   
$L_x/F_x$ of degree $\ell$ and  $\mu_x  \in L_x $  such that \\
1)  $N_{L_x / F_x }(\mu_x ) = \lambda$ \\
2) $ \alpha \cdot (\mu_x )= 0 \in H^3(L _ x , \mu_x ^{\otimes 2})$ \\
3) ind$(\alpha \otimes L_x) < $ ind$(\alpha)$\\
4)   for any branch $(P, \eta)$  there is an isomorphism $\phi_{P, \eta} :  L_\eta \otimes
F_{P, \eta} \to L_P \otimes F_{P, \eta}$ such that 
for generator  $\sigma$ of Gal$(L_P \otimes F_{P, \eta}/F_{P,\eta})$ 
there exists $\theta_{P, \eta} \in L_P\otimes F_{P, \eta}$
such that $\phi_{P, \eta}(\mu_\eta)\mu_P^{-1} = \theta_{P,\eta}^{-\ell^d} \sigma(\theta_{P,\eta})^{\ell^d}$.  \\
Then there exist a cyclic   extension $L$ of degree $\ell$ and  
$\mu \in L^*$ such that \\
$\bullet$ $N_{L/F}(\mu) = \lambda$ and \\
$\bullet$ $  \alpha \cdot (\mu) =  0 \in H^3(L, \mu_n^{\otimes 2})$ \\
$\bullet$ ind$(\alpha \otimes L) < $ ind$(\alpha)$.   
\end{prop}

\begin{proof}    
Let $L/F$, $U_\eta$, $\PP$, $\theta_{U_\eta}$ and $\theta_P$ be as in (\ref{lemma}).
Since each $L_x/F_x$ is cyclic or split, $L/F$ is cyclic.
Let $\sigma$ be a generator of  Gal$(L/F)$.
Let $(P, \eta)$ be a branch. 
 By (\ref{lemma}), there exists 
  $\gamma_{(P, \eta) } \in L\otimes F_{P, \eta}$   
  such that
$\mu_{U_\eta}\mu_P^{-1} = \gamma_{P, \eta}^{-\ell^d}\sigma(\gamma_{P, \eta}^{\ell^d})$.
Applying  (\cite[Theorem 3.6]{HHK1}) for the rational group $GL_1$, 
 there exist $\gamma_{U_\eta} \in   L\otimes F_{U_\eta}$ and 
 $\gamma_P \in   L \otimes F_P$ for every codimension zero point $\eta$ of $X_0$ and
 $P \in \PP$ such that for every branch $(P, \eta)$, $\gamma_{P, \eta} = \gamma_{U_\eta} \gamma_P$.

  Let $\mu_{U_\eta}' = \mu_{U_\eta}  \gamma_{U_\eta}^{\ell^d} \sigma
 ( \gamma_{U_\eta}^{-\ell^d}) \in L \otimes F_{U_\eta}$ and $\mu_P' =
 \mu_P \gamma_P^{-\ell^d} \sigma( \gamma_P^{\ell^d}) \in L \otimes F_P$. If $(P, \eta)$ is a branch,
  then we have 
 $$
 \begin{array}{rcl}
\mu_{U_\eta}' &  =  &   \mu_{U_\eta}  \gamma_{U_\eta}^{\ell^d} \sigma
 ( \gamma_{U_\eta}^{-\ell^d})  \\
& = & \mu_P \theta_{P, \eta}^{-\ell^d}\sigma(\theta_{P, \eta}^{\ell^d}) \gamma_{U_\eta}^{\ell^d} \sigma
 ( \gamma_{U_\eta}^{-\ell^d})   \\
 & = & \mu_P  \gamma_P^{-\ell^d}\sigma( \gamma_P^{\ell^d} )   \\
& = &  \mu_P' \in L \otimes F_{P, \eta}.
\end{array}
$$ 
Hence,  by (\cite[Proposition 6.3]{HH}), there exists $\mu \in L$
such that $\mu = \mu_{U_\eta}'$ and $\mu = \mu'_P$ for every codimension zero point $\eta$ 
of $X_0$ and $P \in \PP$. 
Clearly  $N_{L/F}(\mu) = \lambda$ over $F$. 
Let $P \in \PP$. Since $ \alpha  \cdot (\mu_P) =  0 $ and 
$\alpha \cdot (\gamma_P^{\ell^d})  = 0$,  $ \alpha \cdot (\mu) = 
0 \in H^3(L \otimes F_P, \mu_n^{\otimes 2})$. Similarly  $\alpha \cdot (\mu)  =
0 \in H^3(L \otimes F_{U_\eta}, \mu_n^{\otimes 2})$ for every codimension zero point $\eta$ of $X_0$. 
Hence, by (\cite[Theorem 3.1.5]{HHK4}), $\alpha \cdot (\mu)   = 0$ in $H^3(L, \mu_n^{\otimes 2})$.  
\end{proof}

\begin{prop}
\label{patching_L_mu}
 Suppose that  for each   codimension zero point $\eta$ of   $ X_0$ 
there exist a field (not necessarily cyclic) or split  extension $L_\eta/F_\eta$ of degree $\ell$, 
 $\mu_\eta \in F_\eta$  and for every closed point 
$P$ of $X_0$ there exist a  cyclic or split extension 
$L_P/F_P$ of degree $\ell$ and  $\mu_P \in L_P $  such that for every
point $x$ of $X_0$  \\
1)  $N_{L_x / F_x }(\mu_x ) = \lambda$ \\
2) $ \alpha \cdot (\mu_x )= 0 \in H^3(L _ x , \mu_x ^{\otimes 2})$ \\
3) ind$(\alpha \otimes L_x) < $ ind$(\alpha)$\\
4) for any branch $(P, \eta)$  there is an isomorphism $\phi_{P, \eta} :  L_\eta \otimes
F_{P, \eta} \to L_P \otimes F_{P, \eta}$ such that 
for a generator  $\sigma$ of Gal$(L_P \otimes F_{P, \eta}/F_{P,\eta})$ 
there exists $\theta_{P, \eta} \in L_P\otimes F_{P, \eta}$
such that $\phi_{P, \eta}(\mu_\eta)\mu_P^{-1} = \theta_{P,\eta}^{-\ell^d} \sigma(\theta_{P,\eta})^{\ell^d}$. \\
Then there exist a field   extension $N/F$ of degree coprime to $\ell$, a field 
 extension $L/F$ of degree $\ell$ and   $\mu \in (L \otimes N)^*$ such that \\
$\bullet$ $N_{L \otimes N /N}(\mu) = \lambda$ and \\
$\bullet$ $  \alpha \cdot (\mu) =  0 \in H^3(L\otimes N, \mu_n^{\otimes 2})$ \\
$\bullet$ ind$(\alpha \otimes L) < $ ind$(\alpha)$.   
\end{prop}

\begin{proof}  
Let $L/F$, $U_\eta$, $\PP$, $\theta_{U_\eta}$, $\theta_P$ and $\gamma_{P, \eta}$ be as in (\ref{lemma}).
 Since $L/F$ is a degree $\ell$ extension, there exists a field extension $N/F$ of degree
 coprime to $\ell$ such that $L \otimes N/N$ is a cyclic extension.

 Let $\YY$ be the integral closure of $\XX$ in $N$   and $Y_0$ the reduced special fibre of $\YY$. 
  Let $\phi : Y_0 \to X_0$ be the induced morphism.
 Let $y \in Y_0$ and $x = \phi(y) \in X_0$.  Then the inclusion $F \subset N$,  induces an 
 inclusion $F_x \subset N_y$. Let $L'_y = L  \otimes_F F_x \otimes_{F_x} N_y$.
 Since $L\otimes N/ N $ is  a cyclic   extension of degree $\ell$, 
  $L'_y/N_y$ is either cyclic or split extension of degree $\ell$.

 Let $\eta' \in Y_0$ be a codimension zero point. Then $\eta = \phi(\eta') \in X_0$ is a codimension zero 
 point.  Then $F_\eta \subset FL_{\eta'}$, $L \otimes F_{U_\eta}  \subset 
 L \otimes F_\eta$ and $\theta_{U_\eta} \in L \otimes F_{\eta}$.
 Let $\mu_{\eta'} =  \theta_{U_\eta} \otimes 1 \in L \otimes F_\eta \otimes_{F_\eta} N_{\eta'} = L'_{\eta'}$.
 
 Let $Q\in Y_0$ be a closed point and $P = \phi(Q) \in X_0$. Then $P$ is a closed point of $X_0$
 and $F_P \subset N_Q$.  
 Suppose that $P \in U_\eta$ for some codimension zero point $\eta$ of $X_0$.
 Then $F_{U_\eta} \subset F_P$, $L \otimes F_{U_\eta}  \subset 
 L \otimes F_P$ and $\theta_{U_\eta} \in L \otimes F_P$.
 Let $\mu_{Q} =  \theta_{U_\eta} \otimes 1 \in L \otimes F_P \otimes_{F_P} N_Q  = L'_Q$.
 Suppose that  $P$ is not in  $U_\eta$ for any codimension zero point $\eta$ of $X_0$.
 Let $\mu_Q = \theta_P \otimes 1 \in L \otimes F_P\otimes_{F_P} N_Q$.
 
Let $y \in Y_0$ and $x = \phi(x) \in X_0$.
 Since $N_{L_x/F_x}(\mu_x) = \lambda$, $\alpha \cdot (\mu_x) = 0 \in H^3(L_x, \mu_n^{\otimes 2})$ and
 ind$(\alpha \otimes F_x) < $ ind$(\alpha)$, it follows that  
 $N_{L'_y/N_y}(\mu_y) = \lambda$, $\alpha \cdot (\mu_y) = 0 \in H^3(L'_y, \mu_n^{\otimes 2})$ and
 ind$(\alpha \otimes L'_y) < $ ind$(\alpha)$.

  Let $(Q, \eta')$ be a branch in $Y_0$ and $P = \phi(Q)$, $\eta = \phi(\eta')$.
  Then $(P, \eta)$ is a branch in $X_0$. The isomorphism $\phi_{P, \eta} : L_\eta \otimes F_{P, \eta}
  \to L_P \otimes F_{P, \eta}$ induces an isomorphism $\phi'_{Q,\eta'} : L'_{\eta'} \otimes N_{Q, \eta'}
  \to L'_Q \otimes N_{Q, \eta'}$. By the choice of $\mu_{\eta'}$ and $\mu_Q$ it follows that 
  for any generator $\sigma$ of Gal$(L'_Q \otimes N_{Q, \eta'}/N_{Q,\eta'})$ there exists $\theta_{Q, \eta'}$
  such that 
$\phi'_{Q, \eta'}(\mu_{\eta'})\mu_Q^{-1} = \theta_{Q,\eta'}^{-\ell^d} \sigma(\theta_{Q,\eta'})^{\ell^d}$.  
  Thus, by (\ref{patching_L}), there exists a cyclic   extension $L'/N$ and $\mu' \in L'$ such that 
 $N_{L'/N}(\mu')=  \lambda$, ind$(\alpha \otimes L') < $ ind$(\alpha \otimes N)$  and $
\alpha \cdot (\mu') = 0\in H^3(L', \mu_n^{\otimes 2})$.
 By the construction we have $L' = L \otimes N$.
  \end{proof}

\section{Types of points, special points and type 2 connections }
\label{types_of_points}

Let $F$, $\alpha \in H^2(F, \mu_n)$, $\lambda \in F^*$ with $\alpha \cdot (\lambda) = 0 \in
 H^3(F, \mu_n^{\otimes 2})$, $\XX$ and $X_0$ be as in (\S \ref{patching}). 
Further assume that  \\
$\bullet$ $\XX$ is regular 
 such that ram$_\XX(\alpha) \cup$ supp$_\XX(\lambda) \cup X_0$ 
 is a union of regular curves with normal
crossings. \\
$\bullet$ the intersection of any two distinct irreducible curves in $X_0$ is at most
a one closed point. \\
We fix the following notation.\\
$\bullet$  $\PP$ is the set of points of 
intersection of distinct irreducible curves  in $X_0$.\\
$\bullet$ $\OO_{\XX, \PP}$ is the semi-local ring at the points of $\PP$ on $\XX$.\\  
$\bullet$  if    a  codimension zero points $\eta$ of $X_0$ containing a closed point $P \in \PP$, 
then $\pi_{\eta} \in \OO_{\XX, \PP}$ is a   prime defining $\eta$ on $\OO_{\XX, \PP}$. 

Let  $\eta$ be   a codimension zero point  of   $X_0$
For the rest of this paper,  let   $(E_{\eta}, \sigma_\eta)$ denote   the 
   lift of the residue  of $\alpha$ at $\eta$. Since $\alpha \in H^2(F, \mu_n)$ with $n$ a power of $\ell$, 
 $[E_\eta : F_\eta]$ is a power of $\ell$. 
 If  $\alpha$ is unramified at $\eta$, then $E_\eta = F_\eta$ and  let $M_\eta = F_\eta$.
 If $\alpha$ is ramified at $\eta$, then $E_\eta \neq F_\eta$ and  there is a unique subextension  
of $E_\eta$ of degree $\ell$ and we denote it by  $M_\eta$.

\begin{remark} 
\label{index_M} 
Let $\eta$ be a codimension zero point of $X_0$. Suppose $\alpha$ is ramified at $\eta$. Since 
ind$(\alpha \otimes F_\eta) = $ ind$(\alpha \otimes E_\eta)[E_\eta : F_\eta]$ (cf. \ref{dvr_ind})
and $M_\eta \subset E_\eta$, it follows that ind$(\alpha \otimes M_\eta) < $ ind$(\alpha)$.
\end{remark}

We divide the codimension zero points $\eta$ of $X_0$ as follows: 
\vskip 3mm

\noindent
{\bf Type 1}: $\nu_\eta(\lambda)$ is coprime to $\ell$ and ind$(\alpha \otimes F_\eta)  = $ ind$(\alpha)$ \\
{\bf Type 2}: $\nu_\eta(\lambda)$ is coprime to $\ell$ and ind$(\alpha \otimes F_\eta)  < $ ind$(\alpha)$ \\
{\bf Type 3}: $\nu_\eta(\lambda) = r\ell$,    
$r \alpha \otimes E_\eta \neq 0$     and  ind$(\alpha \otimes F_\eta)  = $ ind$(\alpha)$ \\
{\bf Type 4}: $\nu_\eta(\lambda) = r\ell$,  
  $r \alpha \otimes E_\eta \neq 0$  
  and ind$(\alpha \otimes F_\eta)  < $ ind$(\alpha)$  \\
{\bf Type 5}: $\nu_\eta(\lambda) = r\ell$,  
  $r \alpha \otimes E_\eta =  0$  
   and ind$(\alpha \otimes F_\eta) =  $ ind$(\alpha)$\\
{\bf Type 6}:  $\nu_\eta(\lambda) = r\ell$,  
  $r \alpha \otimes E_\eta =  0$     and ind$(\alpha \otimes F_\eta) < $ ind$(\alpha)$. \\

Let $P$ be a closed point of $\XX$. Suppose $P$ is the point of intersection of two distinct codimension 
zero points  $\eta_1$  and $\eta_2$  of $X_0$.  
We say that the point $P$ is a \\
 1)  {\bf special point of type I} if $\eta_1$ is of type 1 and $\eta_2$ is of type 2, \\
 2)  {\bf special point of type II} if $\eta_1$ is of type 1 and $\eta_2$ is of type 4, \\
 3)  {\bf special point of type III } if $\eta_1$ is of type 3 or  5 and $\eta_2$ is of type 4, \\
 4) {\bf special point of type IV}  if $\eta_1$ is of type 1, 3 or 5 and   $\eta_2$ is of type  5
 with $M_{\eta_2} \otimes F_{P, \eta_2}$  not a field. \\

\begin{lemma} 
\label{hot_point}
Suppose that $\eta_1$ and $\eta_2$ are two  distinct codimension zero points of $X_0$
 and   $P$  a  point of intersection  of $\eta_1$ and $\eta_2$.  Suppose that $\alpha$ is 
 ramified at $\eta_1$.  Let $(E_{\eta_1}, \sigma_1)$ be the lift of residue of 
 $\alpha$ at $\eta_1$.  If $E_{\eta_1} \otimes F_{P, \eta_1}$ is not a field, 
 then  ind$(\alpha \otimes  F_P) < $ ind$(\alpha)$. 
\end{lemma}

\begin{proof}   Suppose that $E_{\eta_1} \otimes F_{P, \eta_1}$ is not a field.
Since $E_{\eta_1}/F_{\eta_1}$ is a cyclic extension, $E_{\eta_1} \otimes F_{P, \eta_1} 
\simeq \prod E_{\eta_1,P}$ with  $[E_{\eta_1, P} : F_{P, \eta_1}] < [E_{\eta_1} : F_{\eta_1}]$. 
We have $(E_{\eta_1}, \sigma_1, \pi_{\eta_1}) \otimes F_{P, \eta_1} = 
(E_{\eta_1, P}, \sigma_1, \pi_{\eta_1})$ (cf. \S 2).   

Write $\alpha \otimes F_{\eta_1} = \alpha_1 + (E_{\eta_1}, \sigma_1, \pi_{\eta_1})$ as in (\ref{rbc}).
Then $\alpha \otimes F_{P, \eta_1} = \alpha_1 \otimes F_{P, \eta_1} + (E_{\eta_1, P}, \sigma_1, \pi_{\eta_1})$.
By (\ref{dvr_ind}), we have ind$(\alpha \otimes F_{\eta_1}) = $ ind$(\alpha_1 \otimes E_{\eta_1}) [E_{\eta_1} : F_{\eta_1}]$. 
We have 
$$
\begin{array}{rcl}
{\rm ind}(\alpha \otimes F_{P, \eta_1})  & \leq &  {\rm ind}(\alpha_1 \otimes 
E_{\eta_1, P}) [E_{\eta_1, P} : F_{P, \eta_1}] \\
& \leq & {\rm ind}(\alpha_1 \otimes 
E_{\eta_1}) [E_{\eta_1, P} : F_{P, \eta_1}] \\
& < & {\rm ind}(\alpha_1 \otimes 
E_{\eta_1}) [E_{\eta_1} : F_{\eta_1}] \\
&  = & {\rm ind}(\alpha \otimes F_{\eta_1}).
\end{array}
$$ 
Thus, by (\ref{local_index}), ind$(\alpha \otimes F_P) < $ ind$(\alpha)$. 
\end{proof}

\begin{lemma}
\label{blowing_up}
Let $\eta \in X_0$ be a point of codimension zero and $P$ a closed point on $\eta$.
Let $\XX_P \to \XX$ be the blow-up at $P$ and $\gamma$  the exceptional curve in 
$\XX_P$. If  $E_\eta \otimes F_{P, \eta}$ not a field or 
$\eta$ is of type 2, 4 or 6,  then $\gamma$ is of type 2, 4 or 6.
\end{lemma}

\begin{proof}  If  $E_\eta \otimes  F_{P, \eta}$ is not a field, then  by (\ref{hot_point}), 
ind$(\alpha \otimes F_P) < $ ind$(\alpha)$. If $\eta$ if of type 2, 4 or 6, then ind$(\alpha \otimes F_\eta)
< $ ind$(\alpha)$ and hence by (\ref{local_index}), ind$(\alpha \otimes F_P) < $ ind$(\alpha)$.
Since $F_P \subset F_\gamma$, we have ind$(\alpha \otimes F_\gamma) \leq $ ind$(\alpha \otimes F_P) 
< $ ind$(\alpha)$. Hence $\gamma$ is of type 2, 4 or 6.
\end{proof}

\begin{lemma}
\label{blowing_up_at_122}  
Let $\eta_1$  and $\eta_2$  be two distinct   codimension zero points of $X_0$  intersecting at 
  a closed point   $P$.   Suppose that $\eta_1$ is of type 1 or 2 and $\eta_2$ is of type 2.
Then there exists a sequence of blow-ups $\psi : \XX' \to \XX$ such that if 
$\tilde{\eta}_i$   are the strict transforms of $\eta_i$, then \\
1)  $\psi : \XX' \setminus \psi^{-1}(P) \to \XX \setminus \{ P \}$ is an isomorphism \\
2) $\psi^{-1}(P)$ is the union of irreducible regular  curves $ \gamma_1, \cdots ,\gamma_m$ \\
3) $\tilde{\eta}_1 \cap \gamma_1 = \{ P_0 \}$, $\gamma_i \cap \gamma_{i+1}   = \{ P_i \}$, $\gamma_{m} 
\cap \tilde{\eta}_2 = \{ P_m \}$,
$\tilde{\eta} \cap \gamma_i = \emptyset$ for all $i > 1$, $\tilde{\eta}_2\cap \gamma_i = 
\emptyset$ for all $i< m$,
$\tilde{\eta}_1 \cap \tilde{\eta}_2 = \emptyset$, $\gamma_i \cap \gamma_j = \emptyset$ for all $ j \neq  i +1$,  \\
 4)   $\gamma_1$ and $\gamma_m$ are of type 6 and $\gamma_i$,
$ 1 < i < m$ are of type 2,  4 or 6,    \\
5) $\psi^{-1}(P)$ has no special points.  
 \end{lemma}

\begin{proof} Let $\XX_P \to \XX$ be the blow-up of $\XX$ at $P$ and $\gamma$ the 
 exceptional curve in $\XX_P$.  
 Let $\tilde{\eta}_i$ be the   strict transform of  $\eta_i$.
 Then $\tilde{\eta}_1$ intersects $ \gamma$  only at one point $P_0$ and $\tilde{\eta}_2$ intersects $\gamma$ 
  at only one point $P_1$.  Since $\eta_2$ is of type 2, by (\ref{blowing_up}), 
  $\gamma$ is of type 2, 4 or 6 and hence $P_1$ is not a special point.
   
Let  $s_1 = \nu_{\eta_1}(\lambda)$,
 $s_2 = \nu_{\eta_2}(\lambda)$. Then   $\nu_{\gamma}(\lambda) = s_1 + s_2$. 
Suppose $s_1+ s_2 = \ell^{d+1}r_0$  for some integer $r_0$, where $\ell^d = $ ind$(\alpha)$. 
Since $\ell^d \alpha = 0$,   $\ell^{d}r_0 \alpha  = 0$. Thus, $\gamma$ is of type 6. Hence $P_0$ is not a special 
point and $\XX_P$ has all  the required properties.

Suppose $s_1 + s_2  =    \ell^t r_0$ with   $ t \leq d $ and $r_0$ coprime to $\ell$. 
Then, blow-up the points $P_0$ and $P_1$ and let $\gamma_1$ and $\gamma_2$ be the 
exceptional curves in this blow-up. Then we have $\eta_{\gamma_1}(\lambda) = 2r + s$
and $\eta_{\gamma_2}(\lambda) = r + 2s$.  If     $2r + s$ is  
not of the form $\ell^{d+1}r_1$  for some $r_1 \geq 1$, then blow-up,  the point of intersection of 
the strict transform of  $\eta_1$ and $\gamma_1$. If  
$r + 2s$ is not of the form 
$\ell^{d+1}r_2$ for some  $r_ 2 \geq 1$,  then blow-up,  the point of intersection of 
the strict transform of   $\eta_2$ and $\gamma_2$.  
Since $r$ and $s$ are coprime to $\ell$, there exist $i$ and $j$ such that 
$ir  + s = \ell^{d+1}r$ and $r + js = \ell^{d+1}r'$ for some $r, r' \geq 1$. Thus,
  we get the required finite sequence of blow-ups.   
\end{proof}

\begin{prop} 
\label{choice_of_model_at_closed_points}
There exists a regular proper model  of $F$ with no special points. 
\end{prop}

\begin{proof}  Let $P \in  \PP$. Then there exist two codimension zero 
points  $\eta_1$ and $\eta_2$   of $X_0$ intersecting at $P$.

Suppose that $P$ is a special point of type I.
Let $\psi: \XX'  \to \XX$ be a sequence of   blow-ups as in (\ref{blowing_up_at_122}).
Then there are no special points in $\psi^{-1}( P )$.
Since there are only finitely many special points in $\XX$,  replacing 
$\XX$ be a finite sequence of blow ups at all special points of type I,
we assume that $\XX$ has no special points of type I.

Suppose $P$ is a special point of type II.
Without loss of generality we assume that, 
 $\eta_1$ is of type 1 and $\eta_2$ is of type 4.
Let $\XX_P \to \XX$ be the blow-up of $\XX$ at $P$ and $\gamma$ the 
 exceptional curve in $\XX_P$.  Since $\eta_2$ is of type 4,  by (\ref{blowing_up}), 
 $\gamma$ is of type 2, 4 or 6.  Since $\eta_1$ is of type 1 and $\eta_2$ is of type 4, 
 $\nu_{\eta_1}(\lambda) $ is coprime to $\ell$ and 
$ \nu_{\eta_2}(\lambda)$ is divisible by $\ell$.
 Since $\nu_\gamma(\lambda) = \nu_{\eta_1}(\lambda) 
 + \nu_{\eta_2}(\lambda)$, $\nu_\gamma(\lambda)$ is coprime to $\ell$ and hence 
 $\gamma$ is of type 2.
  Let $\tilde{\eta}_i$ be the   
 strict transform of $\eta_i$ in $\XX_P$.
 Then $\tilde{\eta}_i$ and $\gamma$ intersect at only one point $Q_i$. 
  Since   $\gamma$ is of type 2, $Q_1$ is a special point of type I and $Q_2$ is
not a special point.  Thus, 
as above, by replacing 
$\XX$ by a sequence of  blow-ups   of $\XX$, we assume that 
$\XX$ has no special points of type I or II.

 Suppose $P$ is a special point of type III. 
 Without loss of generality assume that 
$\eta_1$ is of  type 3 or  5 and $\eta_2$ of type 4. 
 Let $\XX_P \to \XX$ be the blow-up of $\XX$ at $P$, $\gamma$,     $\tilde{\eta}_i$,   
 and   $Q_i$ be as above.  Since $\eta_2$ is of type 4,  by (\ref{blowing_up}),
 $\gamma$ is of type 2, 4 or 6.
 Since $\nu_{\eta_1}(\lambda)$ and $\nu_{\eta_2}(\lambda)$ are divisible by $\ell$,
 $\nu_{\gamma}(\lambda) = \nu_{\eta_1}(\lambda) + \nu_{\eta_2}(\lambda)$ is divisible by $\ell$. 
 Thus $\gamma$ is of type 4 or 6. Hence  $Q_2$ is not a  special point.
By (\ref{local-cyclic}), $\alpha \otimes F_P = (E_P, \sigma, u\pi_{\eta_1}^{d_1}\pi_{\eta_2}^{d_2})$  
for some cyclic extension $E_P/F_P$ and $u \in \hat{A}_P$ a unit and at least one of $d_i$ is 
coprime to $\ell$ (in fact equal to 1). 
In particular, $\alpha \otimes F_P $ is split by the extension $F_P(\sqrt[m]{u\pi_{\eta_1}^{d_1}\pi_{\eta_2}^{d_2}} ~)$,
 where $m$ is the degree of $E_P/F_P$ which is a power of $\ell$.
Suppose $d_1 + d_2$ is coprime to $\ell$. Since $\nu_\gamma(\pi_{\eta_1}^{d_1}\pi_{\eta_2}^{d_2}) =   d _1 + d_2$,
$F_P(\sqrt[m]{u\pi_{\eta_1}^{d_1}\pi_{\eta_2}^{d_2}}~)$  is totally ramified at $\gamma$. Thus, by (\ref{not_type_6}), 
$\gamma$ is  of type 6. Hence $Q_1$ is not a special point. 
Suppose that $d_1 + d_2$ is divisible by  $\ell$.
Let $\pi_\gamma$ be a prime defining $\gamma$ at $Q_1$.
Then, we have $u\pi_{\eta_1}^{d_1} \pi_{\eta_2}^{d_2} = w_1\pi_{\eta_1}^{d_1}\pi_\gamma^{d_1 + d_2}$ 
 for some unit  $w_1$  at  $Q_1$. Since one of $d_i$ is coprime to $\ell$ and $d_1  + d_2$ is 
 divisible by $\ell$, $d_i$ are not divisible by $\ell$. In particular $2d_1  + d_2$ is coprime to 
 $\ell$. Let $\XX_{Q_1}$ be the blow-up of $\XX_P$ at $Q_1$ and $\gamma'$ be the generic point 
 of the exceptional curve in $\XX_{Q_1}$. Then $ 
\nu_{\gamma'}( u\pi_{\eta_1}^{d_1} \pi_{\eta_2}^{d_2}) = \nu_{\gamma'}(w_1\pi_{\eta_1}^{d_1}\pi_\gamma^{d_1 + d_2}) = 2d_1 + d_2$.
Since $2d_1 + d_2$ is coprime to $\ell$,   once again by (\ref{not_type_6}), $\gamma'$ is  of type 6. 
In particular  no  point on the exceptional curve in $\XX_{Q_1}$  is  a special point.
Thus,  replacing $\XX$ by a sequence of blow-ups, we assume that $\XX$ has no special 
points of type I, II or III.

 Suppose $P$ is  a  special point of type IV. 
Without loss of generality assume that, 
 $\eta_1$ is of type 1, 3 or 5  and $\eta_2$ is of type  5, with   
 $M_{\eta_2} \otimes F_{P, \eta_2}$ is not a field.
 Let $\XX_P \to \XX$ be the blow-up of $\XX$ at $P$ and    $\gamma$,   $\tilde{\eta}_i$,
 $Q_i$ be as above. 
 Since $M_{\eta_2} \otimes F_{P, \eta_2}$ is not a field, by (\ref{blowing_up}), $\gamma$ 
 is  of  type 2, 4 or 6.  If $\gamma$ is of type 6, then $Q_1$ and $Q_2$ are not  special points. 
  Suppose $\gamma$ is of type 2 or 4.
  Then $Q_1$ and $Q_2$ are special points of type I, II or III.  Thus, 
as above, by replacing 
$\XX$ by a sequence of  blow-ups   of $\XX$, we assume that 
$\XX$ has no special points. 
 \end{proof}

Let $\eta$ and $\eta'$ be two codimension zero points  of $X_0$ (may not be 
distinct). We say that  there is a {\bf type 2 connection} from $\eta$ to $\eta'$  if  one of the
following holds \\
$\bullet$  one of $\eta$ or $\eta'$ is of type 2 \\
$\bullet$ there exist distinct codimension zero points  $\eta_1, \cdots , \eta_n$ of $X_0$ of type 2
such that $\eta$ intersects  $\eta_1$,  $\eta'$ intersects $\eta_n$,   $\eta_{i} $ intersects 
$\eta_{i + 1}$ for all $1 \leq i \leq n-1$, $\eta$ does not intersect $\eta_i$ for $i > 1$,
$\eta'$ does not intersect $\eta_i$ for $i < n$ and $\eta_i$ does not intersect $\eta_j$
for $j \neq i +1$.

\begin{prop}
\label{choice_of_model}
There exists a regular proper model  $\XX$ of $F$ such that  \\
1) $\XX$ has   no special points \\
2)  if $\eta_1$ and $\eta_2$ are   two (not necessarily  distinct)  codimension 
zero points  of $X_0$ with $\eta_1$ is of  type 3 or 5 and $\eta_2$ is of type 3, 4 or 5, then 
there is no   type 2 connection between $\eta_1$ and $\eta_2$. 
\end{prop}
 
 \begin{proof}  
 Let $\XX$ be a regular proper model with no special points (\ref{choice_of_model_at_closed_points}).
 Let $m(\XX)$ be the number of  type 2 connections between a point of type 3  or 5  and 
 a point of  type   3, 4 or 5. 
 We prove the proposition by  induction on $m(\XX)$.  
  Suppose $m(\XX)  \geq 1$. We show that there is a sequence of  blow-ups $\XX'$ of 
 $\XX$ with no special points and $m(\XX') < m(\XX)$.

Let $\eta$ be a codimension zero point of $X_0$ of type 3 or 5  and 
$\eta'$ a  codimension zero point of $X_0$ of types 3, 4 or 5. 
Suppose $\eta$ and $\eta'$ have a type 2 connection.
Then there exist   $\eta_1, \cdots , \eta_n$ codimension zero points of $X_0$ 
of type 2  with $\eta$ intersecting $\eta_1$, $\eta'$ intersecting $\eta_n$ and
$\eta_i$ intersecting  $\eta_{i +1}$ for $i = 1, \cdots, n-1$. 

Suppose $n = 1$. Let $Q$ be the point of 
the intersection of $\eta$ and $\eta_1$. 
Let $\XX_Q \to \XX$ be the blow-up of $\XX$ at $Q$ and $\gamma$ the 
 exceptional curve in $\XX_Q$.  
  Since $\eta_1$ is of type 2,   by (\ref{blowing_up}), $\gamma$ is of type 2, 4 or 6.
   Since  $\eta$ is  of type 3  or  5 and $\eta_1$ is of type 2, 
 $\ell$  divides $\nu_{\eta}(\lambda)$ 
and $\ell$ does not divide  $\nu_{\eta_1}(\lambda)$. 
Since $\nu_{\gamma}(\lambda) = \nu_{\eta}(\lambda) + \nu_{\eta_1}(\lambda)$, 
$\nu_{\gamma}(\lambda)$ is  not divisible by $\ell$ and hence  $\gamma$ is  of  type 2.
Let $\tilde{\eta}$ and $\tilde{\eta}_1$ be the strict transform of $\eta$  and $\eta_1$
 in $\XX_Q$.   Since $\gamma$ is a point of type 2,  the points of intersection of  $\tilde{\eta}$ and  $\tilde{\eta}_1$
 with  $\gamma$  are not special points.  Hence $\XX_Q$ has no special points.  
By replacing $\XX$ by $\XX_Q$ we assume that $ n \geq 2$ and $\XX$ has no special points.

 Let $P$ be the point of intersection of $\eta_1$ and $\eta_2$.
 Let $\XX'$ be as in (\ref{blowing_up_at_122}).  Then $\XX'$ has no special points and 
 all the exceptional curves in $\XX'$ are of 
 type 2, 4 or 6 and the exceptional curves which intersect the strict transforms of $\eta_1$ and $\eta_2$
 are of type 6. In particular the number of type 2 connections between the 
 strict transforms of $\eta$ and $\eta'$ is one less than the number of type 2 connections between 
 $\eta$ and $\eta'$.   Since all the exceptional curves in $\XX'$ are of type 2, 4 or  6,  
 $m(\XX')  = m(\XX) -1$.  Thus, by induction, we have a regular proper model with required properties. 
\end{proof}

\begin{lemma}
\label{type2_intersection}
Let $\XX$ be as in (\ref{choice_of_model}) and $X_0$ the special fibre of $\XX$.
Let $\eta$ be a codimension zero point of $X_0$ of type 2 and  $\eta'$  a codimension  zero point 
 of $X_0$ of type  3 or 5. Suppose there is a type 2 connection from $\eta$ to $\eta'$. 
If there is a type 2 connection from $\eta$ to a type 3 or 5 point $\eta''$, then 
$\eta' = \eta''$.  Further, if  $\eta_1, \cdots , \eta_n$  codimension zero 
 points of  $X_0$ of type 2 giving a type 2 connection from $\eta$ to $\eta'$ and 
 $\gamma_1, \cdots , \gamma_m$  codimension zero 
 points of  $X_0$ of type 2 giving  another  type 2 connection from $\eta$ to $\eta'$, 
then  $n = m$ and $\eta_i = \gamma_i$ for all $i$. 
\end{lemma}

\begin{proof} Suppose $\eta''$ is a codimension zero point of $X_0$ of type 3 or 5
with type 2 connection to $\eta$. Since $\eta$ is of type 2 and there is a type 2 connection 
from $\eta'$ to $\eta''$. Since no two points of type 3 or 5 have a type 2 connection (cf. \ref{choice_of_model}),
$\eta' = \eta''$.  Suppose $\gamma_1, \cdots , \gamma_m$ is of type 2 connection from 
$\eta$ to $\eta'$. If $\eta_i $ is not equal to $\gamma_i$, then we will have 
type 2 connection from $\eta'$ to $\eta'$ and hence a contradiction to the choice of 
$\XX$ (cf. \ref{choice_of_model}). Thus $n = m$ and  $\eta_i = \gamma_i$ for all $i$.
\end{proof}

Let $\eta$ be a codimension zero point of $X_0$ of type 2 and 
$\eta'$ be a codimension zero point of $X_0$ of type 3 or 5. 
Suppose there is a type 2 connection $\eta_1, \cdots, \eta_n$ from $\eta$ to
$\eta'$. Then, by (\ref{type2_intersection}),  $\eta'$ and $\eta_n$ are uniquely defined by $\eta$.
We call this point of intersection of $\eta_n$ with $\eta'$ as {\bf the point of type 2 intersection of $\eta$ and $\eta'$}.
Once again note that such a closed point is uniquely defined by $\eta$.

\section{Choice of $L_P$ and $\mu_P$ at closed points} 
\label{choice_at_closed_points}

Let $F$, $\alpha \in H^2(F, \mu_n)$, $\lambda \in F^*$ with $\alpha \cdot (\lambda) = 0 \in H^3(F, \mu_n^{\otimes 2})$, 
$\XX$ and $X_0$ be as in (\S \ref{patching} and \S \ref{types_of_points}).  Throughout this section we 
assume that $\XX$ has  no  special points and 
  if $\eta_1$ and $\eta_2$ are   two (not necessarily  distinct)  codimension 
zero points  of $X_0$ with $\eta_1$ is of  type 3 or 5 and $\eta_2$ is of type 3, 4 or 5, then 
there is no   type 2 connection between $\eta_1$ and $\eta_2$.  

Let $\eta$ be a codimension zero point of $X_0$ of type 5.
Then we call $\eta$ of {\bf type 5a} if $\alpha$ is unramified at $\eta$
and of {\bf type 5b} if $\alpha$ is ramified at $\eta$. 
Suppose $\eta$ is of type 5b. Then $\alpha$ is ramified and hence  $M_\eta$ is 
the unique subextension of $E_\eta$ of degree $\ell$, where   $(E_\eta, \sigma_\eta)$ is  the lift of the residue
of $\alpha$.  

For the rest of the paper we assume that $\kappa$ is a finite field.
 
\begin{lemma} 
\label{7_norm} Let $\eta$ be a codimension zero point of $X_0$ of type 5b.
Then ind$(\alpha \otimes M_\eta) < $ ind$(\alpha)$ and 
there exists $\mu_\eta \in M_\eta$ such that $N_{M_\eta/F_\eta}(\mu_\eta) = 
\lambda$ and $\alpha \cdot (\mu_\eta) = 0 \in H^3(M_\eta, \mu_n^{\otimes 2})$.
\end{lemma}

\begin{proof}  Since $\eta$ is of type 5b, $\alpha$ is ramified at $\eta$,  
$\nu_\eta(\lambda) = r\ell$ and $r\alpha \otimes E_\eta = 0$.
By (\ref{index_M}),  ind$(\alpha \otimes M_\eta) < $ ind$(\alpha)$. 
Write $\alpha \otimes F_\eta = \alpha' + (E_\eta, \sigma, \pi_\eta)$  as in (\ref{rbc})
and $\lambda = \theta_\eta \pi_\eta^{r\ell}$ where $\theta_\eta$ is a unit at $\eta$ and 
$\pi_\eta$ is a parameter at $\eta$.  
Let $\beta_0$ be the image of $\alpha' $ in $H^2(\kappa(\eta), \mu_n)$.
Since $\alpha' \otimes E_\eta = \alpha \otimes E_\eta$ and 
$r\alpha \otimes E_\eta = 0$,  $r\beta_0 \otimes E(\eta) = 0$. Let $\theta_0$ be the image of $\theta_\eta$ in
$\kappa(\eta)$.
Then, by (\ref{global_field_zero}),  there exists $\mu_0 \in M_\eta(\eta)$,
such that $N_{M_\eta(\eta)/\kappa(\eta)}(\mu_0) = \theta_0$  and $r\beta_0 \otimes M_\eta(\eta) = (E_\eta(\eta), \sigma, \mu_0)$.
Thus, by (\ref{type-other}),  there exists $\mu_\eta \in M_\eta$ with the required properties. 
\end{proof}
 
\begin{lemma}
 \label{choice_at_P_55}
  Let $P \in \PP$,  $\eta_1$ and $\eta_2$ be    
codimension zero points of  $X_0$  containing   $P$.
Suppose that $\eta_1$ and $\eta_2$ are of type 5.   Then  
there exist a cyclic field extension $L_P/F_P$  of degree $\ell$ and
$\mu_P \in L_P$ such that \\
$1)$ $N_{L_P/F_P}(\mu_P) = \lambda$, \\
$2)$ ind$(\alpha \otimes L_P) < $ ind$(\alpha)$, \\
$3)$  $\alpha \cdot (\mu_P) = 0 \in H^3(L_P, \mu_n^{\otimes 2}),$\\
$4)$ if $\eta_i$ is of type 5a, then  $L_P  \otimes F_{P, \eta_i} /F_{P, \eta_i}$ is an unramified field extension,\\
$5)$  if $\eta_i$ is of type 5b, then $L_P \otimes F_{P, \eta_i} \simeq M_{\eta_i} \otimes F_{P, \eta_i}$.  
  \end{lemma}

\begin{proof}  Since $\XX$ has no special points, $P$ is not a special point of type IV.
Since $\eta_1$ and $\eta_2$ are of type 5 intersecting at $P$,  $M_{\eta_1} \otimes F_{P, \eta_1}$  and $M_{\eta_2} \otimes F_{P, \eta_2}$ are fields. 
Suppose $\eta_i$ is of type 5a.  If $\alpha \otimes F_{P, \eta_i}  = 0$, then  
let $L_{P, \eta_i}/F_{P, \eta_i}$ be any cyclic unramified field extension with $\lambda$ a norm 
and $\mu_{\eta_i} \in L_{P, \eta_i}$ with $N_{L_{P, \eta_i}/F_{P, \eta_i}}(\mu_{\eta_i}) = \lambda$..
If $\alpha \otimes F_{P, \eta_i} \neq 0$,  then  let $L_{P, \eta_i}/F_{P, \eta_i}$ be a cyclic  unramified field extension of 
degree $\ell$ and $\mu_{\eta_i}$ be as in (\ref{local-local}).
Suppose $\eta_i$ is of type 5b. Let  $L_{P, \eta_i} = M_{\eta_i} \otimes F_{P, \eta_i}$ 
and $\mu_{\eta_i} \in M_{\eta_i}$ be as in (\ref{7_norm}).
Then, by the choice $L_{P, \eta_i}/F_{P, \eta_i}$ are unramified field extensions.
By applying  (\ref{patching_at_P}) to $L_{P, \eta_i}$ and $\mu_{\eta_i}$, 
there exist a cyclic field extension  $L_P/F_P$ and $\mu_P \in L_P$ with required properties. 
\end{proof}

\begin{lemma}
\label{type3_square}
Let $\eta$ be a   codimension zero point of $X_0$ of type 3 and $P$
a closed point on the closure of $\eta$.     
Then, there exists a cyclic field extension 
$L_{P, \eta}/F_{P, \eta}$ of degree $\ell$  such that 
 if $\alpha \otimes E_\eta \otimes F_{P, \eta} \neq 0$, then  ind$(\alpha \otimes E_\eta \otimes L_{P, \eta}) < $ 
 ind$(\alpha \otimes E_{\eta} \otimes F_{P, \eta})$.
\end{lemma}

\begin{proof} Since $E_{\eta}/ F_{\eta}$ is a cyclic 
 unramified field extension of degree a  power of $\ell$,   
 $E_\eta \otimes F_{P, \eta} \simeq \prod E_{P,\eta}$
 for some cyclic field extension $E_{P, \eta}/F_{P, \eta}$ of degree a power of $\ell$.
 Let $E(\eta)_P$ be the  residue field of $E_{P, \eta}$.
Then $E(\eta)_P/\kappa(\eta)_P$ is a cyclic extension of degree a power of $\ell$.
 Note that either $E(\eta)_P = \kappa(\eta)_P$ or $E(\eta)_P/\kappa(\eta)_P$ is a
 cyclic extension of degree a positive  power of $\ell$.
 If $E(\eta)_P = \kappa(\eta)_P$, then let $L(\eta)_P/\kappa(\eta)_P$ be any cyclic extension of 
 degree $\ell$. Suppose $E(\eta)_P \neq \kappa(\eta)_P$. Since  $E(\eta)_P/\kappa(\eta)_P$ is 
 cyclic extension of degree a positive  power of $\ell$,  there is only  one   subextension of $E(\eta)_P$ 
  which is cyclic over $\kappa(\eta)_P$ of degree $\ell$. 
 Since $\kappa(\eta)_P$ is a local field containing primitive  $\ell^{\rm th}$   root of unity, 
 there are at least 2  non-isomorphic  cyclic field extensions of $\kappa(\eta)_P$
of degree $\ell$.  Thus there exists a cyclic field extension $L(\eta)_P/\kappa(\eta)_P$ of degree 
$\ell$ which is not isomorphic to a subfield of  $E(\eta)_P$. 
Let $L_{P, \eta}/F_{P, \eta}$ be the unramified extension of degree $\ell$ with residue field 
$L(\eta)_P$.

Suppose $\alpha \otimes E_{\eta}  \otimes F_{P, \eta} \neq 0$.   
Then $\alpha \otimes E_{P, \eta} \neq 0$.
Since $E_{P, \eta}$ and $L_{P,\eta}$ are cyclic field extensions of $F_{P, \eta}$
and $L_{P, \eta}$ is not isomorphic to a subfield of $E_{P, \eta}$, $E_{P, \eta} \otimes 
L_{P, \eta}$ is a field and $[E_{P, \eta} \otimes L_{P, \eta} : E_{P, \eta}] = [L_{P, \eta} : F_{P, \eta}] = \ell$.
In particular $E(\eta)_P \otimes L(\eta)_P$ is  field and $[E(\eta)_P \otimes L(\eta)_P :  E(\eta)_P] = \ell$.
Write $\alpha \otimes F_\eta = \alpha'+ (E_\eta, \sigma_\eta, \pi_\eta)$ as in (\ref{rbc}).
Since $\alpha \otimes E_{\eta} =  \alpha' \otimes E_{\eta}$, $\alpha \otimes E_\eta$ is unramified at $\eta$.
Let $\beta_P$ be the image of $\alpha \otimes E_{\eta} \otimes F_{P, \eta}$ in $H^2(E(\eta)_P, \mu_n)$.
Since  $\alpha \otimes E_{P, \eta}  \neq 0$ and $E_{P, \eta}$ is a completely discretely valued field with 
residue field $E(\eta)_P$, $\beta_P  \neq 0$. Since $E(\eta)_P$ is a local field, 
ind$(\beta_P \otimes E(\eta)_P \otimes L(\eta)_P) < $ ind$(\beta_P)$ and hence 
ind$(\alpha \otimes E_ \eta  \otimes L_{P, \eta}) = $ ind$(\alpha \otimes E_{P, \eta} \otimes L_{P, \eta}) < $ ind$(\alpha \otimes E_{P, \eta})  = $ ind$(\alpha \otimes E_\eta \otimes F_{P, \eta})$.
  \end{proof}

\begin{lemma}
 \label{choice_at_P_2_56}
 Let $P \in \PP$,  $\eta_1$ and $\eta_2$ be    
codimension zero points of  $X_0$  containing   $P$.
Suppose that $\eta_1$ is of type 2 and $\eta_2$ is of type 5 or 6.
Then there exist  $\mu_i \in F_P$, $1 \leq i \leq \ell$, such that  \\
$1)$ $\mu_1 \cdots \mu_\ell = \lambda$,\\
$2) $ $\nu_{\eta_1}(\mu_1) = \nu_{\eta_1}(\lambda)$, $\nu_{\eta_1}(\mu_i) = 0$ for $i \geq 2$, \\
$3)$  $\nu_{\eta_2}(\mu_i) = \nu_{\eta_2}(\lambda)/\ell$ for all $i \geq 1$, \\
$4)$   $\alpha \cdot (\mu_i) = 0 \in H^3(F_P, \mu_n^{\otimes 2})$.\\ 
 \end{lemma}

\begin{proof}  Since $\eta_1$ is of type 2 and $\eta_2$ is of type 5 or 6, 
we have $\lambda = w\pi_{\eta_1}^{r_1} \pi_{\eta_2}^{r_2\ell}$ with
$r_1$ coprime to $\ell$ and $r_2\alpha \otimes E_{\eta_2} = 0$.
Hence, by (\ref{type2_56}), there exists $\theta \in F_P$ such that 
$\alpha \cdot (\theta) = 0$, $\nu_{\eta_1}(\theta) = 0$ and $\nu_{\eta_2}(\theta) = r_2$.
For $i \geq 2$, let $\mu_i = \theta$ and $\mu_1 = \lambda \theta^{1-\ell}$.
Then $\mu_i$ have the required properties. 
 \end{proof}

\begin{lemma}
 \label{choice_at_P_5_6}
 Let $P \in \PP$,  $\eta_1$ and $\eta_2$ be    
codimension zero points of  $X_0$  containing   $P$.
Suppose that $\eta_1$ and $\eta_2$ are  of type 5 or 6.
Then there exist  $\mu_i \in F_P$, $1 \leq i \leq \ell$, such that  \\
$1)$ $\mu_1 \cdots \mu_\ell = \lambda$,\\
 $2)$  $\nu_{\eta_j}(\mu_i) = \nu_{\eta_j}(\lambda)/\ell$ for all $i \geq 0$ and $j = 1, 2$, \\
$4)$   $\alpha \cdot (\mu_i) = 0 \in H^3(F_P, \mu_n^{\otimes 2})$.\\ 
 \end{lemma}

\begin{proof}  Since $\eta_1$ and $\eta_2$ are of type 5 or 6, by (\ref{type56}), 
there exists $\theta \in F_P$ such that $\alpha \cdot (\theta) = 0$ and
$\nu_{\eta_i}(\theta) = \nu_{\eta_i}(\lambda)/\ell$ for $i  = 1, 2$.
For $i \geq 2$, let $\mu_i = \theta \in F_P$
and $\mu_1 = \lambda \theta^{1-\ell} \in F_P$.
Then $\mu_i$ have the required properties. 
  \end{proof}

\begin{lemma}
 \label{choice_at_P_3_5}
 Let $P \in \PP$,  $\eta_1$ be a codimension zero point of  $X_0$ of type 3 
 and $\eta_2$    a  codimension zero point of  $X_0$ of type  5.
 Suppose $\eta_1$ and $\eta_2$ intersect at  $P$. 
Then there exist a cyclic field extension  $L_P/F_P$  
 of degree $\ell$ and
$\mu_P  \in L_P$ such that  \\ 
$1)$ $N_{L_P/F_P}(\mu_P) = \lambda$ \\
$2)$ ind$(\alpha \otimes L_P) < $ ind$(\alpha)$ \\
$3)$  $\alpha \cdot (\mu_P) = 0 \in H^3(L_P, \mu_n^{\otimes 2})$\\
$4)$ $L_P \otimes F_{P, \eta_i}/F_{P, \eta_i}$ is an unramified field extension,\\ 
$5)$  if  $\lambda \in F_P^{*\ell}$ and $\alpha \otimes E_{\eta_1} \otimes F_{P, \eta_1} \neq 0$, 
then ind$(\alpha \otimes (E_{\eta_1} \otimes F_{P, \eta_1}) \otimes (L_P \otimes F_{P, \eta_1})) <  $ 
ind$(\alpha \otimes E_{\eta_1} \otimes F_{P, \eta_1})$,\\
$6)$  if $\eta_2$ is of type 5b, then $L_P \otimes F_{P, \eta_2} \simeq M_{\eta_2} \otimes F_{P, \eta_2}$.
 \end{lemma}

\begin{proof} Suppose $\lambda \not\in F_P^{*\ell}$. 
Let  $L_P = F_P(\sqrt[\ell]{\lambda})$ and
$\mu_P = \sqrt[\ell]{\lambda}$.  Then  $N_{L_P/F_P}(\mu_P) = \lambda$ and 
by (\ref{local_nonsquare}) 2) and 3) are satisfied.  Since $\eta_i$ is of type 3 or 5, 
$\nu_{\eta_i}(\lambda) $ is divisible by $\ell$ and hence 4) is satisfied. 
Since $\lambda \not\in F_P^{*\ell}$, the case 5) does not arise. 
Suppose that $\eta_2$ is of type 5b.  Since $\XX$ has no special points, 
$M_{\eta_2} \otimes F_{P, \eta_2}$ is a field.  Since $\lambda$ is a norm from 
$M_{\eta_2}$ (\ref{7_norm}), by (\ref{dvr_local_field_norm}), we have $L_P \otimes F_{P, \eta_2} \simeq M_{\eta_2} \otimes F_{P, \eta_2}$.

Suppose that $\lambda \in F_P^{*\ell}$.
Let $L_{P, \eta_1}$ be as in (\ref{type3_square}) and $\mu_{P, \eta_1} = \sqrt[\ell]{\lambda}$. 
Write $\alpha \otimes F_{\eta_1} = \alpha_1 \otimes (E_{\eta_1}, \sigma_1, \pi_{\eta_1})$ as in (\ref{rbc}).
Then by (\ref{dvr_ind}), we have ind$(\alpha \otimes F_{\eta_1}) = $ ind$(\alpha \otimes E_{\eta_1})[E_{\eta_1} : 
F_{\eta_1}]$. Since 
$\eta_1$ is of type 3, ind$(\alpha) = $ ind$(\alpha \otimes F_{\eta_1})$ and $r_1 \alpha \otimes E_{\eta_1} \neq 0$, 
where $\nu_{\eta_1}(\lambda) = r_1\ell$. In particular $\alpha \otimes E_{\eta_1} \neq 0$.
By the choice of $L_{P, \eta_1}$ as in (\ref{type3_square}), we have either 
  $\alpha \otimes E_{\eta_1}  \otimes F_{P, \eta_1} = 0$  or 
 ind$(\alpha \otimes E_{\eta_1} \otimes L_{P, \eta_1}) < $ ind$(\alpha \otimes E_{\eta_1} \otimes F_{P, \eta_1})$.
Thus ind$(\alpha \otimes E_{\eta_1}  \otimes F_{P, \eta_1} )  <  $ ind$(\alpha \otimes E_{\eta_1})$. 
We have ${\rm ind}(\alpha \otimes L_{P, \eta_1}) \leq {\rm ind}(\alpha \otimes E_{\eta_1} \otimes L_{P, \eta_1})
[E_{\eta_1} \otimes L_{P, \eta_1} : L_{P, \eta_1}]  < $ ind$(\alpha \otimes E_{\eta_1}) [E_{\eta_1} : F_{\eta_1}]
= $ ind$(\alpha)$. Since $L_{P, \eta_1}$ is a field and cores$_{L_{P,\eta_1}/F_{P, \eta_1}}
(\alpha \cdot (\mu_{P, \eta_1})) = \alpha \cdot (\lambda) = 0$, by (\ref{corestriction}), 
$\alpha \cdot (\mu_{P, \eta_1}) = 0 \in H^3(L_{P, \eta_1}, \mu_n^{\otimes 2})$.
 Since  $\XX$ has no special points,
 $M_{\eta_2} \otimes F_{P, \eta_2}$ is a field.  Let $L_{P, \eta_2} = M_{\eta_2} \otimes F_{P, \eta_2}$
 and $\mu_{P, \eta_2} = \sqrt[\ell]{\lambda}$. Then $N_{L_{P, \eta_2}/F_{P, \eta_2}}(\mu_{P, \eta_2}) = 
 \lambda$ and by (\ref{7_norm}), ind$(\alpha \otimes L_{P, \eta_2}) < $ ind$(\alpha)$ and 
 $\alpha \cdot (\mu_{P, \eta_2}) =0$. 
 Then, by (\ref{patching_at_P}),  there exist $L_P$ and $\mu_P$ with required properties. 
 \end{proof}

\begin{lemma}
 \label{choice_at_P_346}
 Let $P \in \PP$,  $\eta_1$ and $\eta_2$ be  codimension zero points of  $X_0$ of type 3, 4 or 6.
  Suppose $\eta_1$ and $\eta_2$ intersect at  $P$. 
Then there exist a cyclic field extension  $L_P/F_P$  
 of degree $\ell$ and
$\mu_P  \in L_P$ such that  \\ 
$1)$ $N_{L_P/F_P}(\mu_P) = \lambda$ \\
$2)$ ind$(\alpha \otimes L_P) < $ ind$(\alpha)$ \\
$3)$  $\alpha \cdot (\mu_P) = 0 \in H^3(L_P, \mu_n^{\otimes 2})$\\
$4)$ $L_P \otimes F_{P, \eta_i}/F_{P, \eta_i}$ is an unramified field extension,\\
$5)$  if $\eta_i$ is of type 3, $\lambda \in F_P^{*\ell}$ and $\alpha \otimes E_{\eta_i} \otimes F_{P, \eta_i} \neq 0$, 
then ind$(\alpha \otimes (E_{\eta_i} \otimes F_{P, \eta_i}) \otimes (L_P \otimes F_{P, \eta_i})) <  $ 
ind$(\alpha \otimes E_{\eta_i} \otimes F_{P, \eta_i})$.\\
 \end{lemma}

\begin{proof} Suppose $\lambda \not\in F_P^{*\ell}$. 
Let  $L_P = F_P(\sqrt[\ell]{\lambda})$ and
$\mu_P = \sqrt[\ell]{\lambda}$.  Then  $N_{L_P/F_P}(\mu_P) = \lambda$ and 
by (\ref{local_nonsquare}),  2) and 3) are satisfied.  Since $\eta_i$ are of   type 3, 4 or 6, 
$\nu_{\eta_i}(\lambda) $ is divisible by $\ell$ and hence 4) is satisfied.  Since $\lambda \not\in F_P^{*\ell}$,
5) does not arise.

Suppose that $\lambda \in F_P^{*\ell}$.
If $\eta_i$ is of type 3, then let $L_{P, \eta_1}$ be as in (\ref{type3_square}) and $\mu_{P, \eta_1} = \sqrt[\ell]{\lambda}$. 
Then, as in (\ref{choice_at_P_3_5}),  $N_{L_P/F_P}(\mu_P) = \lambda$,
ind$(\alpha \otimes L_{P, \eta_i})< $ ind$(\alpha)$ and $\alpha \cdot (\mu_P) = 0 \in H^3(L_P, \mu_n^{\otimes 2})$.
Suppose that $\eta_i$ is of type 4 or 6. Let $L_{P, \eta_i}/F_{P, \eta_i}$ be a cyclic unramified field extension of degree
$\ell$ and $\mu_{P, \eta_i}$ be as in (\ref{local-local}).

Then, by (\ref{patching_at_P}),  there exist $L_P$ and $\mu_P$ with required properties. 
 \end{proof}

\begin{prop}
\label{choice_at_P}  Let $P \in \PP$. Then there exist a cyclic field  extension or split extension 
$L_P/F_P$ of degree $\ell$ and $\mu_P \in L_P$ such that \\
$1)$ $N_{L_P/F_P}(\mu_P) = \lambda$ \\
$2)$ ind$(\alpha \otimes L_P) < $ ind$(\alpha)$ \\
$3)$  $\alpha \cdot (\mu_P) = 0 \in H^3(L_P, \mu_n^{\otimes 2})$\\
Further, suppose  $\eta$ is a   codimension zero point of $X_0$ containing  $P$. \\
$4)$ If $\eta$ is of type 1, then $L_P  = 
F_P(\sqrt[\ell]{\lambda})$ and $\mu_P = \sqrt[\ell]{\lambda}$. \\
$5)$ Suppose  $\eta$ is of type 2  with a type 2 connection to a type 5 point  $\eta'$.
Let $Q$ be the type 2 intersection of $\eta$ and $\eta'$.  If  $M_{\eta'} \otimes F_{Q, \eta'}$ is
not a field,  then $L_P  =  \prod F_P$ and $\mu_P = (\theta_1, \cdots, \theta_\ell)$ with 
$\theta_i \in F_P$,  $\nu_{\eta}(\theta_1) = \nu_{\eta}(\lambda)$ and
$\nu_{\eta}(\theta_i) = 0$ for $ i\geq 2$.\\
$6)$ Suppose  $\eta$ is of type 2  with a type 2 connection to a type 5 point  $\eta'$.
Let $Q$ be the type 2 intersection of $\eta$ and $\eta'$.  If  $M_{\eta'} \otimes F_{Q, \eta'}$ is
a field, then $L_P  = F_P(\sqrt[\ell]{\lambda})$ and $\mu_P = \sqrt[\ell]{\lambda}$. \\
$7)$ Suppose $\eta$ is of type 2 and  there is no  type 2 connection  from   $\eta$ to any type 5 point. 
 Then $L_P  = F_P(\sqrt[\ell]{\lambda})$ and $\mu_P = \sqrt[\ell]{\lambda}$.  \\
$8)$ If $\eta$ is of type 3,  then $L_P\otimes F_{P, \eta}/F_{P, \eta}$ is an unramified field extension
and further   if  $\lambda \in F_P^{*\ell}$ and $\alpha \otimes E_\eta \otimes  F_{P,\eta} \neq 0$,
then ind$(\alpha \otimes (E_\eta \otimes F_{P, \eta}) \otimes (L_P \otimes F_{P, \eta})) < $ ind$(\alpha \otimes E_\eta \otimes  F_{P,\eta}).$\\  
$9)$ If $\eta$ is of type 4, then $L_P \otimes F_{P, \eta}/F_{P, \eta}$ is an unramified field extension.\\
$10)$ If $\eta$ is of  type 5a, then $L_P \otimes F_{P, \eta}/F_{P, \eta}$ is an unramified field extension.\\
 $11)$ If $\eta$ is of type 5b, then $L_P \otimes F_{P, \eta}
 \simeq M_{\eta} \otimes F_{P, \eta}$ and if $L_P = \prod F_P$, then 
 $\mu_P = (\theta_1, \cdots, \theta_\ell)$ with   $\nu_{\eta}(\theta_i) = \nu_{\eta}(\lambda)/\ell$.\\
 $12)$ If $\eta$ is of type  6,  then either $L_P \otimes F_{P, \eta}/F_{P, \eta}$ is an unramified field extension  
 or  $L_P = \prod F_P$,   
  with  $\mu_P = (\theta_1, \cdots , \theta_\ell)$  and  $\nu_{\eta}(\theta_i) = \nu_{\eta}(\lambda)/\ell$. \\ 
\end{prop}

%

\begin{proof}  Let $\eta_1$ and $\eta_2$ be  two codimension zero points  
 intersecting at   $P$.  By  the choice of $\XX$, $X_0$ is a union of regular curves with 
normal crossings and hence there are no other codimension zero points of $X_0$ passing through $P$. \\

\noindent
{\bf Case I.} 
Suppose that either $\eta_1$ or $\eta_2$, say $\eta_1$,  is of type 1. 
Then $\nu_{\eta_1}(\lambda)$ is coprime to $\ell$ and hence $\lambda \not\in F_P^{* \ell}$. 
Let $L_P = F_P(\sqrt[\ell]{\lambda})$ and $\mu_P = \sqrt[\ell]{\lambda}$. 
Then,  by (\ref{local_nonsquare}), $L_P$ and $\mu_P$ satisfy $1)$, $2)$ and $3)$.
By  choice  $4)$ is satisfied. 
Since $\XX$ has no special points, $\eta_2$  is not   of type 2 or 4. Thus  5), 6), 7) and 9) do not arise. 
Suppose  $\eta_2$ is of type 3,  5 or  6. Then $\nu_{\eta_2}(\lambda)$ is divisible by $\ell$ and hence 
$L_P \otimes F_{P,\eta_2}/F_{P, \eta_2}$ is an unramified field extension.  Thus 8),   10)  
and 12) are satisfied.  Suppose $\eta_2$ is of type  5b.  Since $\XX$ has no special points and $\eta_1$
is of type 1, $M_{\eta_2} \otimes F_{P, \eta_2}$ is a field. Since $\lambda$ is a norm from the
extension $M_{\eta_2}/F_{\eta_2}$ (\ref{7_norm}) and $\lambda \not \in F_{P, \eta_2}^{*\ell}$, 
by (\ref{dvr_local_field_norm} ), $M_{\eta_2} \otimes F_{P, \eta_2} \simeq F_{P, \eta_2}(\sqrt[\ell]{\lambda})$ 
and hence 11) is satisfied. \\

\noindent
{\bf Case II.}   Suppose neither $\eta_1$ nor $\eta_2$ is of type 1.
Suppose either  $\eta_1$ or $\eta_2$ is of type 2, say $\eta_1$ is of type 2.  
Then   $\nu_{\eta_1}(\lambda)$ is coprime to $\ell$ and hence
$\lambda \not\in F_P^{*\ell}$.

Suppose that $\eta_1$ has type 2 connection to   a codimension zero point $\eta'$ of $X_0$ of 
type  5.   Let $Q$ be the closed point on $\eta'$ which is the type 2 intersection point of $\eta_1$
and $\eta'$.  
By  the choice of $\XX$, $\eta_2$ is of type 2, 5 or 6.
Note that if $\eta_2$ is also of type 2,  then $Q$ is also  the point of
type 2 intersection of $\eta_2$ and $\eta'$. Thus if both $\eta_1$ and $\eta_2$ are of type 2, 
$\eta'$ and $Q$ do not depend on whether we start with $\eta_1$ or $\eta_2$.

Suppose that $M_{\eta'} \otimes F_{Q, \eta'}$ is not a field. 
Let $L_P  = \prod F_P$. 
Suppose $\eta_2$ is of type 2. Then, let $\mu_P = (\lambda, 1, \cdots , 1) \in L_P =\prod F_P$.
 Suppose $\eta_2$ is of type 5. Then  by the assumption on 
$\XX$, $\eta_2 = \eta'$, $Q = P$.  Thus $M_{\eta_2} \otimes F_{P, \eta_2} = M_{\eta'} \otimes F_{Q, \eta'}$
is not a field and hence $\eta_2$ is of type 5b. 
Let $\mu_i \in F_P$ be as in (\ref{choice_at_P_2_56}) and $\mu_P = (\mu_1, \cdots , \mu_\ell)$.
Suppose $\eta_2$ is of type 6. Let $\mu_i \in F_P$ be as in (\ref{choice_at_P_2_56}) and 
$\mu_P = (\mu_1, \cdots , \mu_\ell) \in L_P$. 
Then, $L_P$ and $\mu_P$ satisfy 1)  and 3). 
Since $\eta_1$ is of type 2, ind$(\alpha \otimes F_{\eta_1}) < $ ind$(\alpha)$
and hence, by (\ref{local_index}),  ind$(\alpha \otimes F_P) < $ ind$(\alpha)$ and 2) is satisfied.
Since neither $\eta_1$ nor  $\eta_2$ is of type 1, 
the case 4) does not arise. By  choice  $L_P$ satisfies 5). 
  Since there is only one type 5 point with a type 2 connection  to $\eta_1$ or $\eta_2$, 
the case 6) does not arise.  Clearly the case 7) does not arise.
Since $\eta_2$ is not of  type 3, 4 or 5a, the cases 8),  9) and 10)  do not arise. 
 By choice of $L_P$ and $\mu_P$, 11) and 12) are  satisfied.

Suppose  $M_{\eta'} \otimes F_{Q, \eta'}$ is a field.  
Let $L_P = F_P(\sqrt[\ell]{\lambda})$ and $\mu_P = \sqrt[\ell]{\lambda}$.
Since $\lambda \not\in F_P^{*\ell}$, 
 by (\ref{local_nonsquare}),  $L_P$ and $\mu_P$ satisfy $1)$, $2)$ and $3)$.
As above  the cases  4), 5), 7)  and 8)    do not arise.
By choice  6) is satisfied.  Suppose $\eta_2$ is of type 5. Then $\eta_2 =\eta'$,
$Q = P$ and  $\nu_{\eta_2}(\lambda)$ is divisible by $\ell$ and hence 9) is  satisfied.
Suppose $\eta_2$ is of type 5b. Since $M_{\eta_2} \otimes F_{P, \eta_2}$ is a field, 
as in case I, $M_{\eta_2} \otimes F_{P, \eta_2} \simeq L_P\otimes F_{P, \eta_2}$ and hence
10) is satisfied.  Since $\lambda \not\in F_P^{*\ell}$ and if $\eta_2$ is of type 6, 
$\nu_{\eta_2}(\lambda)$ is divisible by $\ell$,  $L_P \otimes F_{P, \eta_2}/F_{P, \eta_2}$ is an 
unramified field extension  and  hence  11) is satisfied.

Suppose that  neither $\eta_1$ nor $\eta_2$  have a type 2 connection to  a point of type 5. 
In particular $\eta_2$ is not of type 5. Then, let $L_P = F_P(\sqrt[\ell]{\lambda})$ 
and $\mu_P = \sqrt[\ell]{\lambda}$. Then, by (\ref{local_nonsquare}), $L_P$ and $\mu_P$ satisfy 1), 2) and 3). 
Since neither $\eta_1$ nor $\eta_2$ is of type 1, the case 4) does not arise.
Since neither $\eta_1$ nor $\eta_2$ has type 2 connection  to a point of type 5,   5) and 6) do  not arise.
By the  choice of $L_P$ and $\mu_P$, 7) is satisfied.    
If $\eta_2$ is of type 4 or 6, $\nu_{\eta_2}(\lambda)$ is divisible by  $\ell$,  
   8),  9)  and 12) are satisfied. 
 Since neither $\eta_1$ nor $\eta_2$ is of type 5,  10)  and 11) do  not arise.\\

\noindent
{\bf Case III.} Suppose neither of $\eta_i$ is of type 1 or 2. 
 Suppose that one of the $\eta_i$, say $\eta_1$, is of type 3.  
Since $\XX$ has no special points,  $\eta_2$ is not of type 4
and  hence  $\eta_2$ is of type 3, 5 or 6. If $\eta_2$ is of type  5, 
 let $L_P$ and $\mu_P$ be as in (\ref{choice_at_P_3_5}).
 If $\eta_2$ is of type 3 or 6,  let $L_P$ and $\mu_P$ be as in (\ref{choice_at_P_346}).
 Then, 1), 2), 3), 8), 9), 10), 11) and 12) are satisfied and other cases do not arise. \\

\noindent
{\bf Case IV.}  Suppose neither of $\eta_i$ is of type 1, 2 or 3. 
 Suppose that one of the $\eta_i$, say $\eta_1$, is of type 4.  
Since $\XX$ has no special points,  $\eta_2$ is not of type 5.
Hence $\eta_2$ is of type 4 or 6.  Let $L_P$ and $\mu_P$ be as in (\ref{choice_at_P_346}).
Then $L_P$ and $\mu_P$ have the required properties.  
\\

\noindent
{\bf Case V.}  Suppose neither of $\eta_i$ is of type 1, 2, 3 or 4. 
Suppose that one of the $\eta_i$ is of type 5, say $\eta_1$  is of type 5.
Then $\eta_2$ is of type 5 or 6.
Suppose that $\eta_2$ is of type 5. 
Since $\XX$ has no special points,   $M_{\eta_i} \otimes F_{P, \eta_i}$ are fields for 
$i = 1, 2$. Let $L_P$ and $\mu_P$ be as in (\ref{choice_at_P_55}).
Then $L_P$ and $\mu_P$ have the required properties. 
Suppose that $\eta_2$ is of type 6.   Suppose  $M_{\eta_1} \otimes F_{P, \eta_1}$ is a field.
Let $L_{P, \eta_1} = M_{\eta_1} \otimes F_{P, \eta_1}$ and $\mu_{\eta_1} \in M_{\eta_1}$ with $N_{M_{\eta_1}/F_{\eta_1}}(\mu_{\eta_1}) = \lambda$
(cf., \ref{7_norm}). Let $L_P$ and $\mu_P$ be as in (\ref{curve_point}) with 
$L_P \otimes F_{P, \eta_1} \simeq L_{P, \eta_1}$. Then $L_P$ and $\mu_P$ have the required properties.
Suppose that $M_{\eta_1} \otimes F_{P, \eta_1}$ is not a field. Let $L_P = \prod F_P$ and 
$\mu_i \in F_P$ be as in (\ref{choice_at_P_5_6}) and $\mu_P = (\mu_1, \cdots , \mu_\ell) \in L_P$.
Then $L_P$ and $\mu_P$ have the required properties.  \\

\noindent
{\bf Case VI.}  Suppose neither of $\eta_i$ is of type 1, 2, 3, 4 or 5.
Then, $\eta_1$ and $t_P\eta_2$ are of type 6.  Let $L_P$ and $\mu_P$ be as in 
(\ref{choice_at_P_346}). 
Then $L_P$ and $\mu_P$ have the required properties.  
\end{proof}

\section{Choice of $L_\eta$ and $\mu_\eta$ at codimension zero points.}
\label{choice_at_codimension_zero_points}
Let $F$, $n = \ell^d$, $\alpha \in H^2(F, \mu_n)$, $\lambda \in F^*$ with 
$\alpha \neq 0$, $\alpha \cdot (\lambda) = 0 \in H^3(F, \mu_n^{\otimes 2})$, 
$\XX$, $X_0$ and $\PP$ be as in (\S \ref{patching}, \S \ref{types_of_points} and \S \ref{choice_at_closed_points}). 
Assume that $\XX$ has  no  special points and there is no type 2 connection between  a
codimension zero  point of $X_0$ of  type  3 or 5  and  a codimension zero point of $X_0$ of  type 3, 4 or 5.
Further we assume that for every closed point $P$ of $X_0$, the residue field $\kappa(P)$ at $P$  
is a finite field. 
Let $P$ be a  closed point $P$ of $X_0$. Then there exists $t_P \geq d$ such that there is no primitive 
$\ell^{t_P^{\rm th}}$ root of unity in $\kappa(P)$.

For a codimension zero point $\eta$ of $X_0$, let $\PP_\eta = \eta \cap \PP$.

\begin{prop}
\label{codimension_1}
 Let $\eta$ be a codimension zero point of $X_0$ of type 1.  
For each $P \in \PP_\eta$,   let  $(L_P, \mu_P)$  be chosen as in (\ref{choice_at_P}) and 
$L_\eta = F_\eta(\sqrt[\ell]{\lambda})$ and $\mu_\eta = \sqrt[\ell]{\lambda} \in L_\eta$.
Then  \\
$1)$ $N_{L_\eta/F_\eta}(\mu_\eta) = \lambda$ \\
$2)$ $\alpha \cdot (\mu_\eta) = 0 \in H^3(L_\eta, \mu_n^{\otimes 2})$ \\
$3)$ ind$(\alpha \otimes L_\eta) < $ ind$(\alpha)$ \\
$4)$  for  $P \in \PP_\eta$,   there is an isomorphism $\phi_{P,\eta} : L_\eta \otimes F_{P, \eta} \to L_P  \otimes F_{P, \eta}$ 
and $$\phi_{P, \eta}(\mu_\eta \otimes 1) (\mu_P \otimes 1)^{-1} = 1.$$ 
\end{prop}

\begin{proof} By  choice, we have $N_{L_\eta/F_\eta}(\mu_\eta) = \lambda$.
Since $\eta$ is of type 1, $\nu_\eta(\lambda)$ is coprime to $\ell$
and hence by (\ref{dot-zero}), $L_\eta$ and $\mu_\eta$ satisfies  2) and 3).
Let $P \in \PP_\eta$. 
Since $\eta$ is of type 1, by the choice of $L_P$  and $\mu_P$  (cf.  \ref{choice_at_P}(4)), we
have $L_P = F_P(\sqrt[\ell]{\lambda})$ and $\mu_P = \sqrt[\ell]{\lambda}$.  Hence $L_\eta$ and $\mu_\eta$
satisfy 4). 
\end{proof}

\begin{lemma}
\label{codimension_2_split}
 Let $\eta$ be a codimension zero point of $X_0$.  
For each $P \in \PP_\eta$,   let $\theta_P \in F_P$ with $\alpha \cdot (\theta_P) = 0 \in H^3(F_{P, \eta}, \mu_n^{\otimes 2})$. 
Suppose   $\nu_\eta(\theta_P) = 0$ for all $P \in \PP_\eta$.
Then  there exists $\theta_\eta   \in F_\eta$ such that  \\
$1)$  $\alpha \cdot (\theta_\eta) = 0 \in H^3(F_\eta, \mu_n^{\otimes 2})$ \\
$2)$  for  $P \in \PP_\eta$,     $\theta_P^{-1} \theta_\eta \in   F_{P, \eta}^{\ell^{2t_P}}$.
\end{lemma}
 
\begin{proof} Let $\pi_\eta \in F_\eta$ be a parameter. 
Write $\alpha \otimes F_\eta = \alpha' + (E_\eta, \sigma_\eta, \pi_\eta)$ as in (\ref{rbc}).
Let  $E(\eta)$ be  the residue field of $E_\eta$.
Since $\alpha \cdot (\theta_P) = 0 \in H^3( F_{P,\eta}, \mu_n^{\otimes 2})$ and $\nu_\eta(\theta_P) = 0$,
 by (\ref{dot-zero}), we have  $(E(\eta)\otimes \kappa(\eta)_P, \sigma_0,  \overline{\theta}_P) = 0 \in H^2(\kappa(\eta)_P, \mu_n)$,
 where  $\overline{\theta}_P$ is the image of $\theta_P \in \kappa(\eta)_P$.
  Hence
$\overline{\theta}_P$ is a norm from $E(\eta) \otimes \kappa(\eta)_P$ for all $P \in \PP_\eta$.
For $P \in \PP_\eta$, let $\tilde{\theta}_P \in E(\eta) \otimes \kappa(\eta)_P$ with 
$N_{E(\eta) \otimes \kappa(\eta)_P/\kappa(\eta)_P}(\tilde{\theta}_P) = \overline{\theta}_P$.
By weak approximation, there exists $\tilde{\theta} \in \kappa(\eta)$  which is sufficiently close to 
$\tilde{\theta}_P$ for all $P \in \PP_\eta$. 
Let $\theta_0  = N_{E(\eta)/\kappa(\eta)}(\tilde{\theta}) \in \kappa(\eta)$.
 Then $\theta_0 $ is sufficiently close to $\overline{\theta}_P$ for all $P \in \PP_\eta$.
In particular, $\theta_0^{-1}\overline{\theta}_P \in \kappa(\eta)_P^{\ell^{2t_P}}$.  
Let $\theta_\eta \in F_\eta$ have  image $\theta_0$ in $\kappa(\eta)$.
Then $(E_\eta, \sigma_\eta, \theta_\eta)  = 0 $ and hence, by (\ref{dot-zero}), 
$\alpha \cdot (\theta_\eta) = 0$. Since $\theta_0^{-1}\overline{\theta}_P \in \kappa(\eta)_P^{\ell^{2t_P}}$
and $F_{P, \eta}$ is a complete discretely valid field with residue field $\kappa(\eta)_P$, 
 it follows that $\theta_\eta^{-1}\theta_P \in F_{P, \eta}^{\ell^{2t_P}}$.
\end{proof}

\begin{prop}
\label{codimension_2_splitp}
 Let $\eta$ be a codimension zero point of $X_0$ of type 2.  
 Suppose there is  a type 2 connection between  $\eta$  and  a codimension zero point 
$\eta'$ of $X_0$ of type 5. Let   $Q$  be the point of type 2 intersection of $\eta$ and $\eta'$.
Suppose that  $M_{\eta'} \otimes F_{Q, \eta'}$ is not a field. 
For each $P \in \PP_\eta$, let $\mu_P = (\theta_1^P, \cdots , \theta_\ell^P) \in L_P = \prod F_P$
 be as in (\ref{choice_at_P}).
 Let   $L_\eta = \prod F_\eta$.
Then  there exists $\mu_\eta  = (\theta^{\eta}_1, \cdots , \theta^{\eta}_\ell) \in L_\eta$ such that  \\
$1)$ $N_{L_\eta/F_\eta}(\mu_\eta) = \lambda$ \\
$2)$ $\alpha \cdot (\mu_\eta) = 0 \in H^3(L_\eta, \mu_n^{\otimes 2})$ \\
$3)$ ind$(\alpha \otimes L_\eta) < $ ind$(\alpha)$,\\
$4)$  $\mu_P^{-1} \mu_\eta \in (L_\eta \otimes F_{P, \eta})^{\ell^{2t_P}}$
for all $P \in \PP_\eta$. 
\end{prop}

\begin{proof}  Let $i \geq 2$.
By  choice (cf. \ref{choice_at_P}(5)), we have $\nu_\eta(\theta_i^P) = 0$ and
$\alpha \cdot (\theta_i^P) = 0  \in H^3(F_P, \mu_n^{\otimes 2})$ for all $P \in \PP_\eta$.  By (\ref{codimension_2_split}), 
there exists $\theta_i^{\eta} \in F_{\eta}$ such that $\alpha \cdot (\theta_i^\eta) = 0  \in H^3(F_\eta, \mu_n^{\otimes 2})$ 
and $(\theta_i^P)^{-1}\theta_i^{\eta} \in  F_{P, \eta}^{\ell^{2t_P}}$
for all $P \in\PP_\eta$.
Let $\theta_1^{\eta} = \lambda (\theta_2^\eta \cdots \theta_\ell^\eta)^{-1}$.
Then $\theta_1^\eta \cdots \theta^\eta_\ell = \lambda$ and 
$(\theta_1^P)^{-1}\theta_1^{\eta} \in  F_{P, \eta}^{\ell^{2t_P}}$. Since $\alpha \cdot (\lambda) = 0$ and 
$\alpha \cdot (\theta_i^{\eta}) = 0 \in H^3(F_\eta, \mu_n^{\otimes 2})$ for $i \geq 2$,
we have $\alpha \cdot (\theta_1) = 0  \in H^3(F_\eta, \mu_n^{\otimes 2})$.
 Hence $L_\eta = \prod F_\eta$ and 
 $\mu_\eta = (\theta_1^\eta, \cdots, \theta_\ell^\eta) \in L_\eta$.
 Since $\eta$ is of type 2, ind$(\alpha \otimes F_\eta) < $ ind$(\alpha)$ and hence 
 $L_\eta$, $\mu_\eta$ have the required properties.   
\end{proof}

\begin{prop}
\label{codimension_2} 
 Let $\eta$ be a codimension zero point of $X_0$ of type 2.  
 For each $P \in \PP_\eta$,   let  $(L_P, \mu_P)$  be chosen as in (\ref{choice_at_P}).
Suppose one of the following holds. \\
$\bullet$ There is  a type 2 connection between  $\eta$  and  codimension zero point 
$\eta'$ of $X_0$ of type 5 with  $Q$   the point of type 2 intersection of $\eta$ and $\eta'$ and 
$M_{\eta'} \otimes F_{Q, \eta'}$ is  a field.\\
$\bullet$ There is no type 2 connection between  $\eta$ and  
any  codimension zero point of $X_0$ of type 5.\\
Let $L_\eta =  F_\eta(\sqrt[\ell]{\lambda})$ and $\mu_\eta = \sqrt[\ell]{\lambda}$.
Then    \\
$1)$ $N_{L_\eta/F_\eta}(\mu_\eta) = \lambda$ \\
$2)$ $\alpha \cdot (\mu_\eta) = 0 \in H^3(L_\eta, \mu_n^{\otimes 2})$ \\
$3)$ ind$(\alpha \otimes L_\eta) < $ ind$(\alpha)$ \\
$4)$ for  $P \in \PP_\eta$,   there is an isomorphism 
$\phi_{P,\eta} : L_\eta \otimes F_{P, \eta} \to L_P  \otimes F_{P, \eta}$ 
and 
$$\phi_{P, \eta}(\mu_\eta \otimes 1) (\mu_P \otimes 1)^{-1} = 1.$$ 

\end{prop}

\begin{proof} Since $\nu_\eta(\lambda)$ is coprime to $\ell$,
 by (\ref{dot-zero}), 
 $\alpha \cdot (\mu_\eta) = 0 \in H^3(L_\eta, \mu_n^{\otimes 2})$ and 
 ind$(\alpha \otimes L_\eta) < $ ind$(\alpha)$. Clearly $N_{L_\eta/F_\eta}(\mu_\eta) = \lambda$.
 By the choice of $(L_P, \mu_P)$ (cf. \ref{choice_at_P}), for $P \in \PP_\eta$,
 we have   $L_P = F_P(\sqrt[\ell]{\lambda})$ and $\mu_P = \sqrt[\ell]{\lambda}$.
Thus $L_\eta $ and $\mu_\eta$ have the required properties. 
\end{proof}

\begin{lemma}
\label{type345a_ind}
 Let $\eta$ be a codimension zero point of $X_0$ of type 3, 4 or 5a. 
 Let $P \in \eta$. Suppose there exists   $L_{P, \eta}/F_{P, \eta}$ a degree $\ell$ unramified  field extension and
 $\mu_{P, \eta} \in L_{P, \eta}$ such that \\
 1) $N_{L_{P,\eta}/F_{P,\eta}}(\mu_{P,\eta}) = \lambda$, \\
 2) ind$(\alpha \otimes L_{P, \eta}) < $ ind$(\alpha)$,\\
 3) $\alpha \cdot (\mu_{P, \eta}) = 0 \in H^3(L_{P, \eta}, \mu_n^{\otimes 2})$, \\ 
 4)  If $\eta$ is of type 3, $\lambda \in F_P^{*\ell}$ and $\alpha \otimes E_\eta \otimes  F_{P,\eta} \neq 0$,
then ind$(\alpha \otimes (E_\eta \otimes F_{P, \eta}) \otimes (L_{P, \eta})) < $ ind$(\alpha \otimes E_\eta \otimes  F_{P,\eta}).$\\ 
   Then   ind$(\alpha \otimes (E_\eta \otimes F_{P, \eta}) \otimes (L_{P, \eta})) < $ ind$(\alpha)/ [E_\eta : F_\eta]$.
 \end{lemma}

\begin{proof}
Write $\alpha \otimes F_\eta = \alpha' + (E_\eta, \sigma_\eta, \pi_\eta)$ as in (\ref{rbc}).
Then, by (\ref{dvr_ind}), ind$(\alpha \otimes F_\eta) = $ ind$(\alpha'\otimes E_\eta) [E_\eta : F_\eta]
 = $ ind$(\alpha \otimes E_\eta)[E_\eta : F_\eta]$. Let $t = [E_\eta : F_\eta]$ and $\beta$
 be the image of $\alpha'$ in $H^2(\kappa(\eta), \mu_n)$.

 Suppose $\eta$ is of type 4. Then ind$(\alpha \otimes F_\eta) < $ ind$(\alpha)$ and
 hence ind$(\alpha \otimes  E_\eta )  = $ ind$(\alpha \otimes F_\eta)/t < $ ind$(\alpha) / t$.
 We have ind$(\alpha \otimes (E_\eta \otimes F_{P, \eta}) \otimes (L_{P, \eta})) \leq $ ind$(\alpha \otimes E_\eta)
 < $ ind$(\alpha)/ t$.

Suppose that $\eta$ is of type  5a.  Then $\alpha$ is unramified at $\eta$ and hence $E_\eta = F_\eta$
 and $t = 1$.  The lemma is clear  if $\alpha \otimes F_{P, \eta} = 0$. Suppose
$\alpha \otimes F_{P, \eta} \neq 0$.  Then $\beta \neq 0$.  Since $L_{P, \eta}$ is a
an unramified field extension, 
the residue field 
$L_P(\eta)$ of $L_{P, \eta}$ is a field extension of $\kappa(\eta)_P$ of degree $\ell$.
Since $\kappa(\eta)_P$ is a local field and ind$(\beta)$ is divisible by $\ell$, 
ind$(\beta \otimes L_P(\eta)) < $ ind$(\beta)$ (\cite[p. 131]{CFANT}). In particular ind$(\alpha \otimes L_{P, \eta}) < $
ind$(\alpha)$.

Suppose that $\eta$ is of type 3. Then $r\alpha \otimes E_{\eta} \neq 0$ and 
hence $\alpha' \otimes E_\eta = \alpha  \otimes E_\eta \neq 0$.
In particular ind$(\alpha \otimes F_\eta) > t$ and  $\beta \otimes E(\eta) \neq 0$.  
If $\alpha \otimes  E_\eta \otimes F_{P, \eta}  = 0$, 
then ind$(\alpha \otimes  (E_\eta \otimes F_{P, \eta})   \otimes (L_{P, \eta}) ) = 1 < $ ind$(\alpha) /t$. 
Suppose that $\alpha \otimes  E_\eta \otimes F_{P, \eta} \neq 0$.
Suppose $\lambda \in F_{P} ^{*\ell}$. Then, by the  choice of $L_{P, \eta}$,
ind$(\alpha \otimes  (E_\eta \otimes F_{P, \eta})   \otimes (L_{P, \eta}) )
< $ ind$(\alpha \otimes E_\eta \otimes F_{P, \eta}) \leq $ ind$(\alpha \otimes E_\eta) = $ ind$(\alpha) /t$.
Suppose $\lambda \not\in F_{P} ^{*\ell}$.
Then $\lambda \not\in F_{P, \eta}^{*\ell}$. 
Since $L_{P, \eta}$ is a field extension of degree $\ell$ and $\lambda$ is a norm from 
$L_{P, \eta}$, by (\ref{dvr_local_field_norm}),  $L_{P, \eta} \simeq F_{P, \eta}(\sqrt[\ell]{\lambda})$.
Since $\eta$ is of type 3, $\nu_\eta(\lambda) =  r\ell$
and $\lambda = \theta_\eta \pi_\eta^{r\ell}$ with $\theta_\eta \in F_\eta$ a unit at $\eta$.
Let   $\overline{\theta}_\eta$ be the image  of $\theta_\eta$ in $\kappa(\eta)$.  
Then $\overline{\theta}_\eta \not\in \kappa(\eta)_P^{\ell}$ and $L_P(\eta) = \kappa(\eta)_P(\sqrt[\ell]{\overline{\theta}_\eta})$. 
Since $\alpha \cdot (\lambda) = 0$, by (\ref{dot-zero}),
$r\ell \alpha'  = (E_\eta, \sigma_\eta, \theta_\eta)$ and hence $r\ell \beta = (E(\eta), \sigma_0, \overline{\theta}_\eta)$.
 Thus, by (\ref{local_index_cal}), 
ind$(\beta \otimes E(\eta)_P  \otimes L_P(\eta)) < $ ind$(\beta \otimes  E(\eta))$.
 Thus
$$
\begin{array}{rcl}
 {\rm ind}(\alpha \otimes (E_\eta \otimes F_{P, \eta}) \otimes (L_{P, \eta})) 
& = & {\rm ind}(\alpha' \otimes (E_\eta \otimes F_{P, \eta}) \otimes (L_{P, \eta})) \\
&  = &     {\rm ind}(\beta \otimes E(\eta)_P \otimes L_P(\eta))  \\
& < &{\rm ind}(\beta \otimes E(\eta))   =  {\rm ind}(\alpha' \otimes E_\eta) \\
& = & {\rm ind}(\alpha \otimes E_\eta) = {\rm ind}(\alpha)/t.
\end{array}
$$
\end{proof}

\begin{prop}
\label{codimension_345a}
 Let $\eta$ be a codimension zero point of $X_0$ of type 3, 4 or 5a.  
For each $P \in \PP_\eta$,  let   $(L_P,  \mu_P) $ be chosen as in (\ref{choice_at_P}). 
Then there exist a field   extension $L_\eta/F_\eta$
of degree $\ell$ and $\mu_\eta \in L_\eta$ such that \\
$1)$ $N_{L_\eta/F_\eta}(\mu_\eta) = \lambda$ \\
$2)$ $\alpha \cdot (\mu_\eta) = 0 \in H^3(L_\eta, \mu_n^{\otimes 2})$ \\
$3)$ ind$(\alpha \otimes L_\eta) < $ ind$(\alpha)$ \\
$4)$  for  $P \in \PP_\eta$,   there is an isomorphism $\phi_{P,\eta} : L_\eta \otimes F_{P, \eta} \to L_P  \otimes F_{P, \eta}$  and 
$$\phi_{P, \eta}(\mu_\eta \otimes 1) (\mu_P \otimes 1)^{-1}  \in (L_P \otimes F_{P, \eta})^{\ell^{2t_P}}.$$ 
\end{prop}

\begin{proof} Write $\alpha \otimes F_\eta = \alpha' + (E_\eta, \sigma_\eta, \pi_\eta)$ as in (\ref{rbc}).
 By (\ref{dot-zero}),  $r\ell\alpha' = (E_\eta, \sigma_\eta, \theta_\eta)$.
Let $\beta$ be the image of 
$\alpha' $ in $H^2(\kappa(\eta), \mu_n)$ and $E(\eta)$ the residue field of $E_\eta$.
Then   $r\ell \beta = (E(\eta), \sigma_0, \theta_0) 
\in H^2(\kappa(\eta), \mu_n)$, where $\sigma_0$ is the automorphism of $E(\eta)$ induced by $\sigma_\eta$
and $\theta_0$ is the image of $\theta_\eta$ in $\kappa(\eta)$.

Let $S$ be a finite   set of places of $\kappa(\eta)$ containing   the places given by closed points of $\PP_\eta$
and places $\nu$ of $\kappa(\eta)$ with  $\beta \otimes \kappa(\eta)_\nu  \neq 0$. For each $\nu \in S$, we now give
a cyclic field extension $L_\nu/\kappa(\eta)_\nu$ of degree $\ell$ and $\mu_\nu \in L_\nu$ satisfying 
the conditions of (\ref{uglylemma}) with $E_0 = E(\eta)$ and $d = $ ind$(\alpha)/t$.

Let $\nu \in S$. Then $\nu$ is given by a closed point $P$ of $\eta$.
If $P \in \PP$, let $L_{P, \eta} = L_P \otimes F_{P, \eta}$ and $\mu_{P,\eta} = \mu_P \otimes 1 
\in L_{P, \eta}$. Suppose $P \not\in \PP$.  
Suppose that $\lambda \not\in F_P^{*\ell}$. Then $\lambda \not\in F_{P, \eta}^{*\ell}$.
Let $L_{P,\eta} =  F_{P, \eta}(\sqrt[\ell]{\lambda})$ and $\mu_{P, \eta} =
\sqrt[\ell]{\lambda}$. Suppose that $\lambda \in F_P^{*\ell}$. 
If $\eta$ is of type 3, then  let $L_{P,\eta}/F_{P,\eta}$ be a cyclic unramified field extension of degree $\ell$
as in (\ref{type3_square}) and $\mu_{P, \eta} =
\sqrt[\ell]{\lambda}$. If   $\eta$ is of type 4 or 5a, then let $L_{P, \eta}/F_{P, \eta}$
be a  cyclic unramified field extension of degree $\ell$ as in (\ref{local-local}) and $\mu_{P, \eta} =
\sqrt[\ell]{\lambda}$.

Since  $L_{P, \eta}/F_{P,\eta}$ is an unramified field extension of degree $\ell$,
the rescue field $L_P(\eta)$ is a degree $\ell$ field extension of $\kappa(\eta)_P$.
Let $L_\nu = L_P(\eta)$.
  We have $\nu_\eta(\lambda) = r\ell$ for some integer $r$ and
   $\lambda = \theta_\eta \pi_\eta^{r\ell}$ for some parameter $\pi_\eta$
at $\eta$ and $\theta_\eta \in F_\eta$ a unit at $\eta$. 
Further   $\pi_\eta$ is a parameter in $L_{P, \eta}$.
Since $N_{L_{P, \eta}/F_{P, \eta}}(\mu_{P, \eta}) = \lambda$,  $\mu_{P,\eta} = \theta_{P, \eta} \pi^{r}_\eta$
for some $\theta_{P, \eta} \in L_P \otimes F_{P, \eta}$ which is a unit at $\eta$. 
Let   $\mu_\nu$ be the image of $\theta_{P, \eta}$ in 
$L_\nu = L_P(\eta)$.  Then    
$N_{L_\nu/\kappa(\eta)_\nu}(\mu_\nu) = \theta_0$.  
Since the corestriction map $H^2(L_\nu, \mu_n) \to H^2(\kappa(\eta)_\nu, \mu_n)$ is 
injective,    $r\beta \otimes L_{\nu}  =  (E_0\otimes L_\nu, \sigma_0\otimes 1, \mu_\nu).$ 
Let $t = [E_\eta : F_\eta]$. By (\ref{type345a_ind}), we have 
ind$(\alpha \otimes (E_\eta \otimes F_{P, \eta}) \otimes L_{P, \eta}) < $ ind$(\alpha) /t$.
Since $\alpha \otimes E_\eta = \alpha' \otimes E_\eta$, we have 
ind$(\alpha' \otimes (E_\eta \otimes F_{P, \eta}) \otimes L_{P, \eta}) < $ ind$(\alpha)/\ell^d$.
Since ind$(\beta \otimes E_0 \otimes L_\mu) = $ ind$(\alpha' \otimes (E_\eta \otimes F_{P, \eta}) 
\otimes (L_P \otimes F_{P, \eta}))$, ind$(\beta \otimes E_0 \otimes L_\nu) < $ ind$(\alpha) /t$.

Since $\kappa(\eta)$ is a global-field,  by  (\ref{uglylemma}),  there exist 
 a field extension $L_0/\kappa(\eta)$ of  degree $\ell$ and $\mu_0 \in L_0$ such that 
  
   1) $N_{L_0/k}(\mu_0) = \theta_0$ 
   
   2) $r\beta \otimes L_0  =  (E(\eta) \otimes L_0, \sigma_0 \otimes 1, \mu_0 )  $ 
   
   3)  ind$(\beta \otimes E(\eta)  \otimes L_0) <   $ ind$(\alpha)/t$
   
   4) $L_0 \otimes \kappa(\eta)_P  \simeq L_P(\eta)$ for all $P \in \PP_\eta$ 
   
   5) $\mu_0$ is close to $\overline{\theta}_P$ for all $P \in \PP_\eta$.\\
  Then, by (\ref{type-other}), there exist a field extension $L_\eta/F_\eta$ of degree $\ell$ and $\mu  \in L_\eta$ such that \\ 
  $\bullet$ residue field of $L_\eta$ is  $L_0$, \\
 $\bullet$ $\mu $ a unit in the valuation ring of $L_\eta$, \\
 $\bullet$ $\overline{\mu}  = \mu_0$, \\
 $\bullet$    $N_{L_\eta/F_\eta} (\mu) = \theta_\eta$,  \\
 $\bullet$  $\alpha \cdot (\mu  \pi_\eta^r) \in H^3(L_\eta, \mu_n^{\otimes 2})$ is unramified. \\
 Since $L_\eta$ is a complete discretely valued field with residue field $L_0$ a global field, 
 $H^3_{nr}(L_\eta, \mu_n^{\otimes 2}) = 0$ (\cite[p. 85]{Serre_Gal_coh}) and hence $\alpha \cdot (\mu \pi_\eta^r) =0$. 
 Since $L_\eta/F_\eta$ is unramified and 
  $\alpha \otimes L_\eta = \alpha' \otimes L_\eta + (E_\eta \otimes L_\eta, \sigma_\eta, \pi_\eta)$,
    ind$(\alpha \otimes L_\eta) \leq $ ind$(\alpha'\otimes E_\eta \otimes L_\eta) [E_\eta\otimes  L_\eta
  : L_\eta] = $ ind$(\beta \otimes E(\eta)\otimes L_0) t < $ ind$(\alpha)$ . 
   Thus  $L_\eta$ and $\mu_\eta =\mu \pi_\eta^r \in L_\eta$  have the required properties. 
  \end{proof}

\begin{prop}
\label{codimension_5b}
 Let $\eta$ be a codimension zero point of $X_0$ of type  5b.  
  Let $(E_\eta, \sigma_\eta)$ be the residue of $\alpha$ at $\eta$
and $M_\eta$ be the unique subfield of $E_\eta$ with $M_\eta/F_\eta$ a cyclic extension of degree $\ell$. 
For each $P \in \PP_\eta$,  let $L_P$ and $\mu_P$ be as in (\ref{choice_at_P}).   
Then there exists  $\mu_\eta \in M_\eta$ such that \\
$1)$ $N_{M_\eta/F_\eta}(\mu_\eta) = \lambda$ \\
$2)$ $\alpha \cdot (\mu_\eta) = 0 \in H^3(M_\eta, \mu_n^{\otimes 2})$ \\
$3)$ ind$(\alpha \otimes  M_\eta) < $ ind$(\alpha)$ \\
 $4)$  for  $P \in \PP_\eta$,   there is an isomorphism $\phi_{P,\eta} : M_\eta \otimes F_{P, \eta} \to L_P  \otimes F_{P, \eta}$  and 
 $$\phi_{P, \eta}(\mu_\eta \otimes 1) (\mu_P \otimes 1)^{-1}  \in (L_P \otimes F_{P, \eta})^{\ell^{2t_P}}.$$
 \end{prop}

\begin{proof} Let $E(\eta)$ and $M(\eta)$ be the residue fields of $E_\eta$ and
$M_\eta$ at $\eta$.  Since $\eta$ is of type 5b,   $M(\eta)$ is the unique subfield of $E(\eta)$  with 
$M(\eta)/\kappa(\eta)$ a cyclic  field extension of degree $\ell$. 
Let $\pi_\eta$ be a parameter at $\eta$. 
Since $\eta$ is  of type 5,   $\nu_\eta(\lambda) = r\ell$ and $\lambda = \theta_\eta \pi_\eta^{r\ell}$
for some $\theta_\eta \in F$ a unit at $\eta$.  Let $\overline{\theta}_\eta$ be the image of $\theta_\eta$ 
in $\kappa(\eta)$.  Let $P \in \PP_\eta$. Suppose $M_\eta \otimes F_{P, \eta}$ is a field.
Since $N_{M_\eta \otimes F_{P, \eta}/ F_{P, \eta}}(\mu_P) = 
\lambda = \theta_\eta \pi_\eta^{r\ell}$,  we have 
$\mu_P = \mu'_P \pi_\eta^r$ with $\mu'_P \in M_\eta \otimes F_{P, \eta}$ 
a unit at $\eta$ and $N_{M_\eta \otimes F_{P, \eta}/ F_{P, \eta}}(\mu'_P) = \theta_\eta$. 
Suppose $M_\eta \otimes F_{P, \eta}$ is not a field.  Then, by the choice of $\mu_P$ (cf. \ref{choice_at_P}(10)),
we have $\mu_P = \mu'_P\pi_\eta^r$, where $\mu'_P = (\theta_1', \cdots , \theta_\ell') \in 
M_\eta \otimes F_{P, \eta'} = \prod F_{P, \eta}$
with each $\theta'_i \in F_{P, \eta}$ is a unit at $\eta$.  
Let   $\overline{\mu'}_P$ be the image of $\mu'_P$ in the residue field  $M(\eta) \otimes \kappa(\eta)_P$
of $M_\eta \otimes F_{P,\eta}$ at $\eta$. 
Write 
$\alpha \otimes F_\eta = \alpha' + (E_\eta, \sigma_\eta, \pi_\eta)$ as in (\ref{rbc}).  
Let $\beta$ 
be the image of $\alpha'$ in $H^2(\kappa(\eta), \mu_n)$. Since $\alpha \cdot (\lambda) = 0$, by 
(\ref{dot-zero}),   $r\ell \beta = (E(\eta), \sigma_\eta, \overline{\theta}_\eta)$. 
 Since $\alpha \cdot (\mu_P)  = 0$ in $H^3(M_\eta \otimes F_{P, \eta}, \mu_n^{\otimes 2})$,
once again by (\ref{dot-zero}),   $r\beta  \otimes \kappa(\eta)_P 
= (E(\eta)  \otimes 
M(\eta) \otimes \kappa(\eta)_P,  \sigma_\eta, \overline{\mu'}_P)$. Since $\kappa(\eta)$ is a
global field, by (\ref{global_field_zero_cor}), there exists $\mu'_\eta \in M(\eta)$ such that \\
$1)$ $N_{M(\eta)/\kappa(\eta)}(\mu'_\eta) = \overline{\theta}_\eta$ \\
$2)$ $r\beta  \otimes M(\eta)
= (E(\eta)  \otimes 
M(\eta),  \sigma_\eta, {\mu'}_\eta)$\\
$3)$ $\overline{\mu'}_P$ is close to $\mu'_\eta$ for all $P \in \PP_\eta$.\\
Since $M_\eta$ is complete, there exists $\tilde{\mu_\eta} \in M_\eta$
such that $N_{M_\eta/F_\eta}(\tilde{\mu'_\eta}) = \theta_\eta$ and 
the image of $\tilde{\mu'_\eta}$ in $M(\eta)$ is $\mu'_\eta$. Let 
$\mu_\eta = \tilde{\mu'_\eta}\pi_\eta^r$.  Since $M_\eta/F_\eta$ is of degree $\ell$,
ind$(\alpha \otimes M_\eta) < $ ind$(\alpha \otimes F_\eta)$ (cf. \ref{index_M} ).
Thus  $ \mu_\eta$ has the required properties. 
\end{proof}

\begin{prop}
\label{codimension_6}
 Let $\eta$ be a codimension zero point of $X_0$ of type   6. 
For each $P \in \PP_\eta$,  let   $L_P $ and $\mu_P$ be as in  (\ref{choice_at_P}).
 Then there exist a field  extension $L_\eta/F_\eta$
of degree $\ell$ and $\mu_\eta \in L_\eta$ such that \\
$1)$ $N_{L_\eta/F_\eta}(\mu_\eta) = \lambda$ \\
$2)$ $\alpha \cdot (\mu_\eta) = 0 \in H^3(L_\eta, \mu_n^{\otimes 2})$ \\
$3)$ ind$(\alpha \otimes L_\eta) < $ ind$(\alpha)$ \\
$4)$   for  $P \in \PP_\eta$,   there is an isomorphism $\phi_{P,\eta} : M_\eta \otimes F_{P, \eta} \to L_P  \otimes F_{P, \eta}$  and 
$$\phi_{P, \eta}(\mu_\eta \otimes 1) (\mu_P \otimes 1)^{-1}  \in (L_P \otimes F_{P, \eta})^{\ell^{2t_P}}.$$
\end{prop}

\begin{proof}  Let $P \in \PP_\eta$.
Suppose $L_P \otimes F_{P, \eta}$ is a field. 
Let $L_P(\eta)$, $\overline{\theta}_P \in L_P(\eta)$, 
$\theta_0 \in \kappa(\eta)$  and   $\beta$  be as in the proof of  
(\ref{codimension_345a}).  
Then, as in the proof of  (\ref{codimension_345a}), we have $N_{L_P(\eta)/\kappa(\eta)_P}(\overline{\theta}_P) 
= \theta_0$ and ind$(\beta \otimes E_0 \otimes L_P(\eta)) < $ ind$(\alpha)/[E_\eta : F_\eta]$.
As in the proof of (\ref{codimension_5b}), we have 
$r\beta \otimes L_P(\eta) = (E_0 \otimes L_P(\eta), \sigma_0 \otimes 1, \overline{\theta}_P)$.

If   $L_P/F_P$ is not a field, by  choice (cf. \ref{choice_at_P}(11)), we have 
$\mu_P  = (\theta_1\pi_\eta^r, \cdots , \theta_\ell\pi_\eta^{r})$.
Since $\alpha \cdot (\mu_P) = 0$ in $H^3(L_P, \mu_n^{\otimes}) = \prod H^3(F_P, \mu_n^{\otimes 2})$, 
we have $\alpha \cdot (\theta_i\pi_\eta^{r}) = 0 \in H^3(F_P, \mu_n^{\otimes 2})$.
Thus,  by (\ref{dot-zero}),  we have $r\beta \otimes \kappa(\eta)_P = (E_0, \sigma_0 \otimes 1, \overline{\theta}_i)$
for all $i$.  Since $L_P(\eta) = \prod \kappa(\eta)_P$ and $\overline{\theta}_P = (\overline{\theta}_1, 
\cdots ,  \overline{\theta}_\ell)$,  we have 
$r\beta \otimes L_P(\eta) = (E_0 \otimes L_P(\eta), \sigma_0 \otimes 1, \overline{\theta}_P)$.

As in the proof of  (\ref{codimension_345a}), we construct   $L_\eta$ and $\mu_\eta$ with the required 
properties. 
\end{proof}

\begin{lemma}
\label{type2_curvepoint}
Let $\eta$ be  a codimension one point of $X_0$ and $P$ a closed point on $\eta$.
Suppose there exist  $\theta_\eta \in F_\eta $  such that 
$\alpha \cdot (\theta_\eta) = 0 \in H^3(F_\eta, \mu_n^{\otimes 2})$.
Then there exists $\theta_P \in F_P$ such that $\alpha \cdot (\theta_P) = 0 \in 
 H^3(F_P, \mu_n^{\otimes 2})$, $\nu_\eta(\theta_P) = \nu_\eta(\theta_\eta)$ and
 $\theta_P^{-1}\theta_\eta  \in  F_{P, \eta}^{*\ell^{2t_P}}$. 
\end{lemma}

\begin{proof}  Let $\pi$ be a prime  representing $\eta$ at $P$.
Since $\eta$ is regular on $\XX$, there exists a prime $\delta$ at $P$ such that 
the maximal ideal at $P$ is generated by $\pi$ and $\delta$.
Since $F_{P, \eta}$ is a complete discrete valued field with $\pi$ as a parameter,
$\theta_\eta = w\pi^s$ for some $w\in F_\eta$ unit at $\eta$.
Since  the residue field $\kappa(\eta)_P$ of $F_{P, \eta}$ is a complete discrete valued field
with $\overline{\delta}$ as a parameter, we have $\overline{w}  = \overline{u}\overline{\delta}^r$ for some 
$u \in F_P$ unit at $P$.  Let $\theta_P = u\delta^r\pi^s$.
Then clearly $\nu_\eta(\theta_\eta) = \nu_{\eta}(\theta_P)$ and $\theta_P^{-1}\theta_\eta \in F_{P, \eta}^{\ell^{2t_P}}$.
Since $\alpha \cdot (\theta_P)$ is unramified at $P$ except possibly at $\pi$ and $\delta$
and $\alpha \cdot (\theta_P) = \alpha \cdot (\mu_P) = 0 \in H^3(F_{P, \eta}, \mu_n^{\otimes 2})$,
by (\ref{zero}), $\alpha \cdot (\theta_P) =  0 \in H^3(F_P, \mu_n^{\otimes 2})$.
\end{proof}

\section{The main theorem}
\begin{theorem}
\label{main_theorem}
 Let $K$ be a  local field with residue field $\kappa$ and   $F$ the function field
of a curve over $K$.  
Let  $D$ be a central simple algebra over $F$ of exponent $n$, 
 $\alpha$ its class in $H^2(F, \mu_n)$,
 and $\lambda \in F^*$.  If $ \alpha \cdot (\lambda)= 0$ and $n$ is coprime to char$(\kappa)$, 
 then $\lambda$ is a reduced norm from $D^*$. 
\end{theorem}

\begin{proof}   As in the proof of  (\ref{thm_dvr}), we assume that $n = \ell^d$ for   prime $\ell$ with $\ell \neq $char$(\kappa)$ 
and  $F$ contains a   primitive $\ell^{\rm th}$ root of unity.
We prove the theorem by induction on ind$(D)$. 

Suppose that ind$(D) = 1$.Then $D$ is a matrix algebra and hence every element of $F$
is a reduced norm.  Assume that  ind$(D) > 1$.

Without loss of generality we assume that $K$ is algebraically closed in $F$. 
Let $X$ be a smooth projective geometrically integral curve over $K$ with $K(X) = F$. Let $R$ be the
ring of integers in $K$ and $\kappa$ its residue field. Let $\XX$ be
a regular proper model of $F$ over  $R$
such that the union of ram$_\XX(\alpha)$,  supp$_\XX(\lambda)$ and the special fibre $X_0$ of $\XX$
is a union of regular curves with normal crossings.   
By (\ref{choice_of_model}), we assume that  $\XX$ has no special points, 
and there is no type 2 connection between a   codimension zero point of $X_0$ of type 3,  or 5  
 and   codimension zero point of $X_0$ of type 3, 4 or 5.

Let $\PP$ be the set of nodal points of $X_0$. 
For each $P \in \PP$, let $L_P$ and $\mu_P$ be as in (\ref{choice_at_P}).
Let $\eta$ be a codimension zero point of $X_0$ and $\PP_\eta = \PP \cap   \eta  $. 
Let $L_\eta$ and $\mu_\eta$ be as in \ref{codimension_1}, \ref{codimension_2_splitp},
\ref{codimension_2}, \ref{codimension_345a},  
\ref{codimension_5b} or \ref{codimension_6} depending on the type of $\eta$.
Then  $L_\eta/F_\eta$ is an extension of degree $\ell$ and $\mu_\eta \in L_\eta$
such that \\
$1)$ $N_{L_\eta/F_\eta}(\mu_\eta) = \lambda$ \\
$2)$ $\alpha \cdot (\mu_\eta) = 0 \in H^3(L_\eta, \mu_n^{\otimes 2})$ \\
$3)$ ind$(\alpha \otimes L_\eta) < $ ind$(\alpha)$ \\
$4)$   for  $P \in \PP_\eta$,   there is an isomorphism $\phi_{P,\eta} : L_\eta \otimes F_{P, \eta} \to L_P  \otimes F_{P, \eta}$ 
 and $$\phi_{P, \eta}(\mu_\eta \otimes 1) (\mu_P \otimes 1)^{-1}  \in (L_P \otimes F_{P, \eta})^{\ell^{2t_P}}.$$

Let $P \in \XX$ be a closed point with $P \not \in \PP$.
Then there is a unique codimension zero point $\eta$ of $X_0$ with $P \in  \eta $.
We give a choice of a cyclic or split extension  $L_P/F_P$ of degree $\ell$ and $\mu_P \in L_P^*$ such that \\ 
1) $N_{L_P/F_P}(\mu_P) = \lambda$, \\
 2) ind$(\alpha \otimes L_P) <  $ ind$(\alpha)$, \\
 3) $\alpha \cdot (\mu_P) = 0 \in H^3(L_P, \mu_n^{\otimes 2})$,\\
4) there is an isomorphism $\phi_{P,\eta} : L_\eta \otimes F_{P, \eta} \to L_P  \otimes F_{P, \eta}$  and 
$$\phi_{P, \eta}(\mu_\eta \otimes 1) (\mu_P \otimes 1)^{-1}  \in (L_P \otimes F_{P, \eta})^{\ell^{2t_P}}.$$
  
Suppose that $\eta$ is of type 1. Then, by the choice of $L_\eta$ and $\mu_\eta$ ( \ref{codimension_1}),
   $L_P = F_P(\sqrt[\ell]{\lambda})$ and $\mu_P  = \sqrt[\ell]{\lambda}$ have the required properties.  
 
 Suppose that $\eta$ is of type 2. Suppose that  there is a  type 2 connection  
 to a codimension zero point $\eta'$ of $X_0$ of type 5. Let   
  $Q$ be the point of type 2 intersection $\eta$ and $\eta'$. Suppose that  
   $M_{\eta'} \otimes F_{Q, \eta'}$ not a field.  
  Then, by choice (cf. \ref{codimension_2_splitp}),  we have $L_\eta = \prod F_\eta$
   and $\mu_\eta = (\theta_1, \cdots, \theta_\ell)$.
   Since $\alpha \cdot (\mu_\eta) = 0$, we have $\alpha \cdot (\theta_i) = 0$.
   For each $i$, $ 2 \leq i \leq \ell$, by (\ref{type2_curvepoint}), there exists   $\theta^P_i \in F_P$
   such that $\alpha \cdot (\theta_i^P) = 0 \in H^3(F_P, \mu_n^{\otimes 2})$
   and $\theta_i^{-1}\theta_i^P  \in F_{P, \eta}^{*\ell^{2t_P}}$.  
   Let  $\theta_1^P = \lambda (\theta_2^P \cdots \theta_\ell^P)^{-1}$.   
   Then $L_P = \prod F_P $ and $\mu_P = (\theta_1^P, \cdots  , \theta_\ell^P)$ have the required properties.
   Suppose that $M_{\eta'} \otimes F_{Q, \eta'}$ is a field    or  there is no type 2 connection from $\eta$ to any point of type 5.
Then, by the choice (\ref{codimension_2}), we have   $L_\eta = F_\eta(\sqrt[\ell]{\lambda})$ and $\mu_\eta = \sqrt[\ell]{\lambda}$.
Hence  $L_P = F_P(\sqrt[\ell]{\lambda})$ and  $\mu_P  = \sqrt[\ell]{\lambda} \in L_P$ have the required properties.

 Suppose that $\eta$ is not of type 1 or 2.
 Then, by choice $L_\eta/F_\eta$ is  an unramified field extension of degree $\ell$ or 
 the split extension of degree $\ell$. 
 Let $\hat{A}_P$ be the completion of the local ting at $P$ and $\pi$ a prime in $\hat{A}_P$
 defining $\eta$ at $P$.  Since $P \not\in \PP$ and ram$_{\XX}(\alpha)$ is union of regular curves with 
 normal crossings, there exists a prime $\delta \in \hat{A}_P$ such that 
  $\alpha$ is unramified on $\hat{A}_P$ except possibly at $\pi$ and $\delta$.
  Further,  $\lambda = w\pi^r\delta^s$ for some unit $u \in \hat{A}_P$.  Since $\eta$ is not of type 1 or 2,
 $\nu_\eta(\lambda) = r$ is divisible by $\ell$. Thus, 
 by (\ref{curve_point}),  there exist a cyclic extension $L_P/F_P$ and $\mu_P \in L_P$
 such that  \\
 1) $L_P \otimes F_{P,\eta} \simeq L_\eta \otimes F_{P, \eta}$,\\
 2) ind$(\alpha \otimes L_P) <  $ ind$(\alpha)$, \\
 3) $\alpha \cdot (\mu_P) = 0 \in H^3(L_P, \mu_n^{\otimes 2})$, \\
 4) there is an isomorphism $\phi_{P,\eta} : L_\eta \otimes F_{P, \eta} \to L_P  \otimes F_{P, \eta}$  and 
$$\phi_{P, \eta}(\mu_\eta \otimes 1) (\mu_P \otimes 1)^{-1}  \in (L_P \otimes F_{P, \eta})^{\ell^{2t_P}}.$$

Thus for every $x \in X_0$, we have chosen an  extension $L_x/F_x$ of degree 
$\ell$ and $\mu_x \in L_x$ such that \\
1) $N_{L_x / F_x }(\mu_x ) = \lambda$ \\
2) $ \alpha \cdot (\mu_x )= 0 \in H^3(L _ x , \mu_x ^{\otimes 2})$ \\
3)  ind$(\alpha \otimes L_x) < $ ind$(\alpha)$ \\
4)   for any branch  $(P , \eta)$,  
 there is an isomorphism $\phi_{P,\eta} : L_\eta \otimes F_{P, \eta} \to L_P  
 \otimes F_{P, \eta}$  and $\phi_{P, \eta}(\mu_\eta \otimes 1) (\mu_P \otimes 1)^{-1}  \in (L_P \otimes F_{P, \eta})^{\ell^{2t_P}}$.
Further if $P$ is a closed point of $X_0$, then $L_P/F_P$ is cyclic or the split  extension.

Let $(P,  \eta)$ be a branch. Since $\kappa(P)$ has no $\ell^{2t_P^{\rm th}}$ primitive root of unity
and $\kappa(\eta)_P$ is a complete discretely valued field with residue field $\kappa(P)$, 
$\kappa(\eta)_P$ has no $\ell^{2t_P^{\rm th}}$ primitive root of unity.
Since $F_{P, \eta}$ is a complete discretely valued field with residue field $\kappa(\eta)_P$, 
$F_{P,\eta}$ has no $\ell^{2t_P^{\rm th}}$ primitive root of unity.
Since $\phi_{P, \eta}(\mu_\eta \otimes 1) (\mu_P \otimes 1)^{-1}  \in (L_P \otimes F_{P, \eta})^{\ell^{2t_P}}$
and $t_P \geq d$, by (\ref{hilbert90}), for a generator $\sigma$ of Gal$(L_P \otimes F_{P, \eta}/F_{P, \eta})$,
there exists $\theta_{P,\eta} \in L_P \otimes F_{P, \eta}$ such that 
$\phi_{P, \eta}(\mu_\eta \otimes 1) (\mu_P \otimes 1)^{-1} = \theta_{P, \eta}^{-\ell^d} \sigma(\theta_{P, \eta})^{\ell^d}$.

By (\ref{patching_L_mu}), there exist  extensions $L/F$ of degree $\ell$, $N/F$ of degree coprime to $\ell$,
  and $\mu \in L \otimes N $ such that \\
$\bullet$ $N_{L\otimes N/F}(\mu) = \lambda$ and \\
$\bullet$ $  \alpha \cdot (\mu) =  0 \in H^3(L\otimes N, \mu_n^{\otimes 2})$  \\
 $\bullet$ ind$(\alpha \otimes L) <  $ ind$(\alpha)$. 
  
   Since $L\otimes N$ is also a function field of a curve over $p$-adic field,
by induction hypotheses,  $\mu$ is a reduced norm from $D \otimes L \otimes N$ and hence 
$\lambda = N_{L\otimes N/N}(\mu)$ is a reduced norm from $D$. 
Since $N_{N/F}(\lambda) = \lambda^{[N : F]}$,  $\lambda^{[N : F]}$ is a norm from $D$. 
Since $[N : F]$ is coprime to $\ell$, $\lambda$ is a reduced norm from $D$. 
 \end{proof}

 \begin{cor} 
 \label{main_cor}Let $K$ be a local field  with residue field $\kappa$ 
 and  $F$ the function field of a curve 
over $K$.  Let $\Omega$ be the set of  divisorial discrete valuations of $F$.
 Let $D$ be a central simple algebra over $F$ of index coprime to  char$(\kappa)$ and 
 $\lambda \in F$. If $\lambda$ is a reduced norm from $D \otimes F_\nu$ for all 
 $\nu \in \Omega$, then $\lambda$ is a reduced norm from $D$.
  \end{cor}
  
  \begin{proof} Since $\lambda$ is a reduced norm from $F_\nu$ for all $\nu \in \Omega_F$,
  $\alpha \cdot (\lambda) = 0$ in $H^3(F_\nu, \mu_n^{\otimes 2})$ for all $\nu \in 
  \Omega$. Thus, by (\cite[Proposition 5.2]{Ka}), $\alpha \cdot (\lambda) = 0$ in $H^3(F, \mu_n^{\otimes 2})$ and 
  by (\ref{main_theorem}), $\lambda$ is a reduced norm from $D$.  
  \end{proof}

\vskip 5mm

\noindent
Acknowledgments:   We thank Nivedita for the wonderful conversations during the preparation
of the paper. The first author is   partially supported by National Science Foundation grants DMS-1401319
and  DMS-1463882 and the third author  is partially supported by National Science Foundation grants DMS-1301785 and  DMS-1463882.

\providecommand{\bysame}{\leavevmode\hbox to3em{\hrulefill}\thinspace}

\end{document}